\documentclass[11pt]{article}
\usepackage[utf8]{inputenc}
\usepackage[T1]{fontenc}
\usepackage{amsmath,amssymb,amsthm,amsfonts}
\usepackage{mathtools}
\usepackage{mathrsfs}
\usepackage{geometry}
\usepackage{enumitem}
\usepackage{algorithm}
\usepackage{algpseudocode}
\usepackage{caption}
\usepackage{xcolor}
\usepackage{listings}
\usepackage{booktabs}
\usepackage{array}
\usepackage{ytableau}
\usepackage{hyperref}
\hypersetup{colorlinks=true, linkcolor=blue, citecolor=blue, urlcolor=blue}
\definecolor{deepteal}{RGB}{0,100,100}
\definecolor{cobaltblue}{RGB}{0,71,171}
\definecolor{firebrickred}{RGB}{178,34,34}
\lstdefinestyle{pythonstyle}{%
  language=Python,%
  basicstyle=\ttfamily\small,%
  keywordstyle=\color{blue}\bfseries,%
  commentstyle=\color{gray}\itshape,%
  stringstyle=\color{teal},%
  showstringspaces=false,%
  breaklines=true,%
  frame=single,%
  numbers=left,%
  numberstyle=\tiny\color{gray},%
  tabsize=4%
}
\theoremstyle{plain}
\newtheorem{theorem}{Theorem}[section]
\newtheorem{lemma}[theorem]{Lemma}
\newtheorem{proposition}[theorem]{Proposition}
\newtheorem{corollary}[theorem]{Corollary}
\theoremstyle{definition}
\newtheorem{definition}[theorem]{Definition}
\newtheorem{example}[theorem]{Example}
\newtheorem{remark}[theorem]{Remark}
\newtheorem{conjecture}[theorem]{Conjecture}
\theoremstyle{remark}
\newtheorem*{notation}{Notation}

\newcommand{\di}{\mathrm{di}}
\newcommand{\dinv}{\mathrm{dinv}}
\newcommand{\area}{\mathrm{area}}
\newcommand{\defc}{\mathrm{defc}}
\newcommand{\DS}{\mathsf{DS}}
\newcommand{\DSstar}{\mathsf{DS}^{*}}
\newcommand{\rowsert}{\texttt{rowsert}}
\newcommand{\worsert}{\texttt{worsert}}
\title{Dyck Symmetric Functions and Applications to \(q,t\)-Catalan Polynomials}
\author{Graham Hawkes\thanks{\texttt{grhmspm@yahoo.com}}}
\date{\today}
\begin{document}
\maketitle

\begin{abstract}
This paper develops three related combinatorial results for Dyck-type
sequences.  First, it constructs a row-insertion algorithm for dual Dyck
sequences and extends it to Dyck tableaux.  This construction gives a
weight-preserving bijection between dual Dyck factorizations and pairs
consisting of a Dyck tableau and a semistandard Young tableau of the same
shape.  As a consequence, the associated dual Dyck symmetric functions are
Schur-positive, and the corresponding affine Dyck symmetric functions have the
conjugate-shape Schur expansion.

Second, it applies these Dyck symmetric functions to the
\(q,t\)-Catalan polynomial.  It gives a two-column tableau formula for
\(C_n(q,t)\), expressing it as a sum over Dyck \(m\)-skeletons and
at-most-two-column Dyck tableaux with summands involving two-variable Schur
functions.

Third, it develops a Dyck-skeleton formula for the deficit range
\(\defc\le 2n-8\).  Full and special Dyck skeletons, together with local
\(\mathrm{East}\), \(\mathrm{West}\), \(\mathrm{up}\), and
\(\mathrm{down}\) moves, organize the low-area half of each low-deficit slice
into skeleton-indexed strings.  The \(q,t\)-symmetry of \(C_n(q,t)\) supplies
the complementary high-area half in the resulting interval formula.
\end{abstract}

\section{Introduction}
\label{sec:introduction}

Dyck-type inequalities impose local restrictions on integer sequences, while
the generating functions built from factorizations of those sequences exhibit
global symmetry and positivity properties that are not immediate from the
definitions.  This paper studies these properties in two connected parts.

The ordinary \(q,t\)-Catalan polynomials arose from the diagonal-harmonics
program of Garsia and Haiman~\cite{GarsiaHaiman1996}.  The positivity and
Dyck-path formulas of Garsia--Haglund and Haglund--Haiman--Loehr--Remmel--
Ulyanov~\cite{GarsiaHaglund2002,HHLRU2005}, together with Haglund's
monograph~\cite{Haglund2008}, made these polynomials a central meeting point
between symmetric functions and Catalan combinatorics.  The full shuffle
theorem, proved by Carlsson and Mellit~\cite{CarlssonMellit2018}, supplies in
particular the symmetry in \(q\) and \(t\) used below.  Our results do not give
a new proof of this global symmetry; instead, they use it to turn a
low-deficit string decomposition into explicit formulas.

Section~\ref{sec:insertion-algorithm} concerns Dyck symmetric functions.  We
define affine and dual Dyck factorizations of a multiset, together with the
statistic \(\di\), and use them to define symmetric functions \(\DS(S,d)\) and
\(\DSstar(S,d)\).  One main result is an explicit tableau formula for these
functions.  The central construction is a row-insertion operation \(\rowsert\)
on dual Dyck sequences,
together with a reverse operation \(\worsert\).  After establishing validity,
reversibility, and \(\di\)-preservation for these operations, the row insertion
is iterated through tableau rows.  The resulting
tableau insertion algorithm gives a weight-preserving bijection between dual
Dyck factorizations and pairs \((P,Q)\), where \(P\) is a Dyck tableau and
\(Q\) is a semistandard Young tableau of the same shape.  This bijection,
Theorem~\ref{thm:tableau-factorization-bijection}, yields the Schur expansion
for \(\DSstar(S,d)\) in
Corollary~\ref{cor:dual-dyck-schur-positivity}; the affine Schur expansion for
\(\DS(S,d)\) is Corollary~\ref{cor:affine-dyck-schur-positivity}.

Sections~\ref{sec:dyck-decompositions} and~\ref{sec:full-skeletons} concern
Dyck sequences and \(q,t\)-Catalan formulas.  We first relate several families
of Dyck data by passing from Dyck sequences to Dyck triples, then to affine and
tableau data.  This yields an application of the Dyck symmetric functions:
Theorem~\ref{thm:section4-qt-catalan-formula} expresses \(C_n(q,t)\) in terms
of Dyck \(m\)-skeletons and Dyck tableaux with at most two columns.  This
two-column formula motivates the special Dyck skeleton formula in the final
section, linking the Section~5 skeletons back to the Dyck-symmetric-function
framework.  For a
fixed length and deficit range, local \(\mathrm{East}\) and \(\mathrm{West}\)
maps control finite windows, while the global \(\mathrm{up}\) and
\(\mathrm{down}\) maps organize Dyck sequences into strings starting from
special skeletons.

The main skeleton formula is proved in
Theorem~\ref{thm:qt-catalan-skeleton} of
Subsection~\ref{subsec:strings-and-formula}, for the deficit range
\(\defc\le 2n-8\).  The \(\mathrm{up}\)-string decomposition accounts
for the low-half terms of the relevant interval contributions; the high-half
terms in the displayed symmetric rational expression are supplied by the
\(q,t\)-symmetry of \(C_n(q,t)\).  In the smaller range
\(\defc\le n-3\), the skeleton formula specializes to a partition-indexed
description.  We also record a middle-coefficient pattern and a
conjectural extension.

Section~5 focuses on the range \(\defc\le 2n-8\), which is intermediate
between fixed-deficit questions and the full deficit range for length \(n\);
the earlier sections work with general ranges and do not depend on this
restriction.  Lee, Li, and Loehr gave a combinatorial proof of the required
ordinary \(q,t\)-symmetry in total-degree slices of fixed deficit
\(0\le k\le 9\)~\cite{LeeLiLoehr2018}.  The author later extended this
chain-decomposition viewpoint to rational \(q,t\)-Catalan polynomials: in the
subfamily \(r\equiv1\pmod s\), which includes the classical case \(r=s+1\), he
showed that, for each fixed bound \(d^*\), proving the conjectural chain
formula for all \(\mathcal C_{r/s}^d\) with \(d\le d^*\) is equivalent to a
finite counting problem together with choices of chains~\cite{Hawkes2024}.  In
the classical case, the author's degree-from-the-top statistic
\(\mathrm{degr}\) is the same deficit measured here by
\(\defc=\binom n2-\area-\dinv\).  This finite-reduction result motivates
looking beyond constant-bounded deficit in chain decompositions; the
Section~5 bound \(\defc\le 2n-8\) grows with \(n\), while still remaining below
the full range.

The organization of the paper follows these two threads.  Section~2 fixes the
global definitions and conventions.  Section~3 develops the insertion
algorithm, the tableau bijection, and the Schur-positivity consequences.
Section~4 gives the Dyck-sequence and Dyck-triple decompositions leading to
the two-column tableau formula.  Section~5 develops full and special skeletons,
the local \(\mathrm{East}/\mathrm{West}\) maps, the
\(\mathrm{up}/\mathrm{down}\) decomposition, and the resulting
\(q,t\)-Catalan formulas.  The appendices collect computational
constructions, example decompositions, and the local proofs for the
well-definedness lemmas used in Section~5.

\section{Definitions and conventions}
\label{sec:definitions}
This section fixes the conventions used throughout the paper.  Affine and dual
Dyck sequences are finite integer sequences, with nonnegativity imposed only by
ordinary Dyck or interval-bounded hypotheses.
\begin{notation}[Global sequence conventions]
Unless a local statement explicitly says otherwise, a finite sequence is
indexed from position \(0\).  Thus
\(x=(x_0,x_1,\ldots,x_{m-1})\), and inequalities involving consecutive entries
are imposed for \(0\le i<m-1\).  In algorithmic arguments we may also write
\(x[i]\) for \(x_i\).  The empty sequence is allowed.  Universal elementwise
conditions on an empty sequence are vacuous, while length, shape, and parameter
conditions are still imposed as stated.
When later theorem statements use \(\area\) or \(\defc\) for a translate of an
ordinary Dyck sequence, those statistics are computed after translating the
sequence so that its first entry is \(0\), provided the resulting sequence is
ordinary Dyck.  Statements that use another transform, such as an adjoint
sequence, say so explicitly.
\end{notation}
\begin{definition}[\(\di\) statistic]
\label{def:di}
For a finite integer sequence \(x=(x_0,x_1,\ldots,x_{m-1})\), define
\[
\di(x)=\#\bigl\{(i,j):0\le i<j<m,\ x_i=x_j+1\bigr\}.
\]
In particular, \(\di(\varnothing)=0\).
\end{definition}
\begin{definition}[Affine and ordinary Dyck sequences]
\label{def:affine-dyck}
A finite integer sequence \(x\) is an \emph{affine Dyck sequence} if, writing
\(x=(x_0,x_1,\ldots,x_{m-1})\), one has
\[
x_{i+1}\le x_i+1
\]
for every \(0\le i<m-1\).
An \emph{ordinary Dyck sequence} is a nonempty affine Dyck sequence with
\(x_0=0\) and \(\min_i x_i=0\).  Equivalently, it is a nonempty affine Dyck
sequence with first entry \(0\) and all entries nonnegative.
\end{definition}
\begin{definition}[Dual Dyck sequences]
\label{def:dual-dyck}
A finite integer sequence \(x=(x_0,x_1,\ldots,x_{m-1})\) is a
\emph{dual Dyck sequence} if
\[
x_{i+1}\ge x_i+2
\]
for every \(0\le i<m-1\).
\end{definition}
\begin{definition}[Affine and dual Dyck factorizations]
\label{def:dyck-factorizations}
Let \(S\) be a finite multiset of integers.  A \emph{factorization} of \(S\)
is an ordered sequence
\[
\mathcal F=(F_0,F_1,F_2,\ldots)
\]
of finite sequences, all but finitely many of which are empty, such that the
concatenation \(F_0F_1F_2\cdots\) is a permutation of the multiset \(S\).
The factorization is an \emph{affine Dyck factorization} if each factor is an
affine Dyck sequence, and it is a \emph{dual Dyck factorization} if each factor
is a dual Dyck sequence.  The statistic of a factorization is
\[
\di(\mathcal F)=\di(F_0F_1F_2\cdots).
\]
We use zero-indexed symmetric-function variables in this paper.  Thus the
weight monomial of \(\mathcal F\) is
\[
x^{\operatorname{wt}(\mathcal F)}=\prod_{r\ge 0} x_r^{|F_r|}.
\]
\end{definition}
\begin{definition}[Dyck and dual Dyck symmetric functions]
\label{def:dyck-symmetric-functions}
Let \(S\) be a finite multiset of integers.  The
\emph{Dyck symmetric function} with parameters \((S,d)\) is
\[
\DS(S,d;\mathbf{x})
 =
\sum_{\substack{\mathcal F\text{ affine Dyck factorization of }S\\
                 \di(\mathcal F)=d}}
 x^{\operatorname{wt}(\mathcal F)}.
\]
The \emph{dual Dyck symmetric function} with parameters \((S,d)\) is
\[
\DSstar(S,d;\mathbf{x})
 =
\sum_{\substack{\mathcal F\text{ dual Dyck factorization of }S\\
                 \di(\mathcal F)=d}}
 x^{\operatorname{wt}(\mathcal F)}.
\]
Thus the affine or dual Dyck condition is imposed on the factors in the
factorization, not on the unordered multiset \(S\) itself.
\end{definition}
\begin{definition}[Dyck tableaux]
\label{def:dyck-tableau}
A \emph{Dyck tableau} with parameters \((S,d)\) is a left-aligned tableau whose
entries are exactly the multiset \(S\), with rows indexed from top to bottom:
row \(0\) is the top row, row \(1\) is the row immediately below it, and so on.
If the rows are \(T_0,T_1,\ldots,T_{\ell-1}\), then their lengths satisfy
\[
|T_0|\ge |T_1|\ge \cdots \ge |T_{\ell-1}|.
\]
Each row, read from left to right, is required to be a dual Dyck sequence.
Columns are aligned on the left and are read from bottom to top for the affine
Dyck condition.  Equivalently, whenever adjacent rows \(T_j\) and \(T_{j+1}\)
both have an entry in column \(p\), with \(T_j\) above \(T_{j+1}\), one has
\[
T_j[p]\le T_{j+1}[p]+1.
\]
The row-reading word of \(P\) is
\[
\operatorname{RR}(P)=T_{\ell-1}T_{\ell-2}\cdots T_0,
\]
namely, rows are read left to right, starting at the bottom row and moving
upward.  The tableau has parameter \(d\) when
\[
\di(\operatorname{RR}(P))=d.
\]
The shape \(\lambda(P)\) is the partition given by the row lengths
\((|T_0|,|T_1|,\ldots,|T_{\ell-1}|)\).  Empty rows may be used when a construction
requires them; their row and column conditions are vacuous, but the displayed
shape and length constraints remain in force.
\end{definition}

\section{Insertion algorithm}
\label{sec:insertion-algorithm}
The insertion algorithm begins with a row operation on dual Dyck sequences.
The row operation is the local bumping step that will later be iterated through
Dyck tableaux.  Throughout this section, all affine and dual Dyck sequences are
finite integer sequences in the sense of Section~\ref{sec:definitions}.
\begin{notation}[Slices and mutable-state conventions]
For a sequence \(x=(x_0,\ldots,x_{m-1})\), we use half-open slice notation
\[
  x[a:b]=(x_a,x_{a+1},\ldots,x_{b-1})
\]
for \(0\le a\le b\le m\).  Thus \(x[0:b]\) is a prefix and \(x[a:m]\) is a
suffix.  We also write \(x[-1]\) for the final entry of a nonempty sequence,
and \(x[-r:]\) for the suffix \(x[m-r:m]\) of length \(r\).
The following algorithms are described as mutable processes.  In a segment
replacement step, all old segments mentioned in the rule are read before any
entry is deleted or replaced.  Thus, for example, in a step that replaces a row
segment by an input segment and appends the old row segment to an output
sequence, the appended block is the pre-replacement row segment.
\end{notation}
\begin{definition}[Maximal \(+2\)-chains]
\label{def:maximal-plus-two-chains}
Let \(x=(x_0,x_1,\ldots,x_{m-1})\) be a finite integer sequence.
If \(0\le p<m\), the \emph{maximal \(+2\)-chain starting at \(p\)} is the
longest contiguous block
\[
  x[p:p+\ell]=(x_p,x_{p+1},\ldots,x_{p+\ell-1})
\]
such that
\[
  x_{p+h+1}=x_{p+h}+2
  \qquad\text{for }0\le h<\ell-1.
\]
Equivalently, \(\ell\ge 1\), the displayed equalities hold internally, and the
block cannot be extended one step to the right while preserving consecutive
\(+2\) differences.
The \emph{maximal \(+2\)-chain ending at \(p\)} is the longest contiguous block
\[
  x[p-\ell+1:p+1]=(x_{p-\ell+1},\ldots,x_{p-1},x_p)
\]
such that
\[
  x_{p-h}=x_{p-h+1}-2
  \qquad\text{for }1\le h<\ell.
\]
Equivalently, it is the longest contiguous block ending at \(p\) whose adjacent
entries differ by exactly \(+2\).
\end{definition}
\begin{definition}[The \(\rowsert\) operation]
\label{def:rowsert}
Let \(R_0\) and \(F_0\) be dual Dyck sequences, possibly empty.  The operation
\[
  \rowsert(R_0,F_0)=(E,R)
\]
is defined by the following mutable process.  Start with
\[
  E=\varnothing,\qquad R=R_0,
  \qquad F=F_0.
\]
The sequence \(R\) is the mutable row, \(F\) is the remaining input, and \(E\)
is the evicted output.  While \(F\) is nonempty, let \(a=F[0]\), and search for
the smallest row index \(i\) satisfying
\[
  a\le R[i]+1.
\]
If no such index exists, apply Case~0 below.  Otherwise use the smallest such
\(i\) and apply the first applicable case.
\begin{description}[leftmargin=2.4cm,style=nextline]
\item[Case 0.] No row index \(i\) satisfies \(a\le R[i]+1\).  Remove \(a\) from
the front of \(F\) and append \(a\) to the end of \(R\).  The sequence \(E\) is
unchanged.
\item[Case 1.] The selected index \(i\) satisfies \(a\le R[i]\).  Remove \(a\)
from the front of \(F\), replace \(R[i]\) by \(a\), and append the old value of
\(R[i]\) to the end of \(E\).
\item[Equality branch.] The remaining possibility is
\[
  a=R[i]+1.
\]
Let \(j\) be the length of the maximal \(+2\)-chain in the current row \(R\)
starting at \(i\), and let \(k\) be the length of the maximal \(+2\)-chain in
the current input \(F\) starting at \(0\).
\item[Case 2.] If \(j\le k\), set
\[
  X=F[0:j],\qquad Y=R[i:i+j]
\]
using the pre-step values.  Remove the prefix \(X\) from \(F\), replace the row
segment \(R[i:i+j]\) by \(X\), and append the old row segment \(Y\) to the end
of \(E\).
\item[Case 3.] If \(j>k\), set \(X=F[0:k]\).  Remove the prefix \(X\) from
\(F\), append \(X\) to the end of \(E\), and leave \(R\) unchanged.
\end{description}
When \(F\) becomes empty, return the pair \((E,R)\).
\end{definition}
\begin{definition}[The \(\worsert\) operation]
\label{def:worsert}
Let \(E_0\) and \(R_0\) be dual Dyck sequences.  The operation
\[
  \worsert(E_0,R_0)=(R,F)
\]
is defined by the following mutable process.  Start with
\[
  E=E_0,
  \qquad R=R_0,
  \qquad F=\varnothing.
\]
The sequence \(E\) is processed from the right, \(R\) is the mutable row, and
\(F\) is built on the left.  While \(E\) is nonempty, let \(b=E[-1]\), and
search for the largest row index \(i\) satisfying
\[
  b\ge R[i]-1.
\]
If no such index exists, apply Case~0 below.  Otherwise use the largest such
\(i\) and apply the first applicable case.
\begin{description}[leftmargin=2.4cm,style=nextline]
\item[Case 0.] No row index \(i\) satisfies \(b\ge R[i]-1\).  Remove \(b\) from
the end of \(E\) and prepend \(b\) to the beginning of \(R\).  The sequence
\(F\) is unchanged.
\item[Case 1.] The selected index \(i\) satisfies \(b\ge R[i]\).  Remove \(b\)
from the end of \(E\), replace \(R[i]\) by \(b\), and prepend the old value of
\(R[i]\) to the beginning of \(F\).
\item[Equality branch.] The remaining possibility is
\[
  b=R[i]-1.
\]
Let \(j\) be the length of the maximal \(+2\)-chain in the current row \(R\)
ending at \(i\), and let \(k\) be the length of the maximal \(+2\)-chain in the
current sequence \(E\) ending at its last position.
\item[Case 2.] If \(j\le k\), set
\[
  X=E[-j:],\qquad Y=R[i-j+1:i+1]
\]
using the pre-step values.  Remove the suffix \(X\) from \(E\), replace the row
segment \(R[i-j+1:i+1]\) by \(X\), and prepend the old row segment \(Y\) to the
beginning of \(F\).
\item[Case 3.] If \(j>k\), set \(X=E[-k:]\).  Remove the suffix \(X\) from
\(E\), prepend \(X\) to the beginning of \(F\), and leave \(R\) unchanged.
\end{description}
When \(E\) becomes empty, return the pair \((R,F)\).
\end{definition}
\begin{notation}[Comparison and alpha indices]
The comparison index of a non-Case~0 \(\rowsert\) step is the smallest row
index \(i\) selected by the inequality \(F[0]\le R[i]+1\).  The comparison index
of a non-Case~0 \(\worsert\) step is the largest row index \(i\) selected by the
inequality \(E[-1]\ge R[i]-1\).
Later tableau arguments use alpha indices to record the row positions, or
half-positions, where individual processed elements interact with the current
row.  These alpha indices are computed relative to the current row at the
moment the relevant segment is processed.
For \(\rowsert\), if a processed input block has entries
\(x_0,\ldots,x_{\ell-1}\), then its alpha indices are as follows.  In Case~1,
the single element has alpha index \(i\).  In Case~2, where the block replaces a
row segment starting at \(i\), the element \(x_h\) has alpha index \(i+h\).  In
Case~3, where the block passes through a longer row \(+2\)-chain starting at
\(i\), the element \(x_h\) has alpha index \(i+h+\frac12\); equivalently, it
passes between the adjacent row entries \(x_h-1\) and \(x_h+1\).
For \(\worsert\), write a removed terminal block from \(E\) in its original
left-to-right order as \(x_0,\ldots,x_{\ell-1}\).  In Case~1, the single
element has alpha index \(i\).  In Case~2, where the block replaces a row
segment ending at \(i\), the element \(x_h\) has alpha index \(i-\ell+1+h\).  In
Case~3, where the block passes through a longer row \(+2\)-chain ending at
\(i\), the element \(x_h\) has alpha index \(i-\ell+h+\frac12\); equivalently,
it passes between adjacent row entries \(x_h-1\) and \(x_h+1\).
\end{notation}
\subsection{Basic validity of row insertion and reverse row insertion}
\begin{lemma}[Row insertion preserves dual Dyck sequences]
\label{lem:rowsert-dual-dyck}
Let \(R_0\) and \(F_0\) be dual Dyck sequences, and let
\[
  \rowsert(R_0,F_0)=(E,R).
\]
Then both \(R\) and \(E\) are dual Dyck sequences.
\end{lemma}
\begin{proof}
We first prove that the mutable row remains dual Dyck.  Assume that the current
row \(R\) is dual Dyck before a step, and let \(a=F[0]\).
In Case~0, no row index satisfies \(a\le R[i]+1\).  If \(R\) is empty, then
appending \(a\) gives a one-element row.  If \(R\) is nonempty, its last entry
is its largest entry; writing it as \(R[m-1]\), the Case~0 condition gives
\(a>R[m-1]+1\).  Since the entries are integers, \(a\ge R[m-1]+2\), so the new
final gap is valid and all earlier gaps are unchanged.
In Case~1, \(a\le R[i]\), and only the adjacent gaps at position \(i\) can
change.  If \(i>0\), minimality of \(i\) gives \(a>R[i-1]+1\), hence
\(a\ge R[i-1]+2\).  If \(R[i+1]\) exists, then the old row inequality
\(R[i+1]\ge R[i]+2\), together with \(a\le R[i]\), gives
\(R[i+1]\ge a+2\).  Nonexistent boundary inequalities require no check.
In Case~2, \(a=R[i]+1\).  Let \(j\) be the maximal \(+2\)-chain length in the
row starting at \(i\), and let \(k\) be the maximal \(+2\)-chain length in the
input starting at \(0\).  Since Case~2 has \(j\le k\), the replacing input block
and the old row block are both \(+2\)-chains of length \(j\), and
\[
  F[h]=R[i+h]+1 \qquad (0\le h<j)
\]
for the pre-step row and input.  Thus the internal gaps of the new segment are
exactly \(2\).  If \(i>0\), then
\[
  F[0]=R[i]+1\ge R[i-1]+3,
\]
so the left boundary is valid.  If \(R[i+j]\) exists, then the old row chain
cannot continue through \(R[i+j]\).  Since the old row is dual Dyck and the
entries are integral,
\[
  R[i+j]\ge R[i+j-1]+3.
\]
Using \(F[j-1]=R[i+j-1]+1\), we get \(R[i+j]\ge F[j-1]+2\).  If a boundary does
not exist, there is nothing to check.  In Case~3, the row is unchanged.  Hence,
by induction over the processing steps, the final row \(R\) is dual Dyck.
It remains to prove that the evicted sequence \(E\) is dual Dyck.  Each
non-Case~0 step appends to \(E\) a nonempty block that is internally dual Dyck:
a singleton in Case~1, an old row \(+2\)-chain in Case~2, and a front input
\(+2\)-chain in Case~3.  Therefore it suffices to check the boundary between
consecutive appended blocks.
First note that once Case~0 occurs, every later step is also Case~0.  Indeed,
Case~0 appends the current input element \(a\) to the end of the row, and the
row remains dual Dyck by the paragraph above.  If \(b\) is the next input entry,
then the current input is a suffix of a dual Dyck sequence, so \(b\ge a+2\).
Every row entry is at most \(a\), hence no row index can satisfy
\(b\le R[i]+1\).  Thus Case~0 does not create any later block of \(E\), and all
blocks appended to \(E\) occur in consecutive steps before the first Case~0.
Consider two consecutive non-Case~0 steps \(p\) and \(p+1\).  Let \(x\) and
\(y\) be the blocks removed from the front of the input in these steps, and let
\(s\) and \(t\) be the corresponding nonempty blocks appended to \(E\).  We must
show
\[
  t[0]\ge s[-1]+2.
\]
If step \(p+1\) is Case~1 or Case~2, then the first entry appended at that step
satisfies \(t[0]\ge y[0]-1\): this is immediate in Case~1 from \(y[0]\le t[0]\),
and in Case~2 from \(y[0]=t[0]+1\).  If step \(p+1\) is Case~3, then
\(t[0]=y[0]\).
Suppose first that step \(p\) is Case~1 or Case~2, and step \(p+1\) is also
Case~1 or Case~2.  Let \(l\) be the last row index occupied by the old row block
\(s\) before step \(p\) replaces it by \(x\).  Since the input is dual Dyck,
\(y[0]\ge x[-1]+2\).  The minimality of the chosen row index at step \(p\), and
the fact that the new entries inside the replaced block are bounded above by
\(x[-1]\), imply that the next Case~1 or Case~2 index cannot lie inside or to
the left of the replaced block.  Thus the next chosen index is strictly to the
right of \(l\).  At every index strictly to the right of \(l\), the row entry is
unchanged from before step \(p\).  Since the old row was dual Dyck, those
unchanged entries are at least \(s[-1]+2\).  Hence the first entry \(t[0]\)
appended at step \(p+1\) is at least \(s[-1]+2\).
If both steps are Case~3, then \(s\) is the maximal front \(+2\)-chain of the
current input.  Therefore the next input entry \(y[0]\) is not \(s[-1]+2\), and
since the input is dual Dyck, integrality gives \(y[0]\ge s[-1]+3\).  Moreover,
step \(p\) being Case~3 means that the row contains \(s[-1]+1\), and the row is
unchanged in that step.  Since the row is dual Dyck, it cannot also contain
\(s[-1]+2\).  Step \(p+1\) being Case~3 would require the row to contain
\(y[0]-1\), so \(y[0]\ne s[-1]+3\).  Hence \(y[0]\ge s[-1]+4\), and
\(t[0]=y[0]\ge s[-1]+4\).
Next suppose that step \(p\) is Case~1 or Case~2, while step \(p+1\) is Case~3.
With \(l\) as above, the same minimality and input-growth argument shows that
the row entry equal to \(y[0]-1\) in step \(p+1\) lies strictly to the right of
\(l\).  Since the row after step \(p\) is dual Dyck, \(y[0]-1\ge s[-1]+2\), and
therefore \(t[0]=y[0]\ge s[-1]+3\).
Finally suppose that step \(p\) is Case~3, while step \(p+1\) is Case~1 or
Case~2.  The maximality of the input \(+2\)-chain appended as \(s\) gives
\(y[0]\ne s[-1]+2\), and the input being dual Dyck gives \(y[0]\ge s[-1]+2\).
Thus \(y[0]\ge s[-1]+3\).  Since \(t[0]\ge y[0]-1\) in Cases~1 and~2, we obtain
\(t[0]\ge s[-1]+2\).
All boundaries between consecutive appended blocks are valid.  Therefore \(E\)
is dual Dyck, completing the proof.
\end{proof}
\begin{definition}[Adjoint sequence]
\label{def:adjoint-sequence}
If \(x=(x_0,x_1,\ldots,x_{m-1})\), define
\[
  x^\dagger=(-x_{m-1},-x_{m-2},\ldots,-x_0).
\]
For a pair of sequences, define
\[
  (E,R)^\dagger=(R^\dagger,E^\dagger).
\]
This is the negated-reversal adjoint used in the row-insertion arguments.
\end{definition}
\begin{lemma}[Adjoints preserve dual Dyck sequences]
\label{lem:adjoint-dual-dyck}
A finite integer sequence \(x\) is dual Dyck if and only if \(x^\dagger\) is
dual Dyck.
\end{lemma}
\begin{proof}
Write \(x=(x_0,\ldots,x_{m-1})\).  The adjacent inequality for \(x^\dagger\) at
position \(h\) is
\[
  -x_{m-2-h}\ge -x_{m-1-h}+2,
\]
which is equivalent to
\[
  x_{m-1-h}\ge x_{m-2-h}+2.
\]
These are exactly the adjacent dual-Dyck inequalities for \(x\), read in reverse
order.
\end{proof}
\begin{proposition}[Reverse row insertion preserves dual Dyck sequences]
\label{prop:worsert-dual-dyck}
Let \(E_0\) and \(R_0\) be dual Dyck sequences, and let
\[
  \worsert(E_0,R_0)=(R,F)
\]
with \(\worsert\) as in Definition~\ref{def:worsert}.  Then
both \(R\) and \(F\) are dual Dyck sequences.
\end{proposition}
\begin{proof}
Run \(\worsert\) on \((E_0,R_0)\) with mutable state
\((E,R,F)\).  Under the pair adjoint convention, the state
\[
  (E,R,F)
\]
corresponds to the \(\rowsert\) state
\[
  (F^\dagger,R^\dagger,E^\dagger).
\]
The search for the largest index \(i\) satisfying \(E[-1]\ge R[i]-1\) is the
adjoint of the search for the smallest index in the row \(R^\dagger\) satisfying
the \(\rowsert\) comparison.  Case~0 for \(\worsert\), which prepends the final
element of \(E\) to \(R\), is exactly the adjoint of row-insertion Case~0, which
appends the first element of the input to the row.  Cases~1--3 likewise match
row-insertion Cases~1--3 under negation and reversal, with maximal \(+2\)-chains
starting at an index becoming maximal \(+2\)-chains ending at the reversed
index.
Thus the \(\worsert\) run on \((E_0,R_0)\) is the adjoint of a \(\rowsert\) run
on \((R_0^\dagger,E_0^\dagger)\).  By
Lemma~\ref{lem:adjoint-dual-dyck}, the adjoint inputs are dual Dyck.  By
Lemma~\ref{lem:rowsert-dual-dyck}, the two output sequences of the adjoint
\(\rowsert\) run are dual Dyck.  Applying Lemma~\ref{lem:adjoint-dual-dyck}
again gives that \(R\) and \(F\) are dual Dyck.
\end{proof}
\subsection{\texorpdfstring{Reversibility and \(\di\) preservation}{Reversibility and di preservation}}
\begin{proposition}[One-step reversibility]
\label{prop:rowsert-one-step-reversibility}
Consider an actual \(\rowsert\) run with mutable state \((E,R,F)\).  Suppose a
step is not Case~0, and let the state immediately before the step be
\[
  (E_{\mathrm{old}},R_{\mathrm{old}},F_{\mathrm{old}})
\]
and the state immediately after the step be
\[
  (E_{\mathrm{new}},R_{\mathrm{new}},F_{\mathrm{new}}).
\]
Then one step of the \(\worsert\) process applied to
\((E_{\mathrm{new}},R_{\mathrm{new}},F_{\mathrm{new}})\) removes exactly the
block appended to \(E\) in that \(\rowsert\) step and restores the old state
\((E_{\mathrm{old}},R_{\mathrm{old}},F_{\mathrm{old}})\).
\end{proposition}
\begin{proof}
There is no Case~0 to check.  We verify Cases~1--3.
In Case~1, let the processed input be \(a=F_{\mathrm{old}}[0]\), and let the
chosen row index be \(i\).  The hypotheses of Case~1 are \(a\le R_{\mathrm{old}}[i]\)
and, by minimality of \(i\), no earlier row index \(h<i\) satisfies
\(a\le R_{\mathrm{old}}[h]+1\).  The \(\rowsert\) step replaces the old row value
\(b=R_{\mathrm{old}}[i]\) by \(a\), appends \(b\) to \(E\), and removes \(a\) from
the front of \(F\).  In the new state, the last element of \(E\) is \(b\), and
the row entry at \(i\) is \(a\).  Since \(b\ge a\), the \(\worsert\)
search selects \(i\) or a later eligible index.  No later index can be eligible:
all row entries to the right of \(i\) are at least \(b+2\), so they exceed \(b+1\).
Thus the selected index is \(i\), and \(\worsert\) Case~1 replaces \(a\) by \(b\)
and prepends \(a\) to \(F\).  The previous state is recovered.
In Case~2, the processed input block \(X=F_{\mathrm{old}}[0:j]\) replaces the
old row block \(Y=R_{\mathrm{old}}[i:i+j]\), where \(X_h=Y_h+1\), and \(Y\) is
appended to \(E\).  In the new state, the terminal block of \(E\) is \(Y\), and
the corresponding row segment is \(X\).  The \(\worsert\) search from
the last element of \(Y\) selects the right end of this segment, since
\(Y_{j-1}=X_{j-1}-1\), and maximality/minimality in the forward step prevents a
later row index from satisfying the reverse search inequality.  The maximal
\(+2\)-chain in the row ending there has length exactly \(j\): it contains the
new block \(X\), and if it extended one position to the left then
\(X_0=R_{\mathrm{new}}[i-1]+2\), so
\[
  Y_0=X_0-1=R_{\mathrm{old}}[i-1]+1,
\]
contradicting the dual-Dyck inequality
\(Y_0=R_{\mathrm{old}}[i]\ge R_{\mathrm{old}}[i-1]+2\) in the old row.  The
terminal chain in \(E\) has length at least \(j\), because the appended block
\(Y\) is a \(+2\)-chain.  Thus \(\worsert\) applies Case~2, removes the terminal
block \(Y\), replaces \(X\) by \(Y\), and prepends \(X\) to \(F\), recovering the
old state.
In Case~3, the processed input block \(X=F_{\mathrm{old}}[0:k]\) passes through
a strictly longer row \(+2\)-chain and is appended to \(E\), while \(R\) is
unchanged.  In the new state, the terminal block of \(E\) is \(X\), and the row
still contains the longer \(+2\)-chain through which \(X\) passed.  By the
Case~3-to-next-block boundary estimate proved in
Lemma~\ref{lem:rowsert-dual-dyck}, if \(E_{\mathrm{old}}\) is nonempty, the
first entry of the newly appended block \(X\) is at least \(3\) larger than the
previous final entry of \(E_{\mathrm{old}}\).
Therefore the terminal \(+2\)-chain of \(E_{\mathrm{old}}X\) is exactly \(X\); it
does not extend left into \(E_{\mathrm{old}}\).  If \(X\) has length \(k\) and
the forward selected row index was \(i\), then the reverse search selects the
row entry at index \(i+k\), immediately after the \(k\) row-chain entries
through which \(X\) passed.  The maximal \(+2\)-chain ending there contains
\(R_{\mathrm{old}}[i:i+k+1]\), so its length is at least \(k+1>|X|\).
Hence \(\worsert\) applies Case~3, removes \(X\) from the end of \(E\), prepends
\(X\) to \(F\), and leaves \(R\) unchanged.  The result is exactly the old state.
\end{proof}
\begin{theorem}[Global reversibility without Case~0]
\label{thm:rowsert-global-reversibility}
Let \(R_0\) and \(F_0\) be dual Dyck sequences.  Suppose
\[
  \rowsert(R_0,F_0)=(E,R)
\]
and that no step of this \(\rowsert\) run applies Case~0.  Then
\[
  \worsert(E,R)=(R_0,F_0).
\]
\end{theorem}
\begin{proof}
Record the mutable states of the \(\rowsert\) run:
\[
  (E_0,R_0,F_0),\ (E_1,R_1,F_1),\ldots,\ (E_N,R_N,F_N),
\]
where \(E_0=\varnothing\), \(R_N=R\), \(F_N=\varnothing\), and \(E_N=E\).  Since
Case~0 never occurs, every step is covered by
Proposition~\ref{prop:rowsert-one-step-reversibility}.  Starting from
\((E_N,R_N,F_N)\), apply the \(\worsert\) steps in reverse.  The first
reverse step recovers \((E_{N-1},R_{N-1},F_{N-1})\), the next recovers
\((E_{N-2},R_{N-2},F_{N-2})\), and so on.  After \(N\) reverse steps the state
is \((\varnothing,R_0,F_0)\), so the returned pair is \((R_0,F_0)\).
\end{proof}
We next record the statistic preserved by row insertion.
\begin{lemma}[\(\di\)-preservation for row insertion]
\label{lem:rowsert-di-preservation}
Let \(R_0\) and \(F_0\) be dual Dyck sequences, and suppose
\[
  \rowsert(R_0,F_0)=(E,R).
\]
Then
\[
  \di(R_0F_0)=\di(ER).
\]
\end{lemma}
\begin{proof}
It is enough to prove that the quantity \(\di(E\cdot R\cdot F)\) is invariant
at each mutable step of \(\rowsert\).  At the beginning this concatenation is
\(R_0F_0\), and at the end it is \(ER\).
In Case~0, the first element of \(F\) is appended to the end of \(R\), so the
concatenated word \(E\cdot R\cdot F\) is unchanged.  Hence \(\di\) is
unchanged.
In Case~1, the input element \(a\) replaces a row value \(b\ge a\), while \(b\)
is appended to \(E\).  The element \(a\) moves left past row entries strictly to
the right of the selected position, all of which are at least \(b+2\), and
therefore at least \(a+2\).  The element \(b\) moves left of the earlier row
entries, all of which are at most \(b-2\).  Thus neither move changes the number
of ordered pairs whose left value is exactly one more than the right value.
In Case~2, the input \(+2\)-chain \(X\) replaces the old row \(+2\)-chain \(Y\),
with \(X_h=Y_h+1\).  The block \(X\) moves left past the row entries to the
right of \(Y\); by maximality of the old row chain, the first such right entry,
if it exists, is at least two larger than the last entry of \(X\), and all later
entries are even larger.  The block \(Y\) moves left of the row entries before
the replaced segment; those earlier entries are at least two smaller than the
first entry of \(Y\).  No cross-block \(\di\)-pair is created or destroyed, and
the internal \(\di\)-contribution of a \(+2\)-chain is zero in both blocks.
In Case~3, the input \(+2\)-chain \(X\) passes through a longer row \(+2\)-chain
and is moved from the right of \(R\) to the left of \(R\).  For each element
\(x\) of \(X\), the row contains adjacent values \(x-1\) and \(x+1\) straddling
the pass-through position.  Before the move, the pair \((x+1,x)\) contributes
one \(\di\)-pair, and after the move the pair \((x,x-1)\) contributes one
\(\di\)-pair.  These changes cancel element by element, and all other crossed
row entries differ from \(x\) by at least \(2\).  Hence the total \(\di\) is
unchanged.
Every step preserves \(\di(E\cdot R\cdot F)\), so the initial and final values
are equal.
\end{proof}

\subsection{Tableau insertion and alpha monotonicity tools}
The tableau-insertion arguments below use the paper row order fixed in
Definition~\ref{def:dyck-tableau}.  Thus a block of rows
\[
  T_0,T_1,\ldots,T_s
\]
is listed from top to bottom, and the adjacent column condition is
\[
  T_j[p]\le T_{j+1}[p]+1
\]
whenever both displayed entries exist.  The row-reading word remains the
bottom-to-top word specified in Definition~\ref{def:dyck-tableau}.
\begin{definition}[The \texttt{tabsert} operation]
\label{def:tabsert}
Let \(T\) be a Dyck tableau and let \(F\) be a dual Dyck sequence.  If \(T\) has
no rows, then \(\texttt{tabsert}(T,F)\) is the empty tableau when
\(F=\varnothing\), and is the one-row tableau with row \(F\) when
\(F\ne\varnothing\).  Otherwise, list the rows of \(T\) for this definition in
paper top-to-bottom order as
\[
  T_0,T_1,\ldots,T_s.
\]
The tableau insertion
\[
  \texttt{tabsert}(T,F)
\]
is defined by iterating \(\rowsert\) through these rows from top to bottom.
Set \(E_0=F\).  For \(j=0,1,\ldots,s\), as long as \(E_j\) is nonempty, apply
\[
  (E_{j+1},T_j')=\rowsert(T_j,E_j)
\]
and replace row \(T_j\) by \(T_j'\).  If some \(E_{j+1}\) is empty, the
operation terminates and all lower rows not yet reached are left unchanged.  If
\(E_{s+1}\) is still nonempty after the bottom row has been processed, append
\(E_{s+1}\) as a new bottom row and terminate.  If \(F\) is empty, the tableau
is unchanged.
The lemmas below prove the tableau-validity assertions needed to use this
row-by-row procedure in the insertion bijection.
\end{definition}
\begin{example}[A four-factor tableau insertion]
\label{ex:four-factor-tableau-insertion}
Consider the dual Dyck factorization
\[
  F_0\mid F_1\mid F_2\mid F_3
  =
  (0,2,4)\mid(1,3)\mid(1,3,5)\mid(0,6).
\]
Inserting the factors in order produces the trace below.  The tableau
\(P^{(i)}\) is the insertion tableau
after \(F_i\) has been inserted, and \(Q^{(i)}\) is the recording tableau of the
same shape, with each new cell labeled by the factor that created it.
\[
\begin{array}{c@{\quad}c@{\quad}c@{\qquad}c}
\text{insertion} & & P^{(i)} & Q^{(i)} \\[0.4em]
\varnothing\ \xleftarrow{\ F_0=(0,2,4)\ }
&
&
\begin{ytableau}
0 & 2 & 4
\end{ytableau}
&
\begin{ytableau}
0 & 0 & 0
\end{ytableau}
\\[2.0em]
P^{(0)}\ \xleftarrow{\ F_1=(1,3)\ }
&
&
\begin{ytableau}
0 & 2 & 4 \\
1 & 3
\end{ytableau}
&
\begin{ytableau}
0 & 0 & 0 \\
1 & 1
\end{ytableau}
\\[2.0em]
P^{(1)}\ \xleftarrow{\ F_2=(1,3,5)\ }
&
&
\begin{ytableau}
1 & 3 & 5 \\
0 & 2 & 4 \\
1 & 3
\end{ytableau}
&
\begin{ytableau}
0 & 0 & 0 \\
1 & 1 & 2 \\
2 & 2
\end{ytableau}
\\[3.0em]
P^{(2)}\ \xleftarrow{\ F_3=(0,6)\ }
&
&
\begin{ytableau}
0 & 3 & 6 \\
0 & 2 & 5 \\
1 & 4 \\
1 & 3
\end{ytableau}
&
\begin{ytableau}
0 & 0 & 0 \\
1 & 1 & 2 \\
2 & 2 \\
3 & 3
\end{ytableau}
\end{array}
\]
The row-level cases in this example show why the factorization, rather than only the
underlying word, matters.  When \(F_1=(1,3)\) is inserted through the row
\((0,2,4)\), the row contains a longer \(+2\)-chain than the input, so
\rowsert\ uses Case~3 and passes \((1,3)\) through while leaving the row
\((0,2,4)\) unchanged.  When \(F_2=(1,3,5)\) is inserted, the same prefix
\((1,3)\) is connected to the following entry \(5\); now the input chain and the
row chain have the same length, so the insertion into \((0,2,4)\) uses Case~2
and displaces the whole row.  Continuing the insertion of \(F_2\), the displaced
row \((0,2,4)\) passes through \((1,3)\) by two Case~1 replacements followed by
a terminal Case~0 append, creating the third row.
The final insertion, \(F_3=(0,6)\), produces a longer cascade: through the top
row it uses Cases~1 and~2, through the next row it uses Cases~3 and~2, and
through the third row it uses Cases~1 and~2 before the final evicted sequence is
appended as the new bottom row.  The last cascade illustrates the role of the
adjacent-row preservation lemmas.  Before inserting \(F_3\), the top two rows are
\[
  a=(1,3,5)
  \qquad\text{over}\qquad
  b=(0,2,4).
\]
After the insertion they are
\[
  a'=(0,3,6)
  \qquad\text{over}\qquad
  b'=(0,2,5).
\]
If the top row changed to \(a'\) while the old second row \(b\) remained in
place, the third column would violate the column condition, since
\(6>4+1\).  The subsequent update of the second row repairs this: with
\(b'\), the third-column inequality is \(6\le 5+1\).
Thus the example exhibits all four \rowsert\ cases and shows that the insertion
tableau depends on the chosen factor boundaries.
\end{example}
The proofs of the next three tableau lemmas use the alpha indices introduced in
the row-insertion definitions.  The two monotonicity tools below make precise
the claim used below: when Case~0 is absent, successive input elements
interact with successive row positions in order.
\begin{lemma}[Row-insertion alpha monotonicity]
\label{lem:rowsert-alpha-monotonicity}
Let \(R\) and \(x=(x_0,x_1,\ldots,x_{m-1})\) be dual Dyck sequences.  Suppose
that \(\rowsert(R,x)\) runs without ever applying Case~0.  For each input
element \(x_h\), let \(\alpha_h\) be its alpha index, computed relative to the
current row at the moment the segment containing \(x_h\) is processed.  Then,
for \(0\le h<m-1\),
\[
  \alpha_{h+1}\ge \alpha_h+1.
\]
\end{lemma}
\begin{proof}
Because Case~0 never occurs, each input element belongs to a Case~1 singleton,
a Case~2 replacement segment, or a Case~3 pass-through segment.
First suppose that \(x_h\) and \(x_{h+1}\) lie in the same processed segment.  A
Case~1 segment has only one element.  In Case~2, consecutive input elements
replace consecutive row positions, so their alpha indices differ by exactly
\(1\).  In Case~3, consecutive pass-through elements pass through consecutive
half-indices: if \(x_h\) passes between row values \(x_h-1\) and \(x_h+1\), then
\(x_{h+1}=x_h+2\) passes between \(x_{h+1}-1=x_h+1\) and
\(x_{h+1}+1=x_h+3\).  Their alpha indices again differ by exactly \(1\).
It remains to consider a boundary between two consecutive processed segments.
Suppose first that the segment ending with \(x_h\) is Case~1 or Case~2.  After
that step, the current row contains \(x_h\) at position \(\alpha_h\).  Since
\(x\) is dual Dyck,
\[
  x_{h+1}\ge x_h+2.
\]
Thus \(x_{h+1}\) cannot replace the same row entry, and it cannot pass through
the half-position immediately to the right of that entry: replacement there
would require \(x_{h+1}\le x_h+1\), while pass-through there would require the
left row entry to equal \(x_{h+1}-1\).  All row positions to the left are even
smaller.  Therefore the next alpha index is at least one full position to the
right:
\[
  \alpha_{h+1}\ge \alpha_h+1.
\]
Finally suppose that the segment ending with \(x_h\) is a Case~3 pass-through
segment.  Write \(\alpha_h=j+\frac12\), so at the relevant moment the row has
adjacent entries
\[
  R[j]=x_h-1,
  \qquad
  R[j+1]=x_h+1.
\]
Because \(x_h\) is the final element of its maximal front \(+2\)-chain in the
input, the next input element is not \(x_h+2\).  Dual Dyckness and integrality
give
\[
  x_{h+1}\ge x_h+3=R[j+1]+2.
\]
Therefore \(x_{h+1}\) cannot replace row position \(j+1\), and it cannot pass
between positions \(j+1\) and \(j+2\).  Its alpha index is at least \(j+2\), so
\[
  \alpha_{h+1}\ge j+2\ge j+\frac12+1=\alpha_h+1.
\]
All cases give the asserted inequality.
\end{proof}
\begin{lemma}[Reverse row-insertion alpha monotonicity]
\label{lem:worsert-alpha-monotonicity}
Let \(E=(e_0,e_1,\ldots,e_{m-1})\) and \(R\) be dual Dyck sequences.  Suppose
that \(\worsert(E,R)\), as defined in Definition~\ref{def:worsert}, runs without
ever applying Case~0.  For each written-order input element \(e_h\), let
\(\alpha_h\) be its alpha index for the
\(\worsert\) run, computed relative to the current row at the moment the
terminal segment containing \(e_h\) is processed.  Then, for \(0\le h<m-1\),
\[
  \alpha_{h+1}\ge \alpha_h+1.
\]
Equivalently, since \(\worsert\) processes \(E\) from right to left, the actual
processing positions move strictly left.  In particular, any fixed row position
is replaced at most once in such a no-Case~0 \(\worsert\) run.
\end{lemma}
\begin{proof}
Use the adjoint correspondence between \(\worsert\) and \(\rowsert\) from
Proposition~\ref{prop:worsert-dual-dyck}.  Since no Case~0 occurs, the row
length is constant throughout the run; write this length as \(n\).  The
\(\worsert\) run on \((E,R)\) is adjoint to the no-Case~0 \(\rowsert\) run on
\((R^\dagger,E^\dagger)\).  If \(\beta_{m-1-h}\) is the \(\rowsert\) alpha index
of the adjoint input element \(-e_h\), then reversal of row positions gives
\[
  \beta_{m-1-h}=n-1-\alpha_h,
\]
for both integer replacement positions and half-integer pass-through positions.
Apply Lemma~\ref{lem:rowsert-alpha-monotonicity} to the adjoint \(\rowsert\)
run.  For \(0\le h<m-1\), the consecutive adjoint input elements corresponding
to \(e_{h+1}\) and \(e_h\) give
\[
  \beta_{m-1-h}\ge \beta_{m-2-h}+1.
\]
Substituting \(\beta_{m-1-h}=n-1-\alpha_h\) and
\(\beta_{m-2-h}=n-1-\alpha_{h+1}\), this becomes
\[
  n-1-\alpha_h\ge n-1-\alpha_{h+1}+1,
\]
which is equivalent to \(\alpha_{h+1}\ge \alpha_h+1\).  The final processing
order statement is the same inequality read in the right-to-left processing
order of \(\worsert\).
\end{proof}
Lemma~\ref{lem:rowsert-alpha-monotonicity} and
Lemma~\ref{lem:worsert-alpha-monotonicity} are the replacement positional tools
for the tableau-insertion cluster.  Lemma~\ref{lem:two-row-tableau-insertion}
uses them to rule out a maximal bad column after two forward row insertions.
Lemma~\ref{lem:reverse-truncated-row-worsert} uses the same reverse monotonicity
to keep the truncated-row \(\worsert\) pass away from Case~0 and to control the
extracted sequence.  Lemma~\ref{lem:second-reverse-insertion-valid-two-row} then
uses the forward monotonicity in the reconstruction argument that restores a
valid two-row tableau.
For the next lemma, a valid two-row window in the paper top-to-bottom convention
means that \(T_0\) and \(T_1\) are dual Dyck rows, that
\[
  |T_0|\ge |T_1|,
\]
and that the adjacent column inequalities
\[
  T_0[p]\le T_1[p]+1
\]
hold for every \(0\le p<|T_1|\).
\begin{lemma}[Two-row tableau insertion]
\label{lem:two-row-tableau-insertion}
Let \((T_0,T_1)\) be a valid two-row window in the paper top-to-bottom
convention, and let \(F\) be a dual Dyck sequence.  Apply row insertion first to
\(T_0\) and then to \(T_1\):
\[
  (G,T_0')=\rowsert(T_0,F),
  \qquad
  (E,T_1')=\rowsert(T_1,G).
\]
Assume that Case~0 never occurs in the first row insertion \(\rowsert(T_0,F)\).
Then \((T_0',T_1')\) is again a valid two-row window in the same paper
convention.  Equivalently, \(T_0'\) and \(T_1'\) are dual Dyck sequences,
\[
  |T_0'|\ge |T_1'|,
\]
and
\[
  T_0'[p]\le T_1'[p]+1
\]
for every \(0\le p<|T_1'|\).
\end{lemma}
\begin{proof}
Lemma~\ref{lem:rowsert-dual-dyck} applied to the two row insertions shows that
\(T_0'\), \(G\), \(T_1'\), and \(E\) are dual Dyck sequences.  It remains to
prove the length and column conditions.
Assume, for contradiction, that the conclusion fails.  Choose \(j\) maximal
among indices with \(0\le j<|T_1'|\) such that either \(j\ge |T_0'|\), or
\(T_0'[j]\) exists and
\[
  T_0'[j]\ge T_1'[j]+2.
\]
We call such a \(j\) a bad index.
We first record the reverse reconstruction used below.  During the second row
insertion, a terminal part of the input \(G\) may have been appended to \(T_1\)
by Case~0.  Write
\[
  T_1'=T_1^{\prime -}\cdot T_1^{\prime +}
\]
where \(|T_1^{\prime -}|=|T_1|\) and \(T_1^{\prime +}\) is precisely that
appended suffix, possibly empty.  Reversing the non-Case~0 part of the second
row insertion gives
\[
  \worsert(E,T_1^{\prime -})=(T_1,E^*).
\]
Thus the input to the second row insertion was
\[
  G=E^*\cdot T_1^{\prime +}.
\]
Since the first row insertion had no Case~0, the row-level reversibility theorem
gives
\[
  \worsert(G,T_0')=(T_0,F),
\]
and this reverse insertion also has no Case~0.
Suppose for the moment that \(T_0'[j]\) exists, and set \(v=T_0'[j]\).  Since
\(j\) is bad, \(T_1'[j]\le v-2\).  Because \(T_1'\) is dual Dyck, any occurrence
of either \(v-1\) or \(v\) in \(T_1'\) would have to occur at a position
\(k>j\).  If \(k\ge |T_0'|\), then \(k\) is a bad index, contradicting the
maximality of \(j\).  If \(k<|T_0'|\), then the dual Dyck property of \(T_0'\)
gives
\[
  T_0'[k]\ge T_0'[j]+2=v+2.
\]
Thus \(T_0'[k]\ge T_1'[k]+2\) when \(T_1'[k]=v\), and even more strongly when
\(T_1'[k]=v-1\).  Again \(k\) would be bad, contrary to maximality.  Therefore
\(T_1'\) contains neither \(v\) nor \(v-1\).
We next show that the input \(G\) to the first reverse insertion contains no
copy of \(v-1\).  Consider how a hypothetical element \(v-1\) of \(G\) behaves
when \(G\) is row-inserted into \(T_1\).  If it is appended by Case~0 or placed
in the row by Case~1 or Case~2, then Lemma~\ref{lem:rowsert-alpha-monotonicity}
on the non-Case~0 part, together with the fact that Case~0 only appends the
remaining suffix, shows that this copy remains in the final row \(T_1'\).  This
is impossible.  If instead it passes through by Case~3, then it passes between
row values \(v-2\) and \(v\).  Lemma~\ref{lem:rowsert-alpha-monotonicity} again
implies that the row position containing the value \(v\) cannot be replaced by
a later input element, and so \(v\) would remain in \(T_1'\), also impossible.
Hence
\[
  v-1\notin G.
\]
It follows that position \(j\) of \(T_0'\) cannot decrease during
\(\worsert(G,T_0')\).  In a no-Case~0 \(\worsert\) run, Case~1 replaces a row
entry by an equal or larger value, Case~3 leaves the row unchanged, and the only
way to replace the entry \(v\) by \(v-1\) is Case~2 with input value \(v-1\),
which is absent from \(G\).  Lemma~\ref{lem:worsert-alpha-monotonicity} also
ensures that a fixed row position is replaced at most once.
We now split into two cases.
First suppose that either \(j\ge |T_1|\), or \(j<|T_1|\) and
\(T_1'[j]\ne T_1[j]\).  Then the entry \(T_1'[j]\) is one of the entries of the
input \(G\) to the second row insertion, and hence it is processed during the
reverse insertion \(\worsert(G,T_0')\).  At the moment when this entry is
processed through the current top row, position \(j\) is either absent, or it
exists and its value is still at least \(T_0'[j]\).  In the latter case the
badness of \(j\) gives
\[
  T_0'[j]\ge T_1'[j]+2.
\]
Thus the alpha index of \(T_1'[j]\) in \(\worsert(G,T_0')\) is at most
\(j-1\): position \(j\) and all positions to its right are too large for
replacement, and the half-position immediately to the left of position \(j\)
would require the row entry at position \(j\) to be \(T_1'[j]+1\).
Let \(i<j\) be maximal such that \(T_1'[i]=T_1[i]\).  If no such \(i\) exists,
put \(i=-1\).  By the choice of \(i\), all entries
\[
  T_1'[i+1],T_1'[i+2],\ldots,T_1'[j]
\]
belong to \(G\), and they occur in \(G\) in this left-to-right order.  The order
claim follows from Lemma~\ref{lem:rowsert-alpha-monotonicity} during the
non-Case~0 part of the second row insertion, while a Case~0 suffix is appended
without changing the remaining input order.
Applying Lemma~\ref{lem:worsert-alpha-monotonicity} to
\(\worsert(G,T_0')\), the leftmost entry \(T_1'[i+1]\) has alpha index at most
\[
  (j-1)-(j-i-1)=i.
\]
If \(i=-1\), this is impossible, since alpha indices in a no-Case~0 reverse
insertion are nonnegative.  Thus \(i\ge0\).
Let \(u=T_1'[i+1]\).  The alpha bound just obtained implies that, during
\(\worsert(G,T_0')\), the element \(u\) either replaces a row entry at an integer
position \(p\le i\), or passes through at a half-position \(p+\frac12\le i\).  In
the replacement case, the row then contains \(u\) at position \(p\); in the
pass-through case, the row contains \(u+1\) at position \(p+1\le i\).  Later
reverse-insertion steps occur strictly to the left by
Lemma~\ref{lem:worsert-alpha-monotonicity}.  Since the final row is dual Dyck,
in either case
\[
  T_0[i]\ge u=T_1'[i+1].
\]
But \(T_1'\) is dual Dyck and \(T_1'[i]=T_1[i]\), so
\[
  T_1'[i+1]\ge T_1'[i]+2=T_1[i]+2.
\]
Consequently \(T_0[i]\ge T_1[i]+2\), contradicting the original column
condition for the valid two-row window \((T_0,T_1)\).
It remains to treat the case \(j<|T_1|\) and \(T_1'[j]=T_1[j]\).  If
\(j\ge |T_0'|\), then the no-Case~0 reverse insertion \(\worsert(G,T_0')\) does
not change the length of the top row, so \(j\ge |T_0|\).  But
\(T_1'[j]=T_1[j]\) implies \(j<|T_1|\), contradicting \(|T_0|\ge |T_1|\).
Therefore \(T_0'[j]\) exists.  With \(v=T_0'[j]\), the exclusion above gives
\(v-1\notin G\), so position \(j\) cannot decrease during
\(\worsert(G,T_0')\).  Hence
\[
  T_0[j]\ge T_0'[j]\ge T_1'[j]+2=T_1[j]+2,
\]
again contradicting the original column condition.
Both cases contradict the validity of \((T_0,T_1)\).  Hence no bad index exists,
and the desired length and column inequalities hold.  Together with row
validity from Lemma~\ref{lem:rowsert-dual-dyck}, this proves the lemma.
\end{proof}

\begin{lemma}[Reverse pass through a truncated row]
\label{lem:reverse-truncated-row-worsert}
Let \((T_0,T_1,T_2)\) be a valid three-row window in top-to-bottom row order.
Thus the rows are dual Dyck sequences,
\[
  |T_0|\ge |T_1|\ge |T_2|,
\]
and the adjacent column inequalities
\[
  T_0[p]\le T_1[p]+1,
  \qquad
  T_1[p]\le T_2[p]+1
\]
hold whenever the displayed entries exist.  Let \(F^+\) be the terminal \(k\)
entries of \(T_1\), where
\[
  0\le k\le |T_1|-|T_2|,
\]
and write
\[
  T_1=T_1^-\cdot F^+.
\]
Set
\[
  (T_1',F^-)=\worsert(T_2,T_1^-).
\]
Then Case~0 of this \(\worsert\) run is never triggered.  Moreover
\(F^-\cdot F^+\) is a dual Dyck sequence, has length at most \(|T_1|\), and
\((T_0,F^-\cdot F^+)\) is a valid two-row window in top-to-bottom row order.
\end{lemma}

\begin{proof}
Write
\[
  n=|T_1^-|,
  \qquad
  m=|T_2|.
\]
The hypothesis on \(k\) gives \(n=|T_1|-k\ge m\), so every index of
\(T_2\) is also an index of \(T_1^-\).

The operation \(\worsert(T_2,T_1^-)\) processes the first input from right to
left.  Each step removes a terminal chunk of the current first input; because
the initial first input is \(T_2\), we may write the chunk processed at a given
step as
\[
  T_2[i:j]
\]
for some \(0\le i<j\le m\).  The chunks are processed in decreasing order of
these intervals.

We prove the following invariant, moving from right to left through the chunks.
Immediately before the step processing \(T_2[i:j]\), the current row agrees
with \(T_1^-\) in positions \(0,1,\ldots,j-1\).  If \(Y\) is the block prepended
to the output sequence \(F\) by this step, then after the step
\[
  Y[q-i]\ge T_1^-[q]
  \qquad (i\le q<j),
\]
and the current row agrees with \(T_1^-\) in positions \(0,1,\ldots,i-1\).  For
the first, rightmost, chunk the row is still exactly \(T_1^-\), so the initial
row-agreement clause holds.

Consider a step processing \(T_2[i:j]\), and assume the pre-step clause of the
invariant: the current row agrees with \(T_1^-\) in positions
\(0,1,\ldots,j-1\).  The last element under consideration is \(T_2[j-1]\).
Since
\((T_1,T_2)\) is a valid adjacent pair and \(j-1<n\),
\[
  T_2[j-1]\ge T_1[j-1]-1=T_1^-[j-1]-1.
\]
The current row still has entry \(T_1^-[j-1]\) at position \(j-1\), so an
eligible row index exists for the \(\worsert\) comparison.  Hence Case~0 cannot
occur at this step.

Let \(p\) be the largest eligible row index selected for \(T_2[j-1]\).  The
previous paragraph gives \(p\ge j-1\).  We next show that each element
\(T_2[q]\), \(i\le q<j\), has alpha index at least \(q\).  In Case~1 the chunk
has length one, and the sole alpha index is \(p\ge j-1=i\).  In Case~2 the chunk
replaces a row segment ending at \(p\), so its alpha indices are
\[
  p-(j-i)+1,
  p-(j-i)+2,
  \ldots,
  p.
\]
Since \(p\ge j-1\), these are at least \(i,i+1,\ldots,j-1\), respectively.

It remains to exclude a possible Case~3 exception.  In Case~3, the chunk
\(T_2[i:j]\) passes through a longer row \(+2\)-chain ending at \(p\), and its
alpha indices are consecutive half-indices ending at \(p-\frac12\).  If
\(p\ge j\), these alpha indices are all at least the corresponding original
indices.  Thus the only possible failure is the subcase \(p=j-1\).

Assume this subcase occurs.  Since the row chain ending at
\(j-1\) is strictly longer than \(T_2[i:j]\), it extends one position farther to
the left, so \(i>0\).  Because this is Case~3 and the current row agrees with
\(T_1^-\) through position \(j-1\),
\[
  T_2[q]=T_1^-[q]-1
  \qquad (i\le q<j),
\]
and the extra row-chain entry gives
\[
  T_1^-[i-1]=T_2[i]-1.
\]
The chunk \(T_2[i:j]\) is the full terminal maximal \(+2\)-chain in the
remaining first input.  Hence \(T_2[i]\ne T_2[i-1]+2\).  Since \(T_2\) is dual
Dyck and has integer entries, this forces
\[
  T_2[i-1]\le T_2[i]-3=T_1^-[i-1]-2.
\]
But \(T_1^-[i-1]=T_1[i-1]\), and the original column condition for
\((T_1,T_2)\) gives
\[
  T_2[i-1]\ge T_1[i-1]-1=T_1^-[i-1]-1,
\]
a contradiction.  Therefore the exceptional Case~3 subcase cannot occur, and
the claimed alpha-index lower bounds hold in every actual step.

We now prove the remaining clauses of the invariant for this step.  In
Cases~1 and~2, the output block \(Y\) consists
of old row entries at the integer alpha indices identified above.  If the element
corresponding to \(T_2[q]\) has alpha index \(a\), then \(a\ge q\).  The current
row is dual Dyck and agrees with \(T_1^-\) through position \(j-1\), so the old
row entry at \(a\) is at least \(T_1^-[q]\).  Thus the corresponding element of
\(Y\) is at least \(T_1^-[q]\).

In Case~3, the output block is the pass-through chunk itself.  If \(T_2[q]\)
passes through at half-index \(h+\frac12\), then the current row contains
adjacent entries
\[
  R[h]=T_2[q]-1,
  \qquad
  R[h+1]=T_2[q]+1.
\]
The alpha-index lower bound gives \(h+\frac12\ge q\), hence \(h+1\ge q+1\).
Using the dual Dyck property of the current row, together with the agreement
with \(T_1^-\) through position \(j-1\), we get
\[
  R[h+1]\ge T_1^-[q]+2.
\]
Therefore \(T_2[q]=R[h+1]-1\ge T_1^-[q]\).  This proves the elementwise lower
bound for \(Y\).  Also, Cases~1 and~2 change only positions with integer alpha
index at least \(i\), while Case~3 changes no row entries.  Thus positions
strictly before \(i\) remain equal to \(T_1^-\).  The invariant follows.

After all chunks have been processed, the output \(F^-\) is the concatenation of
the chunk outputs in the original left-to-right order.  The invariant gives
\[
  F^-[q]\ge T_1^-[q]=T_1[q]
  \qquad (0\le q<m).
\]
Since Case~0 never occurs, each element of \(T_2\) contributes exactly one
output element, so \(|F^-|=m\).

Let
\[
  H=F^-\cdot F^+.
\]
For the first \(m\) positions, the displayed inequality gives \(H[q]\ge T_1[q]\).
For an entry of \(F^+\), write it as \(T_1[n+s]\) with \(0\le s<k\).  In \(H\)
it lies at position \(m+s\).  Since \(m\le n\) and \(T_1\) is dual Dyck, it is
weakly increasing with the index, and therefore
\[
  H[m+s]=T_1[n+s]\ge T_1[m+s].
\]
Thus
\[
  H[q]\ge T_1[q]
\]
for every index \(q\) of \(H\).  The length bound is
\[
  |H|=|F^-|+|F^+|=m+k\le |T_1|.
\]

It remains to prove that \(H\) is dual Dyck.  By
Proposition~\ref{prop:worsert-dual-dyck}, \(F^-\) is dual Dyck, and \(F^+\) is
dual Dyck because it is a suffix of the dual Dyck row \(T_1\).  Only the boundary
between \(F^-\) and \(F^+\) needs checking, and only when both factors are
nonempty.  The last element of \(F^-\) comes from the first, rightmost,
\(\worsert\) step.  At that time the row is exactly \(T_1^-\).  If that step is
Case~1 or Case~2, this last output element is an old row entry of \(T_1^-\); if
it is Case~3, it is one less than an old row entry of \(T_1^-\).  In either
case,
\[
  F^-[-1]\le T_1^-[n-1].
\]
Since \(F^+\ne\varnothing\), its first entry is \(F^+[0]=T_1[n]\), and the dual
Dyck property of \(T_1\) gives
\[
  F^+[0]=T_1[n]\ge T_1[n-1]+2=T_1^-[n-1]+2\ge F^-[-1]+2.
\]
So the boundary also satisfies the dual Dyck gap condition, and \(H\) is dual
Dyck.

Finally, \((T_0,H)\) is a valid two-row window.  The top row \(T_0\) is dual
Dyck, and \(H\) is dual Dyck by the preceding paragraph.  The length condition
is
\[
  |T_0|\ge |T_1|\ge |H|.
\]
For every index \(q\) of \(H\), the original column condition for \((T_0,T_1)\)
and the inequality \(H[q]\ge T_1[q]\) give
\[
  T_0[q]\le T_1[q]+1\le H[q]+1.
\]
This is exactly the required top-to-bottom column condition.  The proof is
complete.
\end{proof}

\begin{lemma}[Second reverse insertion restores a two-row window]
\label{lem:second-reverse-insertion-valid-two-row}
Use the setup of Lemma~\ref{lem:reverse-truncated-row-worsert}.  Thus
\((T_0,T_1,T_2)\) is a valid three-row window in top-to-bottom row order,
\[
  T_1=T_1^-\cdot F^+,
  \qquad
  (T_1',F^-)=\worsert(T_2,T_1^-),
\]
where \(0\le k=|F^+|\le |T_1|-|T_2|\).  Put
\[
  H=F^-\cdot F^+,
\]
and apply the second reverse insertion
\[
  (T_0',F')=\worsert(H,T_0).
\]
Then \((T_0',T_1')\) is a valid two-row window in top-to-bottom row order.
\end{lemma}

\begin{proof}
By Lemma~\ref{lem:reverse-truncated-row-worsert}, the sequence \(H\) is dual
Dyck, \(|H|\le |T_1|\), and \((T_0,H)\) is a valid two-row window.  Proposition
\ref{prop:worsert-dual-dyck} gives that \(T_0'\) is dual Dyck, and the same
proposition applied in Lemma~\ref{lem:reverse-truncated-row-worsert} gives that
\(T_1'\) is dual Dyck.

The length condition is immediate from the reverse-insertion construction.  The
mutable row in \(\worsert\) is never shortened: Cases~1 and~2 replace a row
segment by a segment of the same length, Case~3 leaves the row unchanged, and
Case~0 prepends an entry to the row.  Hence
\[
  |T_0'|\ge |T_0|.
\]
The first reverse pass through \(T_1^-\) has no Case~0 by
Lemma~\ref{lem:reverse-truncated-row-worsert}, so
\[
  |T_1'|=|T_1^-|\le |T_1|.
\]
Since the original three-row window has \(|T_0|\ge |T_1|\), we obtain
\[
  |T_0'|\ge |T_1'|.
\]
It remains only to prove the column inequalities.

Write \(x\gg y\) to mean \(x\ge y+2\).  Suppose, for contradiction, that the
column condition fails, and let \(i\) be the minimal index such that
\[
  T_0'[i]\gg T_1'[i].
\]
The stepwise reversibility of the two reverse insertions says that applying
\(\rowsert\) first to \((T_0',F')\) and then to \((T_1',H)\) recovers the
original adjacent rows:
\[
  \rowsert(T_0',F')=(H,T_0),
  \qquad
  \rowsert(T_1',H)=(T_2,T_1).
\]
The row-insertion monotonicity lemma is used below in the following
consequence: during the non-Case~0 part of such a row insertion, alpha indices
move strictly to the right, so a fixed old row position is affected at most
once; in particular an existing entry of the row can increase by at most one.
The terminal Case~0 step that appends the suffix \(F^+\), when it occurs in the
second displayed row insertion, does not change any earlier row position.

Let
\[
  v=T_0'[i].
\]
If \(i>0\), minimality of \(i\), dual Dyckness of \(T_1'\), and the choice of
\(i\) give
\[
  T_0'[i-1]\le T_1'[i-1]+1\le T_1'[i]-1\le T_0'[i]-3=v-3.
\]
If \(i=0\), there is no entry to the left of \(v\).  Thus the value \(v-1\) can
neither be evicted from \(T_0'\) nor pass through \(T_0'\) in the row insertion
\(\rowsert(T_0',F')\): eviction would require an old row entry equal to
\(v-1\), and pass-through would require adjacent row entries \(v-2\) and \(v\).
Therefore
\[
  v-1\notin H.
\]

Now compare the \(i\)-th entry of \(T_1'\) with the recovered row \(T_1\).  Since
\(T_1'[i]\le v-2\), the only way to avoid \(v\gg T_1[i]\) would be the
remaining possibility
\[
  T_1'[i]=v-2,
  \qquad
  T_1[i]=v-1.
\]
But the value \(v-1\) is not supplied by \(H\), and it is not already the entry
\(T_1'[i]\) in this exceptional case.  Because the row-insertion monotonicity
just recalled allows position \(i\) to increase by at most one, this exceptional
case cannot occur.  Hence
\[
  v=T_0'[i]\gg T_1[i].
\]

There are now two cases.

First suppose that
\[
  T_0[i]\ge T_0'[i]=v.
\]
Then
\[
  T_0[i]\gg T_1[i],
\]
contradicting the column condition for the original valid two-row window
\((T_0,T_1)\).

It remains to consider the case
\[
  T_0[i]<T_0'[i].
\]
Let \(j\ge i\) be maximal such that the old row entry \(T_0'[j]\) is replaced
by an element strictly smaller than itself during the recovery insertion
\(\rowsert(T_0',F')\).  Then the old row entries
\[
  T_0'[i],T_0'[i+1],\ldots,T_0'[j]
\]
occur in the evicted sequence \(H\), in this left-to-right order.

We claim that the alpha index of \(T_0'[i]\), when this entry is row-inserted
into \(T_1'\) as part of the insertion of \(H\), is at least \(i+1\).  Indeed,
both the initial entry \(T_1'[i]\) and the recovered entry \(T_1[i]\) are at
least two less than \(T_0'[i]\): the first by the choice of \(i\), and the
second by the conclusion \(T_0'[i]\gg T_1[i]\).  Since position \(i\) is affected
at most once during the non-Case~0 part of the insertion of \(H\), at the moment
\(T_0'[i]\) is processed the current entry in position \(i\) is still too small
for \(T_0'[i]\) to interact at position \(i\).  Thus its alpha index is at least
\(i+1\).

Applying Lemma~\ref{lem:rowsert-alpha-monotonicity} to the consecutive entries
\(T_0'[i],\ldots,T_0'[j]\) in the insertion of \(H\), the alpha index of
\(T_0'[j]\) is at least \(j+1\).  By the definition of the alpha index, this
implies that the recovered row satisfies
\[
  T_0'[j]\ge T_1[j+1].
\]

By maximality of \(j\), either \(T_0\) ends at position \(j\), or the next entry
of \(T_0'\) is not replaced by a smaller value.  In the first case,
\(|T_0|=j+1\), while the displayed inequality involves the existing entry
\(T_1[j+1]\).  Hence \(|T_1|>|T_0|\), contradicting the length condition for the
original valid window \((T_0,T_1)\).

In the second case, \(T_0[j+1]\) exists and
\[
  T_0[j+1]\ge T_0'[j+1].
\]
Since \(T_0'\) is dual Dyck,
\[
  T_0'[j+1]\gg T_0'[j].
\]
Together with \(T_0'[j]\ge T_1[j+1]\), this gives
\[
  T_0[j+1]\gg T_1[j+1],
\]
again contradicting the column condition for \((T_0,T_1)\).

All possible cases contradict the validity of the original adjacent pair
\((T_0,T_1)\).  Therefore no bad column exists, and
\[
  T_0'[r]\le T_1'[r]+1
\]
for every common column \(r\).  Together with the row-validity and length
conditions proved above, this shows that \((T_0',T_1')\) is a valid two-row
window in top-to-bottom row order.
\end{proof}

\subsection{The tableau--factorization bijection}
The row and tableau insertion results give the Schur-positivity statements
for the first part of the paper.  The proof of the bijection is written in the
top-to-bottom row convention used throughout the tableau-insertion subsection.
The statistic in the statement remains the row-reading statistic of
Definition~\ref{def:dyck-tableau}: rows are read left-to-right, from bottom row
to top row.
For a semistandard Young tableau \(Q\), write \(m_i(Q)\) for the number of
entries of \(Q\) equal to \(i\).  In this paper semistandard tableaux that
record factorization weights use entries in \(\{0,1,2,\ldots\}\), and their
generating functions are Schur functions in the zero-indexed variables
\(\mathbf{x}=(x_0,x_1,x_2,\ldots)\).
\begin{theorem}[Tableau--factorization bijection]
\label{thm:tableau-factorization-bijection}
Let \(S\) be a finite multiset of integers and let \(d\ge 0\).  There is a
bijection between dual Dyck factorizations
\[
  \mathcal F=(F_0,F_1,\ldots)
\]
of the multiset \(S\) satisfying
\[
  \di(F_0F_1\cdots)=d
\]
and pairs \((P,Q)\) such that:
\begin{enumerate}[label=(\roman*)]
\item \(P\) is a Dyck tableau with entries exactly the multiset \(S\) and
      \(\di(\operatorname{RR}(P))=d\);
\item \(Q\) is a semistandard Young tableau of the same shape as \(P\);
\item the content of \(Q\) records the factorization weight, meaning
      \[
        m_i(Q)=|F_i|
      \]
      for every \(i\ge 0\).
\end{enumerate}
Equivalently, the bijection sends the factorization monomial
\[
  \prod_{i\ge 0}x_i^{|F_i|}
\]
to the semistandard-tableau weight of \(Q\).
\end{theorem}
\begin{proof}
We first prove the fixed-shape insertion step.  Fix a Young shape \(\lambda\),
an integer \(k\ge 0\), a multiset \(S\), and an integer \(d\).  We compare the
following two sets.
The first set consists of pairs \((P,F)\), where \(P\) is a Dyck tableau of
shape \(\lambda\), \(F\) is a dual Dyck sequence of length \(k\), the entries
of \(P\) together with the entries of \(F\) form the multiset \(S\), and
\[
  \di(\operatorname{RR}(P)F)=d.
\]
The second set consists of Dyck tableaux \(P'\) of shape \(\mu\), where
\(\mu\setminus\lambda\) is a horizontal strip of size \(k\), the entries of
\(P'\) form the multiset \(S\), and
\[
  \di(\operatorname{RR}(P'))=d.
\]
The local map sends \((P,F)\) to
\[
  P'=\texttt{tabsert}(P,F).
\]
Let the rows of \(P\), in the local order in which \(\texttt{tabsert}\)
processes them, be \(T_0,T_1,\ldots,T_r\), and set \(E_0=F\).  At row \(j\),
the algorithm applies
\[
  (E_{j+1},T_j')=\rowsert(T_j,E_j)
\]
and replaces \(T_j\) by \(T_j'\), continuing until the evicted sequence is empty
or until a final nonempty evicted sequence is appended as a new bottom row.  By
Lemma~\ref{lem:rowsert-dual-dyck}, every updated row and every evicted sequence
is dual Dyck.  Since each row-insertion step only moves entries among the row,
the input sequence, and the evicted sequence, the multiset of entries is
preserved.
It remains in the forward direction to check the tableau shape, the column
condition, and \(\di\).  Consider two adjacent rows during an insertion step.
If the insertion into the upper row has a terminal Case~0 phase, that phase
only appends a terminal suffix to that upper row and sends no entries farther
down.  On the preceding no-Case~0 part, Lemma~\ref{lem:two-row-tableau-insertion}
applies to the adjacent two-row window.  It shows that the updated lower row has
length at most the old upper row length and that the adjacent column condition
is preserved.  The terminal appended suffix on the upper row is to the right of
the old upper row, so it cannot create a new column violation with the lower
row.  The same argument at the bottom boundary shows that a newly appended
bottom row has length at most the old bottom row.  Therefore the final shape is
obtained from \(\lambda\) by adding a horizontal strip of size \(k\), and the
output is a Dyck tableau.
For the statistic, apply Lemma~\ref{lem:rowsert-di-preservation} at each row
step.  The invariant is most naturally stated for an augmented row-reading word.
Immediately before row \(j\) is processed, this word is
\[
  T_r\cdots T_{j+1}\,T_j\,E_j\,T_{j-1}'\cdots T_0',
\]
where \(T_r,\ldots,T_{j+1}\) are the lower rows not yet reached and
\(T_{j-1}',\ldots,T_0'\) are the already updated upper rows.  Thus the current
row step replaces the contiguous block \(T_jE_j\) by \(E_{j+1}T_j'\).  If
\(\rowsert(R,G)=(E,R')\), then
\[
  \di(RG)=\di(ER').
\]
The blocks \(RG\) and \(ER'\) have the same multiset of entries, so their
cross-\(\di\) contributions with the fixed outside context agree: such cross
terms depend only on the value counts in the moving block.  The internal
contribution is preserved by the displayed row-level equality.  Thus the
augmented word has the same \(\di\) before and after each row step.  At the end
of the process the augmented word is exactly the row-reading word of the output
tableau, including the final evicted sequence as a new bottom row if one is
created.  Hence
\[
  \di(\operatorname{RR}(P)F)=\di(\operatorname{RR}(P')).
\]
The forward local map therefore lands in the stated codomain.
We describe the inverse local map.  Let \(P'\) have shape \(\mu\) with
\(\mu\setminus\lambda\) a horizontal strip of size \(k\).  In each row, the
cells of \(\mu\setminus\lambda\) form a terminal suffix.  Write \(F_j^+\) for
that terminal suffix in row \(j\), read left-to-right, and write \(T_j^-\) for
the initial segment that remains after the suffix is removed.
Process rows from bottom to top.  Suppose the accumulated sequence passed up
from the lower part is \(E_{j+1}\).  Apply
\[
  (T_j^{\mathrm{old}},F_j^-)=\worsert(E_{j+1},T_j^-),
\]
and then pass upward the sequence
\[
  E_j=F_j^-F_j^+.
\]
At the end of the upward pass, \(E_0\) is the recovered inserted sequence
\(F\), and the recovered rows \(T_j^{\mathrm{old}}\) form the tableau \(P\).
The validity of this reverse procedure is exactly the role of
Lemmas~\ref{lem:reverse-truncated-row-worsert} and
\ref{lem:second-reverse-insertion-valid-two-row}.  The induction is on the
rows, moving upward.  The carried object \(E_{j+1}\) is not asserted to be the
row already recovered below \(j\).  Rather, it is the lower input \(T_2\) in
Lemma~\ref{lem:reverse-truncated-row-worsert} for the next three-row window.
The induction hypothesis is precisely that this carried sequence is dual Dyck
and satisfies the compatibility and length hypotheses needed to serve as that
\(T_2\).  This is true at the bottom with the empty sentinel lower row.
After the reverse pass through row \(j\), Lemma~\ref{lem:reverse-truncated-row-worsert}
gives that \(E_j=F_j^-F_j^+\) is again a dual Dyck sequence compatible with
the next row above, when such a row exists.  The adjacent recovered rows below
are validated one step later by Lemma~\ref{lem:second-reverse-insertion-valid-two-row},
when the reverse pass through the row above is performed.

At row \(j\), the number of peeled-off boxes is \(k_j=|F_j^+|\).  At the bottom
row, \(E_{j+1}\) is empty.  Otherwise, the previous lower-row pass gives
\(|E_{j+1}|\le \mu_{j+1}\).  Since \(\mu\setminus\lambda\) is a horizontal
strip, \(\mu_{j+1}\le \lambda_j=|T_j^-|\).  Hence in every row
\(|E_{j+1}|\le |T_j^-|\), which is the length hypothesis needed in
Lemma~\ref{lem:reverse-truncated-row-worsert}:
\[
  0\le k_j\le |T_1|-|T_2|.
\]
Here the lemma's \(T_1\) is the full current row \(T_j^-F_j^+\), and its
\(T_2\) is \(E_{j+1}\).  At the top row there is no further row above to check
after the pass, but the same proof as
Lemma~\ref{lem:reverse-truncated-row-worsert}, with the top row \(T_0\) and the
final column-condition conclusion omitted, still gives that Case~0 does not
occur and that \(E_0\) is dual Dyck.
Inducting upward over the rows produces a valid Dyck tableau \(P\) of shape
\(\lambda\) and a dual Dyck sequence \(F\) of length \(k\).  The entries are
preserved because each \(\worsert\) step only rearranges the current row and
accumulated sequence, and the same \(\di\)-preservation argument as above gives
\[
  \di(\operatorname{RR}(P)F)=\di(\operatorname{RR}(P'))=d.
\]
Thus the inverse local map lands in the stated domain.
The two local maps are mutual inverses.  On the no-Case~0 portions, this is the
row-level reversibility of Theorem~\ref{thm:rowsert-global-reversibility}.  The
terminal Case~0 suffixes
that are appended in the forward map are precisely the suffixes \(F_j^+\) peeled
off before applying \(\worsert\) in the inverse map.  Therefore the inverse
process reconstructs the same row states and carried sequences in reverse
order, and the fixed-shape local map is a bijection.
We iterate the local bijection.  Start with a dual Dyck factorization
\((F_0,F_1,\ldots)\) and set \(P^{(-1)}=\varnothing\).  For \(i\ge 0\), set
\[
  P^{(i)}=\texttt{tabsert}(P^{(i-1)},F_i).
\]
Only finitely many factors are nonempty.  At step \(i\), the local bijection
says that the shape grows by a horizontal strip of size \(|F_i|\).  Label the
cells added at step \(i\) by \(i\).  This produces a filling \(Q\) of the final
shape.  The filling \(Q\) is semistandard: rows are weakly increasing because
later labels are appended to row ends.  Columns are strictly increasing because
each shape difference is a horizontal strip, and the nested Young shapes force
any cell above a newly added cell in the same column to have been added at an
earlier step.  The content condition \(m_i(Q)=|F_i|\) is immediate.  Empty
factors add no cells and correspond to labels absent from \(Q\).
The local \(\di\)-preservation gives, step by step,
\[
  \di(\operatorname{RR}(P^{(i-1)})F_i)
  =
  \di(\operatorname{RR}(P^{(i)})).
\]
To pass from this local equality to the whole factorization word, include the
untouched tail in the induction.  If every factor is empty, then \(P\) is empty
and the desired \(\di\)-identity is immediate.  Otherwise, let \(N\) be the last
nonempty factor index.
We prove by induction on \(i\) that
\[
  \di(F_0F_1\cdots F_iF_{i+1}\cdots F_N)
  =
  \di(\operatorname{RR}(P^{(i)})F_{i+1}\cdots F_N),
\]
and that \(\operatorname{RR}(P^{(i)})\) has the same value multiset as
\(F_0F_1\cdots F_i\).  The multiset assertion follows at each step because
the local insertion rearranges entries.  The displayed local
\(\di\)-preservation gives equality for the processed block after adding
\(F_i\).  The remaining cross terms with the fixed tail
\(F_{i+1}\cdots F_N\) depend only on the value counts in the processed block,
so they are unchanged when \(F_0F_1\cdots F_i\) is replaced by
\(\operatorname{RR}(P^{(i)})\).  Taking \(i=N\) gives
\[
  \di(F_0F_1\cdots)=\di(\operatorname{RR}(P)),
\]
so the final tableau has parameter \(d\).
Conversely, given \((P,Q)\), let \(P^{(i)}\) be the subtableau consisting of
cells whose recording labels are at most \(i\).  Since \(Q\) is semistandard,
\(P^{(i)}\setminus P^{(i-1)}\) is a horizontal strip.  Applying the inverse
local bijection for the labels in decreasing order recovers \(P^{(i-1)}\) and
the factor \(F_i\).  Because the local maps are mutual inverses at every step,
this recovers a unique dual Dyck factorization and is inverse to the forward
construction.  The theorem follows.
\end{proof}
\begin{corollary}[Schur positivity of dual Dyck symmetric functions]
\label{cor:dual-dyck-schur-positivity}
For every finite multiset \(S\) of integers and every \(d\ge 0\),
\[
  \DSstar(S,d;\mathbf{x})
  =
  \sum_P s_{\lambda(P)}(\mathbf{x}),
\]
where \(P\) ranges over all Dyck tableaux with entries the multiset \(S\) and
\(\di(\operatorname{RR}(P))=d\).
\end{corollary}
\begin{proof}
By definition,
\[
  \DSstar(S,d;\mathbf{x})
  =
  \sum_{\mathcal F}\prod_{i\ge 0}x_i^{|F_i|},
\]
where \(\mathcal F=(F_0,F_1,\ldots)\) ranges over the dual Dyck factorizations
with parameter \(d\).  Theorem~\ref{thm:tableau-factorization-bijection} sends
such factorizations to pairs \((P,Q)\), where \(P\) is a Dyck tableau with
parameter \((S,d)\), \(Q\) is a semistandard Young tableau of shape
\(\lambda(P)\), and \(m_i(Q)=|F_i|\).  Hence the monomial attached to
\(\mathcal F\) becomes
\[
  \prod_{i\ge 0}x_i^{m_i(Q)}=x^{\operatorname{wt}(Q)}.
\]
Therefore
\[
  \DSstar(S,d;\mathbf{x})
  =
  \sum_P\sum_{Q\in \operatorname{SSYT}(\lambda(P))}
  x^{\operatorname{wt}(Q)}.
\]
For each fixed shape \(\lambda\), the standard semistandard-tableau generating
function is
\[
  \sum_{Q\in \operatorname{SSYT}(\lambda)}x^{\operatorname{wt}(Q)}
  =s_\lambda(\mathbf{x}).
\]
Substituting this identity gives the claimed Schur expansion.
\end{proof}
\begin{corollary}[Schur positivity of affine Dyck symmetric functions]
\label{cor:affine-dyck-schur-positivity}
For every finite multiset \(S\) of integers and every \(d\ge 0\),
\[
  \DS(S,d;\mathbf{x})
  =
  \sum_P s_{\lambda(P)'}(\mathbf{x}),
\]
where \(P\) ranges over all Dyck tableaux with entries the multiset \(S\) and
\(\di(\operatorname{RR}(P))=d\), and \(\lambda(P)'\) denotes the conjugate
partition.
\end{corollary}
\begin{proof}
We use the fundamental-quasisymmetric expansion of the two factorization
sums.  We use the zero-indexed version of the usual fundamental
quasisymmetric basis.  For \(D\subseteq\{0,1,\ldots,n-2\}\), write
\[
  F_{n,D}(\mathbf{x})
  =
  \sum_{\substack{0\le i_0\le \cdots\le i_{n-1}\\
                  j\in D\Rightarrow i_j<i_{j+1}}}
  x_{i_0}\cdots x_{i_{n-1}}.
\]
Fix a word \(\pi=(\pi_0,\ldots,\pi_{n-1})\) with multiset \(S\) and
\(\di(\pi)=d\).  A factorization of \(\pi\) into consecutive factors is encoded
by a weakly increasing sequence of labels
\[
  a_0\le a_1\le \cdots\le a_{n-1},
\]
where \(a_j\) is the factor containing \(\pi_j\).  Equality
\(a_j=a_{j+1}\) means that the two adjacent letters lie in the same factor, and
strict inequality means that there is a factor boundary between them.
For dual Dyck factorizations, equality is allowed at position \(j\) only when
\[
  \pi_{j+1}\ge \pi_j+2.
\]
Thus a strict increase is forced on the complementary set
\[
  D^*(\pi)=\{j:\pi_{j+1}\le \pi_j+1\},
\]
and the contribution of all dual Dyck factorizations of the fixed word \(\pi\)
is \(F_{n,D^*(\pi)}(\mathbf{x})\).
For affine Dyck factorizations, equality is allowed at position \(j\) only when
\[
  \pi_{j+1}\le \pi_j+1.
\]
Thus the forced-strict set is
\[
  D(\pi)=\{j:\pi_{j+1}\ge \pi_j+2\},
\]
and the contribution of all affine Dyck factorizations of \(\pi\) is
\(F_{n,D(\pi)}(\mathbf{x})\).  Because the entries are integers, the two
conditions \(\pi_{j+1}\le \pi_j+1\) and \(\pi_{j+1}\ge \pi_j+2\) are exact
complements, so
\[
  D(\pi)=\{0,1,\ldots,n-2\}\setminus D^*(\pi).
\]
Let \(\Omega\) be the linear involution on the fundamental basis defined by
\[
  \Omega(F_{n,D})
  =
  F_{n,\{0,\ldots,n-2\}\setminus D}.
\]
This is the usual descent-complement involution on quasisymmetric functions.
Summing the preceding fixed-word comparison over all words \(\pi\) with
multiset \(S\) and \(\di(\pi)=d\), \(\Omega\) sends
\(\DSstar(S,d;\mathbf{x})\) to \(\DS(S,d;\mathbf{x})\).  By
Corollary~\ref{cor:dual-dyck-schur-positivity}, the dual function is symmetric,
and on symmetric functions \(\Omega\) restricts to the standard involution
\(\omega\).  Therefore
\[
  \DS(S,d;\mathbf{x})
  =
  \omega(\DSstar(S,d;\mathbf{x})).
\]
Using Corollary~\ref{cor:dual-dyck-schur-positivity} and the standard identity
\(\omega(s_\lambda)=s_{\lambda'}\), we get
\[
  \DS(S,d;\mathbf{x})
  =
  \omega\left(\sum_P s_{\lambda(P)}(\mathbf{x})\right)
  =
  \sum_P s_{\lambda(P)'}(\mathbf{x}),
\]
as claimed.
\end{proof}

\section{Dyck sequences and their decompositions}
\label{sec:dyck-decompositions}
This section begins the chain of decompositions used later in the
\(q,t\)-Catalan formula.  The affine and ordinary Dyck conventions are those
of Section~\ref{sec:definitions}: affine sequences are integer-valued
step-condition objects, while an ordinary Dyck sequence is normalized with
first entry and minimum equal to \(0\).  All components below are ordered
finite sequences; order is part of the data.
\begin{notation}[Concatenation and statistics]
For finite sequences \(A,B,C\), we write
\[
  A:B:C
\]
for their concatenation.  The notation \(\operatorname{len}(A)\) denotes the
length of \(A\), and \(|P|\) denotes the number of cells of a tableau \(P\).
For a nonempty sequence \(A\), the notation \(A[-1]\) denotes its final
entry.
For any finite sequence \(x=(x_0,\ldots,x_{r-1})\), define
\[
  \operatorname{nv}(x)
  =\#\{(i,j):0\le i<j<r,\ x_i=x_j\},
  \qquad
  \dinv(x)=\di(x)+\operatorname{nv}(x).
\]
When the entries of \(x\) are nonnegative, we write
\[
  \area(x)=\sum_{i=0}^{r-1}x_i .
\]
These conventions apply to concatenations such as \(E:F:G\) whenever the
components have the stated nonnegative interval bounds.
\end{notation}
\begin{definition}[Dyck sequence statistics]
\label{def:section4-dyck-statistics}
In this section, a \emph{Dyck sequence} is an ordinary Dyck sequence in the
sense of Definition~\ref{def:affine-dyck}.  Thus it is nonempty, starts with
\(0\), has all entries nonnegative, and satisfies the affine Dyck inequality.
For a Dyck sequence \(D\), its area is \(\area(D)\), and its dinv is
\(\dinv(D)\) as in the notation above.
\end{definition}
\begin{definition}[Reverse Dyck and interval-bounded affine Dyck sequences]
\label{def:section4-bounded-dyck}
A finite sequence \(E=(e_0,\ldots,e_{r-1})\) is a \emph{reverse Dyck sequence}
if the reversed sequence \((e_{r-1},\ldots,e_0)\) is an affine Dyck sequence.
Equivalently,
\[
  e_{i+1}\ge e_i-1
  \qquad\text{for }0\le i<r-1.
\]
The empty sequence is allowed and satisfies this condition vacuously.
For integers \(a\le b\), an \emph{affine \([a,b]\) Dyck sequence} is an affine
Dyck sequence all of whose entries lie in the interval \([a,b]\).  A
\emph{reverse \([a,b]\) Dyck sequence} is defined similarly using the reverse
Dyck condition.  If \(a>b\), the only sequence satisfying the interval
condition is the empty sequence.
\end{definition}
\begin{definition}[Extractable element]
\label{def:extractable-element}
Let \(C=(c_0,\ldots,c_{r-1})\) be a Dyck sequence.  An index \(j\) is called
\emph{eligible} if the following two conditions hold:
\begin{enumerate}[label=(\roman*)]
  \item there is exactly one index \(i<j\) with \(c_i=c_j-1\);
  \item if \(j+1<r\), then \(c_{j+1}\le c_j\).
\end{enumerate}
If at least one eligible index exists, the \emph{extractable element} of
\(C\) is the entry \(c_j\) at the leftmost eligible index \(j\).  Thus
extractability is a global leftmost-selection convention rather than a claim
that only one index can satisfy the two local eligibility conditions.  When a
construction asks for a nonfinal extractable element, the same leftmost rule is
applied among eligible indices \(j<r-1\).
\end{definition}
\begin{definition}[Dyck \(m\)-skeleton]
\label{def:dyck-m-skeleton}
Let \(m\ge 0\).  A Dyck sequence \(F\) is a \emph{Dyck \(m\)-skeleton} if
\[
  \max(F)=m=F[-1]
\]
and \(F\) has no nonfinal extractable element.  Equivalently, the only
eligible index in \(F\), if one exists at all, is allowed to be the last index.
\end{definition}
\begin{definition}[Dyck triples]
\label{def:dyck-triples}
Fix a nonnegative integer \(m\).  We use the following four families of
objects and their statistics.
\begin{enumerate}[label=\textup{(\arabic*)}, wide]
\item \textbf{Family 1.}
A Family 1 object is a Dyck sequence \(D\) with \(\max(D)=m\).  Its statistics
are \(\area(D)\) and \(\dinv(D)\).
\item \textbf{Type 1 Dyck triples.}
A Type 1 triple is an ordered triple \((E,F,G)\) such that \(F\) is a Dyck
\(m\)-skeleton, \(E\) is a reverse \([1,m]\) Dyck sequence, and \(G\) is an
affine \([0,m-1]\) Dyck sequence.  Its statistics are
\[
\begin{aligned}
  \area(E,F,G)&=\area(E:F:G),\\
  \dinv(E,F,G)&=\dinv(E:F:G)-\operatorname{len}(E).
\end{aligned}
\]
\item \textbf{Type 2 Dyck triples.}
A Type 2 triple is an ordered triple \((F,G,E)\) where \(F\) and \(G\) satisfy
the same conditions as in Type 1, while \(E\) is a reverse \([0,m-1]\) Dyck
sequence.  Its statistics are
\[
\begin{aligned}
  \area(F,G,E)&=\area(F:G:E)+\operatorname{len}(E),\\
  \dinv(F,G,E)&=\dinv(F:G:E)-\operatorname{len}(E).
\end{aligned}
\]
\item \textbf{Type 3 Dyck triples.}
A Type 3 triple is an ordered triple \((F,G,E)\) where \(F\) and \(G\) satisfy
the same conditions as in Type 1, while \(E\) is an affine \([0,m-1]\) Dyck
sequence.  Its statistics are the same as for Type 2:
\[
\begin{aligned}
  \area(F,G,E)&=\area(F:G:E)+\operatorname{len}(E),\\
  \dinv(F,G,E)&=\dinv(F:G:E)-\operatorname{len}(E).
\end{aligned}
\]
\item \textbf{Type 4 Dyck triples.}
A Type 4 triple is an ordered triple \((F,P,Q)\) such that \(F\) is a Dyck
\(m\)-skeleton, \(P\) is an at-most-two-column Dyck tableau whose entries lie
in \([0,m-1]\), and \(Q\) is a binary reverse semistandard Young tableau of
the same shape as \(P\).  Thus the entries of \(Q\) lie in \(\{0,1\}\), its rows
are strictly increasing from left to right, and its columns are weakly
increasing from top to bottom.  Its statistics are
\[
\begin{aligned}
  \area(F,P,Q)&=\area(F)+\sum_{u\in P}u+\sum_{v\in Q}v,\\
  \dinv(F,P,Q)&=\dinv\bigl(F:\operatorname{RR}(P)\bigr)-\sum_{v\in Q}v.
\end{aligned}
\]
\end{enumerate}
Empty auxiliary components are allowed.  Their elementwise interval, affine,
and reverse-Dyck conditions are vacuous, while the displayed length, shape,
and parameter requirements remain in force.  The following subsections
construct the recorded area- and dinv-preserving bijections between these
families.
\end{definition}
\subsection{The map from Family 1 to Type 1}
\begin{definition}[The map \(\phi_1\)]
\label{def:phi1}
Let \(D=(d_0,\ldots,d_{n-1})\) be a Dyck sequence with \(\max(D)=m\).
Let \(k\) be the last index with \(d_k=m\), and write
\[
  C=D[0:k+1],
  \qquad
  G=D[k+1:n].
\]
Starting with this mutable prefix \(C\), repeatedly remove the extractable
entry of \(C\) whenever that extractable entry is not the final entry of
\(C\), and append the removed value to a growing sequence \(E\).  When no
nonfinal extractable entry remains, set \(F=C\).  We define
\[
  \phi_1(D)=(E,F,G).
\]
\end{definition}
\begin{lemma}[Structure of extractable elements]
\label{lem:extractable-structure}
Let \(C=(c_0,\ldots,c_{r-1})\) be a Dyck sequence, and suppose that the
extractable entry of \(C\) occurs at index \(i\) with value \(x=c_i\).  Then
\(c_i\) is the leftmost occurrence of the value \(x\) in \(C\).  Consequently
\(i>0\) and \(c_{i-1}=x-1\).  Moreover, deleting the entry \(c_i\) from \(C\)
produces another Dyck sequence.
\end{lemma}
\begin{proof}
The value \(x\) is positive: if \(x=0\), then eligibility would require an
entry \(-1\) to the left of \(i\), impossible in an ordinary Dyck sequence.
Assume for contradiction that the value \(x\) occurs before \(i\), and let
\(k<i\) be its leftmost occurrence.  Since \(C\) starts at \(0\), has
nonnegative entries, and satisfies \(c_{t+1}\le c_t+1\), no entry before
position \(k\) is greater than or equal to \(x\).  The adjacent inequality at
the first occurrence of \(x\) then forces \(c_{k-1}=x-1\).  Thus the unique
\(x-1\) lying to the left of the eligible index \(i\) already lies to the left
of \(k\).
Let \(s\ge 0\) be maximal such that
\[
  (c_k,c_{k+1},\ldots,c_{k+s})=(x,x+1,\ldots,x+s).
\]
This rising chain must stop before \(i\), because \(c_i=x\) and \(k<i\).  The
index \(k+s\) is eligible: its unique predecessor value is either the already
identified unique \(x-1\), if \(s=0\), or the preceding entry \(x+s-1\) in the
chain, if \(s>0\); and maximality of the chain, together with the affine Dyck
inequality, gives the successor condition.  This eligible index lies strictly
before \(i\), contradicting the definition of the extractable entry as the
leftmost eligible entry.  Hence \(c_i\) is the leftmost occurrence of \(x\).
Since \(c_i\) is the first occurrence of \(x\) and \(C\) is affine Dyck, the
entry immediately before it must be \(x-1\): the affine inequality gives
\(c_i\le c_{i-1}+1\), while any value \(\ge x\) before \(i\) would contradict
leftmostness.  Finally, deleting \(c_i=x\) can only affect the adjacent
inequality across the deletion site.  The new adjacent pair is
\((c_{i-1},c_{i+1})\), if \(c_{i+1}\) exists, and eligibility gives
\(c_{i+1}\le x=c_{i-1}+1\).  The first entry remains \(0\), all entries remain
nonnegative, and all other affine inequalities are unchanged.  Thus the
resulting sequence is again Dyck.
\end{proof}
\begin{proposition}[Well-definedness of \(\phi_1\)]
\label{prop:phi1-well-defined}
For every Family 1 object \(D\) with \(\max(D)=m\), the output
\(\phi_1(D)=(E,F,G)\) is a Type 1 Dyck triple of height \(m\).
\end{proposition}
\begin{proof}
The extraction loop terminates because each iteration removes one entry from
\(C\).  It never removes the final entry of \(C\).  Initially that final entry
is the last occurrence of the maximum value \(m\) in \(D\), so the same
terminal \(m\) remains at the end of the mutable prefix throughout the loop.
The suffix \(G\) inherits the affine Dyck inequality from \(D\).  Since \(k\)
is the last index at which the maximum value \(m\) occurs, every entry of
\(G\) lies in \([0,m-1]\).  If \(G\) is empty, including the case \(m=0\), the
interval and affine conditions are vacuous.  Thus \(G\) is an affine
\([0,m-1]\) Dyck sequence.
The initial prefix \(D[0:k+1]\) is a Dyck sequence, and Lemma~\ref{lem:extractable-structure}
shows inductively that every extraction leaves a Dyck sequence.  Therefore the
terminal sequence \(F\) is Dyck.  Its last entry is still \(m\), no larger
entry has been introduced, and the entry \(m\) remains present, so
\(\max(F)=m=F[-1]\).  The loop stops exactly when \(F\) has no nonfinal
extractable entry.  Hence \(F\) is a Dyck \(m\)-skeleton.
Every extracted value lies in \([1,m]\).  It is at most \(m\) because it came
from the original prefix of \(D\), and it cannot be \(0\) because an eligible
entry of value \(0\) would require a \(-1\) to its left.
It remains to prove the reverse-Dyck condition for \(E\).  Let \(x\) and
\(y\) be consecutive extracted values, in the order appended to \(E\).  Suppose
\(x\) is extracted from position \(i\) of a current Dyck sequence \(C\), leaving
\(C'\), and suppose \(y\) is then extracted from position \(j\) of \(C'\).  By
Lemma~\ref{lem:extractable-structure}, the entry just before the removed
\(x\) is \(x-1\), so \(C'[i-1]=x-1\).  Also, because \(x\) was eligible,
\(C'[i]=C[i+1]\le x\) whenever this successor exists.
Assume, toward a contradiction, that \(y<x-1\).  If \(j<i-1\), then the
entries, successor, and predecessor-count data relevant to position \(j\) were
unchanged by removing \(x\), so the same position would have been eligible in
\(C\) before the leftmost eligible position \(i\), a contradiction.  If
\(j=i-1\), then \(y=C'[i-1]=x-1\), also a contradiction.
Thus \(j\ge i\).  Since \(C'\) starts at \(0\) and reaches the value \(x-1\)
at position \(i-1<j\), it reaches every value \(y,y+1,\ldots,x-1\) before
position \(j\).  Let \(p_t\) be the first occurrence of \(t\) in \(C'\), for
\(y\le t\le x-1\).  Then \(p_y<j\).  The extractability of \(C'[j]=y\) gives a
unique \(y-1\) to the left of \(j\); that occurrence must lie before the first
\(y\), so \(p_y\) has exactly one \(y-1\) to its left.  Since \(j\) is the
leftmost eligible index of \(C'\), the earlier index \(p_y\) cannot be
eligible.  Its predecessor-count condition holds, so its successor condition
must fail.  Together with the affine inequality, this forces
\(C'[p_y+1]=y+1\).

Suppose for some \(t<x-1\) that the first occurrence \(p_t\) lies before \(j\)
and has exactly one \(t-1\) to its left.  Then \(p_t\) is not eligible, so its
successor must be \(t+1\).  Hence \(p_{t+1}=p_t+1\).  Since \(p_t\) is the first
occurrence of \(t\), it is the unique \(t\) to the left of the first \(t+1\).
Inducting on \(t\) gives a consecutive first-occurrence chain
\[
  y,y+1,\ldots,x-1
\]
before \(j\), and the first occurrence of \(x-1\) has exactly one \(x-2\) to
its left.  But the extraction of \(x\) from \(C\) had exactly one \(x-1\) to
the left of the removed position \(i\), namely the predecessor \(C[i-1]\).
After deletion this same entry is \(C'[i-1]\), so \(p_{x-1}=i-1\).
If \(C'[i]\le x-1\), then the index \(i-1\) would be eligible before \(j\), a
contradiction.

The only remaining possibility is \(C'[i]=x\).  In that case, let \(s\ge0\)
be maximal such that
\[
  (C'[i],C'[i+1],\ldots,C'[i+s])=(x,x+1,\ldots,x+s).
\]
The endpoint of this maximal rising chain is eligible by the same argument
used in Lemma~\ref{lem:extractable-structure}: its unique predecessor value is
either \(C'[i-1]=x-1\), if \(s=0\), or the preceding entry of the chain, if
\(s>0\), and maximality plus the affine inequality gives the successor
condition.  Therefore the next leftmost eligible index lies on this chain and
has value at least \(x\), contradicting \(y<x-1\).
Hence every pair of consecutive extracted values satisfies \(y\ge x-1\).  This
is exactly the reverse-Dyck adjacent inequality for \(E\).  Together with the
interval bound \([1,m]\), this proves that \(E\) is a reverse \([1,m]\) Dyck
sequence.  Thus \((E,F,G)\) is a Type 1 Dyck triple.
\end{proof}
\subsection{The inverse map from Type 1 to Family 1}
\begin{definition}[Injection]
\label{def:injection}
Let \(C\) be a Dyck sequence and let \(e\) be a positive integer with
\(e\le \max(C)+1\).  The \emph{injection} of \(e\) into \(C\), denoted
\(\operatorname{inject}(C,e)\), is obtained by finding the first index \(p\)
with \(C[p]=e-1\) and inserting \(e\) immediately after position \(p\).
\end{definition}
\begin{definition}[The inverse map \(\phi_1^{-1}\)]
\label{def:phi1-inverse}
Let \((E,F,G)\) be a Type 1 Dyck triple of height \(m\), and write
\(E=(e_0,\ldots,e_{\ell-1})\).  Define a sequence of intermediate Dyck
sequences by setting \(C^{(\ell)}=F\) and, for
\(r=\ell-1,\ell-2,\ldots,0\), setting
\[
  C^{(r)}=\operatorname{inject}(C^{(r+1)},e_r).
\]
If \(E\) is empty, this says that \(C^{(0)}=F\).  Set
\(D=C^{(0)}:G\).  We define
\[
  \phi_1^{-1}(E,F,G)=D.
\]
\end{definition}
\begin{proposition}[The maps \(\phi_1\) and \(\phi_1^{-1}\) are inverse]
\label{prop:phi1-inverse}
For fixed \(m\ge 0\), the maps \(\phi_1\) and \(\phi_1^{-1}\) are mutual
inverses between Family 1 Dyck sequences with maximum \(m\) and Type 1 Dyck
triples of height \(m\).
\end{proposition}
\begin{proof}
We first record the local fact that makes the inverse map well-defined.  Let
\(C\) be a Dyck sequence with \(\max(C)=m=C[-1]\), and let \(1\le e\le m\).
Because \(C\) starts at \(0\), ends at \(m\), and satisfies the affine step
condition, it contains every value \(0,1,\ldots,m\) before or at its final
entry.  Hence the first occurrence of \(e-1\) exists and, since \(e-1<m\),
occurs before the final \(m\).  Inserting \(e\) after this first occurrence
preserves the affine inequalities at the insertion site: the new left adjacent
pair is \((e-1,e)\), and the old successor of the first \(e-1\), if it exists,
was at most \((e-1)+1=e\).  Thus injection produces another Dyck sequence with
maximum \(m\) and last entry \(m\).
Now let \((E,F,G)\) be a Type 1 triple of height \(m\).  If \(m=0\), the
interval conventions force \(E\) and \(G\) to be empty.  In general, each entry
of \(E\), if any, lies in \([1,m]\), so the local fact applies inductively to
the right-to-left injections into \(F\).  The mutable sequence \(C\) remains
Dyck with maximum and last entry \(m\).  Concatenating with \(G\) gives a Dyck
sequence: the only new adjacent inequality to check is the boundary from the
last entry \(m\) of \(C\) to the first entry of \(G\), and if \(G\) is nonempty
that first entry is at most \(m-1\).  Therefore \(\phi_1^{-1}(E,F,G)\) is a
Family 1 object.
We next show \(\phi_1^{-1}\circ\phi_1=\mathrm{id}\).  Let \(D\) be a Family 1
object, and write \(D=C^0:G\) at the last occurrence of \(m\).  Suppose the
extraction process removes the entries \(e_0,e_1,\ldots,e_{\ell-1}\), producing
successive prefixes
\[
  C^0,C^1,\ldots,C^\ell=F.
\]
Fix one extraction step, in which \(x=e_t\) is removed from position \(i\) of
\(C^t\).  The eligibility condition says that exactly one \(x-1\) lies to the
left of position \(i\), and Lemma~\ref{lem:extractable-structure} says that the
immediate predecessor \(C^t[i-1]\) is \(x-1\).  After deleting \(x\), this same
entry is the first occurrence of \(x-1\) in \(C^{t+1}\).  Thus injecting \(x\)
into \(C^{t+1}\) places \(x\) immediately after the same predecessor and
recovers \(C^t\).  Reversing the extraction steps from \(t=\ell-1\) down to
\(0\) recovers \(C^0\), and then concatenating with the unchanged suffix \(G\)
recovers \(D\).
It remains to show \(\phi_1\circ\phi_1^{-1}=\mathrm{id}\).  If \(E\) is empty,
then \(\phi_1^{-1}(E,F,G)=F:G\).  The last occurrence of \(m\) in this sequence
is the final entry of \(F\), because all entries of \(G\) are at most \(m-1\),
and the extraction loop removes nothing from \(F\) because \(F\) is a Dyck
\(m\)-skeleton.  Hence the result is \((\emptyset,F,G)\).
Assume now that \(E=(e_0,\ldots,e_{\ell-1})\) is nonempty.  Let
\(C_\ell=F\), and for \(k=\ell-1,\ell-2,\ldots,0\), let
\[
  C_k=\operatorname{inject}(C_{k+1},e_k).
\]
Since \(E\) is reverse Dyck, for every \(k<\ell-1\) we have
\(e_k\le e_{k+1}+1\).  Because entries are integral, each adjacent pair falls
into one of the two cases \(e_k\le e_{k+1}\) or \(e_k=e_{k+1}+1\).
We prove by downward induction on \(k\) that, in \(C_k\), the occurrence of
\(e_k\) just inserted in the step \(C_{k+1}\mapsto C_k\) is a nonfinal eligible
index and no earlier index is eligible.  This says exactly that the most
recently injected entry is the extractable entry selected by
Definition~\ref{def:extractable-element}; it does not assert uniqueness among
all locally eligible indices to its right.
For the base case \(k=\ell-1\), the sequence \(F\) has no nonfinal eligible
index, because it is a Dyck \(m\)-skeleton.  Injecting \(e_{\ell-1}\) after the
first \(e_{\ell-1}-1\) creates an eligible entry: it has exactly one such
predecessor value to its left, and its old successor, if any, is at most
\(e_{\ell-1}\).  No earlier index becomes eligible.  Entries before the
insertion site have unchanged local data except for the insertion site's
predecessor, whose new successor is one larger than itself and hence fails the
successor condition.  Thus the invariant holds at \(k=\ell-1\).
For the induction step, assume the invariant for \(k+1\).  Let the previously
inserted occurrence of \(e_{k+1}\) be at position \(q\) in \(C_{k+1}\).  It is
the extractable entry of \(C_{k+1}\), so Lemma~\ref{lem:extractable-structure}
shows that it is the leftmost occurrence of the value \(e_{k+1}\).  Inject
\(e_k\) after the first occurrence \(p\) of \(e_k-1\).  The newly inserted
entry is eligible by the same local argument used above.
If \(e_k\le e_{k+1}\), then the first occurrence of \(e_k-1\) occurs before
position \(q\): an affine Dyck sequence starting at \(0\) cannot reach the
value \(e_{k+1}\ge e_k\) without first reaching \(e_k-1\).  Hence the new
insertion lies weakly to the left of the previous inserted occurrence.  Every
index before the new insertion either has unchanged eligibility data from
\(C_{k+1}\) or is the index \(p\), which now has successor \(e_k=(e_k-1)+1\)
and so is not eligible.  By the induction hypothesis, there was no eligible
index before the previous inserted occurrence; therefore none lies before the
new inserted entry.
If \(e_k=e_{k+1}+1\), then \(e_k-1=e_{k+1}\), and the first occurrence of this
value is exactly the previous inserted entry at \(q\), by
Lemma~\ref{lem:extractable-structure}.  Injection places \(e_k\) immediately
after it.  The old occurrence of \(e_{k+1}\) ceases to be eligible, because its
new successor is \(e_k=e_{k+1}+1>e_{k+1}\).  The new \(e_k\) is eligible: the
old \(e_{k+1}\) is the only \(e_k-1\) to its left, and its successor, if any,
is the old successor of an eligible \(e_{k+1}\), hence is at most
\(e_{k+1}\le e_k\).  Earlier indices have unchanged local data and were not
eligible by induction.  The invariant follows in this case as well.
The induction is complete.  Starting from \(C_0\), the extraction loop in
\(\phi_1\) therefore removes the entries \(e_0,e_1,\ldots,e_{\ell-1}\) in order
and leaves \(F=C_\ell\).  It then stops because \(F\) is a Dyck
\(m\)-skeleton.  Finally, the splitting step recovers the same suffix \(G\): in
\(C_0:G\), the last entry of \(C_0\) is \(m\), while every entry of \(G\) lies
in \([0,m-1]\).  Hence the last occurrence of \(m\) is the final entry of
\(C_0\).  Thus \(\phi_1(\phi_1^{-1}(E,F,G))=(E,F,G)\), completing the proof.
\end{proof}
\subsection{\texorpdfstring{Area and \(\dinv\) preservation for \(\phi_1\)}{Area and dinv preservation for phi1}}
\begin{proposition}[Statistics preserved by \(\phi_1\)]
\label{prop:phi1-statistics}
Let \(D\) be a Family 1 Dyck sequence and let \(\phi_1(D)=(E,F,G)\).  Then
\[
  \area(E,F,G)=\area(D),
  \qquad
  \dinv(E,F,G)=\dinv(D).
\]
\end{proposition}
\begin{proof}
It is enough to analyze one extraction step.  Write the full
concatenation before extraction as
\[
  E_{\mathrm{old}}:A:x:B:G,
\]
where \(x\) is the extractable entry in the mutable prefix.  After extraction
the full concatenation is
\[
  E_{\mathrm{old}}:x:A:B:G.
\]
Thus only the relative order of \(x\) with the entries of \(A\) changes.
Extraction preserves the multiset, so area is unchanged and
\(\operatorname{nv}\) is unchanged.
For \(\di\), before the move an entry \(a\in A\) contributes with \(x\) exactly
when \(a=x+1\); after the move it contributes exactly when \(a=x-1\).  Since
\(x\) is extractable, exactly one entry of \(A\) has value \(x-1\).  By
Lemma~\ref{lem:extractable-structure}, this occurrence of \(x\) is the
leftmost \(x\), and in a normalized affine Dyck sequence no value \(x+1\) can
appear before the leftmost \(x\).  Hence each extraction increases \(\di\) by
exactly \(1\) and leaves \(\operatorname{nv}\) unchanged.  After
\(\operatorname{len}(E)\) extractions,
\[
  \di(E:F:G)=\di(D)+\operatorname{len}(E),
  \qquad
  \operatorname{nv}(E:F:G)=\operatorname{nv}(D).
\]
Using the Type 1 statistic definition,
\[
  \dinv(E,F,G)=\dinv(E:F:G)-\operatorname{len}(E)=\dinv(D),
\]
while \(\area(E,F,G)=\area(E:F:G)=\area(D)\).
\end{proof}
\subsection{The map from Type 1 to Type 2}
\begin{definition}[The map \(\phi_2\)]
\label{def:phi2}
Let \((E,F,G)\) be a Type 1 Dyck triple of height \(m\), and write
\(E=(e_0,\ldots,e_{\ell-1})\).  Define
\[
  E^-=(e_0-1,\ldots,e_{\ell-1}-1),
\]
with the convention that \(E^-=\emptyset\) when \(E=\emptyset\).  The map
\(\phi_2\) is
\[
  \phi_2(E,F,G)=(F,G,E^-).
\]
\end{definition}
\begin{proposition}[The Type 1--Type 2 shift]
\label{prop:phi2-bijection}
For fixed \(m\ge 0\), the map \(\phi_2\) is an area- and
\(\dinv\)-preserving bijection from Type 1 Dyck triples of height \(m\) to
Type 2 Dyck triples of height \(m\).
\end{proposition}
\begin{proof}
Let \((E,F,G)\) be Type 1.  Since every entry of \(E\) lies in \([1,m]\),
every entry of \(E^-\) lies in \([0,m-1]\).  The reverse-Dyck condition is
invariant under translating all entries by a constant: for adjacent entries,
\[
  e_i\le e_{i+1}+1
  \qquad\Longleftrightarrow\qquad
  e_i-1\le (e_{i+1}-1)+1.
\]
Thus \(E^-\) is a reverse \([0,m-1]\) Dyck sequence.  The components \(F\) and
\(G\) are unchanged and already satisfy the Type 2 requirements.  Hence
\((F,G,E^-)\) is Type 2.  Conversely, from a Type 2 triple \((F,G,H)\), adding
\(1\) to every entry of \(H\) gives a reverse \([1,m]\) Dyck sequence
\(H^+\), and \((H^+,F,G)\) is Type 1.  These two entrywise translations are
inverse, so \(\phi_2\) is a bijection.  The empty-component convention covers
\(m=0\), where the interval \([1,0]\) or \([0,-1]\) forces the shifted
component to be empty.
Set \(H=F:G\).  Since \(\sum E=\sum E^-+\operatorname{len}(E)\),
\[
  \area(E,F,G)=\sum E+\sum H
  =\sum H+\sum E^-+\operatorname{len}(E^-)
  =\area(F,G,E^-).
\]
Thus area is preserved.
For \(\dinv\), first observe that translating \(E\) by \(-1\) preserves both
internal \(\di\) and internal equal-value pairs, so
\(\dinv(E)=\dinv(E^-)\).  It remains only to compare cross terms with
\(H=F:G\).  In the Type 1 concatenation \(E:H\), a pair with \(x\in E\) and
\(y\in H\) contributes to \(\dinv\) exactly when
\[
  x=y+1 \quad\text{or}\quad x=y.
\]
For the corresponding entry \(b=x-1\in E^-\), these alternatives are
\[
  y=b \quad\text{or}\quad y=b+1.
\]
In the Type 2 concatenation \(H:E^-\), the cross pair \((y,b)\) contributes
exactly under the same two alternatives, namely when \(y=b\) or \(y=b+1\).
The internal contribution of \(H\) is unchanged.  Hence
\[
  \dinv(E:F:G)=\dinv(F:G:E^-).
\]
Finally \(\operatorname{len}(E)=\operatorname{len}(E^-)\), and the Type 1 and
Type 2 definitions subtract this same length.  Therefore
\(\dinv(E,F,G)=\dinv(F,G,E^-)\).
\end{proof}
\subsection{The map from Type 2 to Type 3}
We next use a normal-form transport between reverse and affine interval Dyck
sequences.  This construction is applied only to the final component of a
Type 2 triple; the skeleton component \(F\) and the affine component \(G\) stay
fixed.
For a word \(Y\) and an integer \(k\), write
\[
  Y|_{\{k,k+1\}}
\]
for the subsequence of \(Y\) consisting of entries equal to \(k\) or
\(k+1\), in their original order.  If \(B\) is a word in the alphabet
\(\{k,k+1\}\), the \emph{rank} of an occurrence of \(k+1\) in \(B\) is the
number of occurrences of \(k\) preceding it in \(B\).
\begin{lemma}[Local rank-insertion step]
\label{lem:local-rank-insertion}
Fix \(k\ge 1\).
\begin{enumerate}[label=\textup{(\alph*)}]
\item Let \(Y\) be an affine \([0,k+1]\) Dyck sequence.  Delete all entries
\(k+1\), and record \(B=Y|_{\{k,k+1\}}\).  The remaining sequence \(X\) is
affine \([0,k]\).  Conversely, from \(X\) and \(B\), insert the occurrences
of \(k+1\), in their order in \(B\), at the leftmost positions, to the right
of previously inserted \(k+1\)'s, so that each occurrence has its prescribed
number of \(k\)'s before it in the resulting \(k/(k+1)\)-subword.  This
reconstructs \(Y\).
\item Let \(Y\) be a reverse \([0,k+1]\) Dyck sequence.  Delete all entries
\(k+1\), and record \(B=Y|_{\{k,k+1\}}\).  The remaining sequence \(X\) is a
reverse \([0,k]\) Dyck sequence.  Conversely, from \(X\) and \(B\), insert the
occurrences of \(k+1\) using the rightmost rank rule: process the prescribed
occurrences of \(k+1\) from right to left, and place each at the rightmost
position, to the left of previously inserted \(k+1\)'s, so that it has its
prescribed number of \(k\)'s before it in the resulting \(k/(k+1)\)-subword.
This reconstructs \(Y\).
\end{enumerate}
In both cases,
\[
  \di(Y)=\di(X)+\di(B).
\]
\end{lemma}
\begin{proof}
We prove the affine statement first.  Deleting entries \(k+1\) leaves all
entries in \([0,k]\).  It cannot create an affine-Dyck violation.  Indeed, if
a new adjacency \(u,v\) crosses a deleted block and \(u\) is present, then in
the original sequence the first deleted \(k+1\) immediately after \(u\) forced
\(k+1\le u+1\).  Since \(u\) remains after deletion, \(u\le k\), so
\(u=k\).  The new right endpoint satisfies \(v\le k\), hence
\(v\le u+1\).
For the inverse construction, the leftmost rank rule reproduces the prescribed
\(k/(k+1)\)-subsequence \(B\) by definition.  The sequence remains affine.  A
newly inserted \(k+1\) has no left neighbor, or has left neighbor \(k\) or a
previously inserted \(k+1\), so the left adjacent inequality is valid.  Its
right neighbor, if any, is at most \(k+1\), so the right adjacent inequality is
automatic.  No other adjacent inequality changes.
These deletion and insertion operations are inverse.  Indeed, in an affine
\([0,k+1]\) sequence, an occurrence of \(k+1\) can only be first in the word or
can have immediate predecessor \(k\) or \(k+1\).  Thus, once its rank in the
recorded \(k/(k+1)\)-subsequence and the order among equal inserted values are
fixed, it already occupies the leftmost possible position compatible with
those data.
The reverse statement is the mirror image.  Deleting \(k+1\)'s from a reverse
\([0,k+1]\) sequence preserves the reverse adjacent inequality: if a new
adjacency \(u,v\) crosses a deleted block and \(v\) is present, then the last
deleted \(k+1\) immediately before \(v\) forced \(k+1\le v+1\); since
\(v\) remains, \(v=k\), and then \(u\le k\le v+1\).  The rightmost rank
rule inserts each new \(k+1\) so that it has no right neighbor, or has right
neighbor \(k\) or \(k+1\), which is exactly what is needed for the new right
adjacent reverse inequality.  The left adjacent reverse inequality is automatic
because all entries are at most \(k+1\).  Conversely, in a reverse
\([0,k+1]\) sequence, each occurrence of \(k+1\) is last or has immediate
successor \(k\) or \(k+1\), so it is forced to be in the rightmost possible
position with its rank data.  Hence the reverse deletion and insertion
operations are inverse.
Finally, in either the affine or reverse setting, the only \(\di\)-pairs that
involve a deleted or newly inserted \(k+1\) are pairs \((k+1,k)\).  Their
number is exactly \(\di(B)\), because \(B\) is the full subsequence of the
\(k\)'s and \(k+1\)'s with their order preserved.  All other \(\di\)-pairs are
unchanged.  Therefore \(\di(Y)=\di(X)+\di(B)\).
\end{proof}
\begin{proposition}[Affine--reverse transport]
\label{prop:affine-reverse-transport}
For every \(M\ge 0\), there is a multiset-preserving and
\(\di\)-preserving bijection
\[
  \mathrm{fw}_M:
  \{\text{reverse }[0,M]\text{ Dyck sequences}\}
  \longrightarrow
  \{\text{affine }[0,M]\text{ Dyck sequences}\}.
\]
Its inverse is denoted \(\mathrm{bk}_M\).
\end{proposition}
\begin{proof}
If \(M\le 1\), every \([0,M]\) interval sequence is both affine and reverse,
so we take \(\mathrm{fw}_M\) and \(\mathrm{bk}_M\) to be the identity.  Assume
\(M\ge 2\).
Starting from a reverse \([0,M]\) Dyck sequence \(Y\), apply the reverse
deletion step of Lemma~\ref{lem:local-rank-insertion} for
\(k=M-1,M-2,\ldots,1\).  This records words \(B_k\) and leaves a binary
sequence, which is both affine and reverse.  Reconstruct from the same binary
sequence using the affine leftmost insertion steps for \(k=1,2,\ldots,M-1\)
with the recorded words \(B_k\).  The resulting affine \([0,M]\) sequence is
\(\mathrm{fw}_M(Y)\).
The inverse map \(\mathrm{bk}_M\) starts from an affine \([0,M]\) Dyck
sequence, applies the affine deletion steps for \(k=M-1,M-2,\ldots,1\), and
then reconstructs from the same binary normal form using the reverse rightmost
insertion steps for \(k=1,2,\ldots,M-1\).  Lemma~\ref{lem:local-rank-insertion}
shows that the two constructions use the same normal-form data and are inverse
to one another.
Each local deletion/insertion step preserves the total multiset.  The
\(\di\)-identity in Lemma~\ref{lem:local-rank-insertion} shows that both a
sequence and its normal-form data have the same total value
\[
  \di(\text{binary base})+
  \sum_{k=1}^{M-1}\di(B_k).
\]
Thus \(\mathrm{fw}_M\) and \(\mathrm{bk}_M\) preserve \(\di\) as well as the
multiset.
\end{proof}
\begin{definition}[The map \(\phi_3\)]
\label{def:phi3}
Let \((F,G,E)\) be a Type 2 Dyck triple of height \(m\).  If \(m=0\), the
interval convention forces \(E=\emptyset\), and we set \(E'=\emptyset\).  If
\(m>0\), set
\[
  E'=\mathrm{fw}_{m-1}(E).
\]
The Type 2 to Type 3 map is
\[
  \phi_3(F,G,E)=(F,G,E').
\]
\end{definition}
\begin{proposition}[The Type 2--Type 3 transport]
\label{prop:phi3-bijection}
For fixed \(m\ge 0\), the map \(\phi_3\) is an area- and
\(\dinv\)-preserving bijection from Type 2 Dyck triples of height \(m\) to
Type 3 Dyck triples of height \(m\).
\end{proposition}
\begin{proof}
Only the final component changes.  Proposition~\ref{prop:affine-reverse-transport}
shows that \(E'\) is an affine \([0,m-1]\) Dyck sequence with the same multiset
and the same \(\di\) as \(E\), and that the construction is invertible using
\(\mathrm{bk}_{m-1}\).  Hence \((F,G,E')\) is Type 3 and \(\phi_3\) is a
bijection.
Area is preserved because \(E\) and \(E'\) have the same multiset and the same
length:
\[
  \area(F:G:E)+\operatorname{len}(E)
  =\area(F:G:E')+\operatorname{len}(E').
\]
For \(\dinv\), set \(H=F:G\).  The internal equal-value contribution
\(\operatorname{nv}(E)\) is determined by the multiset, and
\(\di(E)=\di(E')\) by Proposition~\ref{prop:affine-reverse-transport}.  The
cross contribution between \(H\) and the final component also depends only on
the final component's multiset: for \(h\in H\) and a final-component entry
\(e\), the pair contributes to \(\dinv(H:E)\) exactly when
\[
  h=e+1 \quad\text{or}\quad h=e.
\]
The same criterion, with the same multiplicities, holds for \(E'\).  Therefore
\[
  \dinv(F:G:E)=\dinv(F:G:E').
\]
Since \(\operatorname{len}(E)=\operatorname{len}(E')\), the common length
subtraction in the Type 2 and Type 3 statistic definitions gives
\[
  \dinv(F,G,E)=\dinv(F,G,E').
\]
Thus \(\phi_3\) preserves both statistics.
\end{proof}

\subsection{The chosen correspondence from Type 3 to Type 4}

\begin{proposition}[The Type 3--Type 4 correspondence]
\label{prop:phi4-bijection}
For fixed \(m\ge 0\), the construction below chooses an area- and
\(\dinv\)-preserving bijection
\[
  \phi_4:
  \{\text{Type 3 triples of height }m\}
  \longrightarrow
  \{\text{Type 4 triples of height }m\}
.
\]
The chosen bijection fixes the Dyck \(m\)-skeleton component.
\end{proposition}

\begin{proof}
Let \((F,G,E)\) be a Type 3 Dyck triple of height \(m\), and set
\[
  H=G:E.
\]
The two components \(G\) and \(E\) are affine \([0,m-1]\) Dyck sequences,
so \((G,E)\) is a two-factor affine Dyck factorization of the multiset of
entries of \(H\).  We use the affine, conjugate-shape form of the
Section~\ref{sec:insertion-algorithm} tableau--factorization result in the
two-factor specialization.  More explicitly, take the coefficient of
\(x_0^{|G|}x_1^{|E|}\) in
Corollary~\ref{cor:affine-dyck-schur-positivity}.  On the affine-factorization
side this coefficient counts the two-factor affine factorizations of words with
the same multiset and \(\di\)-value as \(H\).  On the tableau side it counts
pairs \((P,Q')\), where \(P\) is a Dyck tableau with the same multiset and
\(\di(\operatorname{RR}(P))=\di(H)\), and \(Q'\) is an ordinary semistandard
tableau of shape \(\lambda(P)'\) with \(|G|\) entries equal to \(0\) and
\(|E|\) entries equal to \(1\).
Transposing \(Q'\) gives a binary
reverse semistandard tableau \(Q\) of shape \(\lambda(P)\).  For each fixed
multiset, \(\di\)-value, and pair of factor lengths \((|G|,|E|)\), the equality
of the two finite coefficients, equivalently the equal cardinalities of the two
fibers, allows us to choose a bijection from this
affine-factorization fiber to the tableau fiber with the same entry multiset,
\(\di\)-value, and factor lengths \((|G|,|E|)\).  Thus \((G,E)\) is sent to a
pair \((P,Q)\) with the following properties:
\begin{enumerate}[label=\textup{(\roman*)}]
\item \(P\) is a Dyck tableau with entries in \([0,m-1]\), and because only two
      affine factors are present, \(P\) has at most two columns;
\item \(Q\) is a binary reverse semistandard Young tableau of the same shape as
      \(P\): rows are strictly increasing from left to right and columns are
      weakly increasing from top to bottom;
\item \(H\) and \(\operatorname{RR}(P)\) have the same multiset of entries;
\item \(\di(H)=\di(\operatorname{RR}(P))\);
\item the number of entries of \(Q\) equal to \(1\) is \(\operatorname{len}(E)\).
\end{enumerate}
Thus \(\phi_4\) is defined nonconstructively here; no separate concrete
algorithm is introduced.  Define
\[
  \phi_4(F,G,E)=(F,P,Q).
\]
Since the chosen correspondence between the finite fibers is
bijective and the component \(F\) is unchanged, this gives a bijection between
Type 3 and Type 4 triples of height \(m\).

We check the statistics.  Multiset preservation gives
\[
  \area(H)=\area(\operatorname{RR}(P))=\sum_{u\in P}u.
\]
Under the binary convention for \(Q\),
\[
  \sum_{v\in Q}v=\#\{v\in Q:v=1\}=\operatorname{len}(E).
\]
Therefore
\[
\begin{aligned}
  \area(F,G,E)
  &=\area(F:G:E)+\operatorname{len}(E)\\
  &=\area(F)+\area(H)+\operatorname{len}(E)\\
  &=\area(F)+\sum_{u\in P}u+
    \sum_{v\in Q}v
   =\area(F,P,Q).
\end{aligned}
\]

For \(\dinv\), the equality \(\di(H)=\di(\operatorname{RR}(P))\) and multiset
preservation imply
\[
  \dinv(H)=\dinv(\operatorname{RR}(P)).
\]
The cross contribution between the fixed prefix \(F\) and the suffix also
only depends on the suffix multiset: a pair consisting of \(f\in F\) and a
suffix entry \(h\) contributes to \(\dinv(F:H)\) exactly when
\[
  f=h+1\qquad\text{or}\qquad f=h.
\]
Since \(H\) and \(\operatorname{RR}(P)\) have the same multiset, these cross
contributions agree.  Hence
\[
  \dinv(F:H)=\dinv\bigl(F:\operatorname{RR}(P)\bigr).
\]
Using again \(\operatorname{len}(E)=\sum_{v\in Q}v\), we obtain
\[
\begin{aligned}
  \dinv(F,G,E)
  &=\dinv(F:G:E)-\operatorname{len}(E)\\
  &=\dinv\bigl(F:\operatorname{RR}(P)\bigr)-\sum_{v\in Q}v
   =\dinv(F,P,Q).
\end{aligned}
\]
Thus \(\phi_4\) preserves both area and \(\dinv\).
\end{proof}

\begin{theorem}[Two-column tableau formula for \(C_n(q,t)\)]
\label{thm:section4-qt-catalan-formula}
For \(n\ge 1\), we use the convention
\[
  C_n(q,t)=\sum_D q^{\area(D)}t^{\dinv(D)},
\]
where \(D\) ranges over Dyck sequences of length \(n\).  Then
\[
  C_n(q,t)
  =
  \sum_{(F,P)}
  q^{\area(F:\operatorname{RR}(P))}
  t^{\dinv(F:\operatorname{RR}(P))-|P|}
  s_{\lambda(P)'}(q,t).
\]
The sum is over all pairs \((F,P)\) such that \(F\) is a Dyck
\(m\)-skeleton for some \(m\ge 0\), \(P\) is an at-most-two-column Dyck tableau
with entries in \([0,m-1]\), and
\[
  |F|+|\operatorname{RR}(P)|=n.
\]
Here \(\lambda(P)\) is the shape of \(P\), \(\lambda(P)'\) is its conjugate
partition, and \(s_{\lambda(P)'}(q,t)\) is the Schur function in the two
variables \(q,t\).
\end{theorem}

\begin{proof}
The maps \(\phi_1,\phi_2,\phi_3\), and \(\phi_4\) give an area- and
\(\dinv\)-preserving bijection from Dyck sequences of length \(n\) to Type 4
triples \((F,P,Q)\) satisfying
\[
  |F|+|\operatorname{RR}(P)|=n.
\]
Therefore
\[
\begin{aligned}
  C_n(q,t)
  &=\sum_{(F,P,Q)}
    q^{\area(F,P,Q)}t^{\dinv(F,P,Q)}\\
  &=\sum_{(F,P,Q)}
    q^{\area(F:\operatorname{RR}(P))+\sum_{v\in Q}v}
    t^{\dinv(F:\operatorname{RR}(P))-\sum_{v\in Q}v}.
\end{aligned}
\]
Now fix \((F,P)\) and sum over all binary reverse semistandard tableaux \(Q\)
of shape \(\lambda(P)\).  Since entries of \(Q\) are \(0\) and \(1\),
\[
  \sum_{v\in Q}v=\#1(Q),
  \qquad
  \#0(Q)+\#1(Q)=|Q|=|P|.
\]
Thus
\[
  q^{\sum_{v\in Q}v}t^{-\sum_{v\in Q}v}
  =q^{\#1(Q)}t^{-\#1(Q)}
  =t^{-|P|}q^{\#1(Q)}t^{\#0(Q)}.
\]
The fixed \((F,P)\)-contribution is therefore
\[
  q^{\area(F:\operatorname{RR}(P))}
  t^{\dinv(F:\operatorname{RR}(P))-|P|}
  \sum_Q q^{\#1(Q)}t^{\#0(Q)}.
\]
It remains to identify the inner sum.  Transposing a binary reverse
semistandard tableau of shape \(\lambda(P)\) turns its strictly increasing rows
and weakly increasing columns into the ordinary semistandard condition on shape
\(\lambda(P)'\): rows are weakly increasing and columns are strictly
increasing.  With variables assigned by \(0\mapsto t\) and \(1\mapsto q\), this
gives
\[
  \sum_{\substack{Q\text{ binary reverse SSYT}\\
                  \operatorname{shape}(Q)=\lambda(P)}}
  q^{\#1(Q)}t^{\#0(Q)}
  =s_{\lambda(P)'}(t,q).
\]
Schur functions are symmetric in the variables, so
\(s_{\lambda(P)'}(t,q)=s_{\lambda(P)'}(q,t)\).  Substituting this identity into
the fixed \((F,P)\)-contribution gives the stated formula.
\end{proof}

\section{\texorpdfstring{Full Dyck skeletons and the \(q,t\)-Catalan formula}{Full Dyck skeletons and the q,t-Catalan formula}}
\label{sec:full-skeletons}

The preceding section expressed the \(q,t\)-Catalan polynomial in terms of
Dyck \(m\)-skeletons and two-column tableaux.  We introduce a related
family of skeletons that indexes the high-total-degree part of
\(C_n(q,t)\).  The maps used later in this section are analogous to the
insertion constructions of Section~\ref{sec:insertion-algorithm}, but they act on
ordinary Dyck sequences by controlled extraction, local switching, and
reinsertion operations.

Throughout this section, a Dyck sequence means an ordinary Dyck sequence in
the convention of Definition~\ref{def:affine-dyck}, and extractable elements
are selected by the leftmost convention of Definition~\ref{def:extractable-element}.
We also use the area and \(\dinv\) statistics recalled in
Section~\ref{sec:dyck-decompositions}.

\subsection{Skeletons, deficit, and the target formula}

\begin{definition}[Full Dyck skeleton]
\label{def:full-dyck-skeleton}
A Dyck sequence \(S\) is a \emph{full Dyck skeleton} if it has no extractable
element under the leftmost convention of
Definition~\ref{def:extractable-element}.
\end{definition}

\begin{remark}
\label{rem:full-versus-m-skeleton}
A full Dyck skeleton differs from a Dyck \(m\)-skeleton in that no final entry
is exempted from non-extractability. Thus a full skeleton whose final entry is
its maximum \(m\) is a Dyck \(m\)-skeleton, but a Dyck \(m\)-skeleton may still
have an extractable final entry.
\end{remark}

\begin{definition}[Special Dyck skeleton]
\label{def:special-dyck-skeleton}
For \(n\ge 4\), set
\[
  \epsilon_n=(0,0,1,\underbrace{0,\ldots,0}_{n-4\text{ entries}},1).
\]
For \(n\ge4\), a \emph{special Dyck skeleton} of length \(n\) is a full Dyck
skeleton of length \(n\) that is not equal to \(\epsilon_n\).  For \(n<4\),
every full Dyck skeleton is special.
\end{definition}

For a finite nonnegative integer word \(T=(T_0,\ldots,T_{n-1})\), define its
\emph{deficit} by
\[
  \defc(T)=\binom{n}{2}-\area(T)-\dinv(T).
\]
When \(T\) is a Dyck sequence, this is the deficit statistic used for the
corresponding term of \(C_n(q,t)\).
We also set
\[
  \Delta_n=\binom{n}{2}-(2n-8).
\]
For a Dyck sequence \(D\), since
\(\area(D)+\dinv(D)=\binom{n}{2}-\defc(D)\), the terms of total degree at least
\(\Delta_n\) are exactly the terms with deficit at most \(2n-8\).

\begin{lemma}[Arbitrary-word deficit correction]
\label{lem:extended-deficit-pairs}
Let \(T=(T_0,\ldots,T_{n-1})\) be a finite nonnegative integer word.
For \(0\le i<j<n\), call \((i,j)\) type~\textup{(A)} if \(T_i>T_j+1\), and
call it type~\textup{(B)} if \(T_i<T_j\) and \(i\) is not the first occurrence
of the value \(T_i\) in the full word \(T\).  Then
\begin{align*}
  \defc(T)
  &=
  \#\{\text{type~\textup{(A)} or type~\textup{(B)} pairs in }T\} \\
  &\quad -
  \sum_{j=0}^{n-1}
  \#\{\,v\in\{0,\ldots,T_j-1\}:
       v\text{ occurs in no position }<j\,\}.
\end{align*}
\end{lemma}

\begin{proof}
Fix \(j\), and write \(e=T_j\).  Since \(T\) is nonnegative,
\[
  e
  =
  \#\{\,v<e:v\text{ occurs to the left of }j\,\}
  +
  \#\{\,v<e:v\text{ does not occur to the left of }j\,\}.
\]
If \(v<e\) occurs to the left of \(j\), then the first occurrence of \(v\)
also occurs to the left of \(j\), and this first occurrence is unique.  Thus,
after summing over \(j\),
\[
  \area(T)
  =
  \#\{\,i<j:T_i<T_j\text{ and }i\text{ is the first occurrence of }T_i
  \text{ in }T\,\}
  + M(T),
\]
where
\[
  M(T)=
  \sum_{j=0}^{n-1}
  \#\{\,v<T_j:v\text{ does not occur to the left of }j\,\}.
\]

The statistic \(\dinv(T)\) counts exactly the pairs \(i<j\) with
\(T_i=T_j\) or \(T_i=T_j+1\).  Therefore
\(\defc(T)\) is obtained from all pairs \(i<j\) by removing
the \(\dinv\)-pairs, removing the pairs \(T_i<T_j\) whose left endpoint is the
first occurrence of its value, and then subtracting \(M(T)\).  Among the
remaining pairs with \(T_i<T_j\) are exactly the type~\textup{(B)} pairs, and
among the remaining pairs with \(T_i>T_j\) are exactly the
type~\textup{(A)} pairs.  This gives the displayed formula.
\end{proof}

\begin{proposition}[Deficit pairs]
\label{prop:deficit-pair-count}
Let \(D=(x_0,\ldots,x_{n-1})\) be a Dyck sequence. Then \(\defc(D)\) is the
number of pairs \((i,j)\) with \(0\le i<j<n\) satisfying one of the following
two conditions:
\begin{enumerate}[label=\textup{(\Alph*)}]
  \item \(x_i>x_j+1\);
  \item \(x_i<x_j\) and \(i\) is not the first occurrence of the value \(x_i\)
  in \(D\).
\end{enumerate}
\end{proposition}

\begin{proof}
This is the Dyck-sequence specialization of
Lemma~\ref{lem:extended-deficit-pairs}.  It remains only to check that the
missing-value correction in that lemma vanishes.  Fix \(j\), and let \(v\) be
an integer with \(0\le v<x_j\). Since \(x_0=0\) and \(x_j>v\), let \(t\le j\)
be the first index with \(x_t>v\). Then \(t>0\), and by minimality
\(x_{t-1}\le v\). The Dyck step condition gives
\[
  x_t\le x_{t-1}+1\le v+1.
\]
Together with \(x_t>v\), this forces \(x_t=v+1\) and \(x_{t-1}=v\). Hence
every value below \(x_j\) occurs to the left of \(j\), so the correction term
is zero.
\end{proof}

Theorem~\ref{thm:qt-catalan-skeleton} then gives the corresponding
special-skeleton expansion for the terms of \(C_n(q,t)\) with deficit at most
\(2n-8\), equivalently with total degree at least \(\Delta_n\).


\subsection{The East map}
\label{subsec:east-map}

We next isolate the local seven-entry map used by the global
\(\mathrm{up}\) and \(\mathrm{down}\) operations.  The local map moves the
boundary between an affine block and a reverse block one step to the right.
Before defining it, we record the small amount of far-apart combinatorics used
in the case analysis.

\begin{definition}[Far-apart decomposable seven-multiset]
\label{def:far-apart-decomposable}
Let \(S\) be a multiset of seven integers. We say that \(S\) is
\emph{far-apart decomposable} if its seven selected occurrences can be
partitioned into three unordered pairs and one singleton so that, for each pair
\(\{a,b\}\), one has \(|a-b|\ge 2\).
\end{definition}

We write \(u\gg v\) to mean \(u\ge v+2\).

\begin{lemma}[A far-apart criterion]
\label{lem:far-apart-separated-values}
Let \(S\) be a multiset of seven integers. Suppose that \(S\) contains five
distinct selected occurrences \(a,b,c,d,e\) such that
\[
  a\gg c\gg e
  \qquad\text{and}\qquad
  b\gg d.
\]
Then \(S\) is far-apart decomposable.
\end{lemma}

\begin{proof}
Choose a sixth selected occurrence and call it \(f\); the remaining seventh
occurrence will be the singleton.  If \(f\ge c\), then
\[
  \{a,c\},\qquad \{f,e\},\qquad \{b,d\}
\]
are three far-apart pairs.  Indeed, \(a\gg c\), \(b\gg d\), and
\(f\ge c\gg e\).  If instead \(f<c\), then integrality gives
\(f\le c-1\).  Since \(a\ge c+2\), we have \(a\ge f+3\), so
\(\{a,f\}\) is far apart.  In this case
\[
  \{a,f\},\qquad \{c,e\},\qquad \{b,d\}
\]
are three far-apart pairs.  In either case these pairs, together with the
seventh occurrence as singleton, form a far-apart decomposition of \(S\).
\end{proof}

\begin{corollary}[Five distinct values]
\label{cor:five-distinct-far-apart}
Let \(S\) be a multiset of seven integers. If \(S\) contains five selected
occurrences with five distinct values
\[
  a>b>c>d>e,
\]
then \(S\) is far-apart decomposable.
\end{corollary}

\begin{proof}
Because the entries are integers, the inequalities imply
\(a\ge c+2\), \(c\ge e+2\), and \(b\ge d+2\).  Thus
\(a\gg c\gg e\) and \(b\gg d\), so the result follows from
Lemma~\ref{lem:far-apart-separated-values}.
\end{proof}

\begin{definition}[The sets \(L(S,k)\) and \(R(S,k)\)]
\label{def:LR-windows}
Let \(S\) be a multiset of seven integers and let \(k\ge 0\).  Define
\(L(S,k)\) to be the set of seven-entry sequences
\[
  (x_{-3},x_{-2},x_{-1},x_0,x_1,x_2,x_3)
\]
whose entries are the multiset \(S\), whose \(\di\)-value is \(k\), and such
that
\[
  (x_{-3},x_{-2},x_{-1},x_0)
  \quad\text{is affine Dyck,}
  \qquad
  (x_1,x_2,x_3)
  \quad\text{is reverse Dyck.}
\]
Define \(R(S,k)\) to be the set of seven-entry sequences with the same multiset
and \(\di\)-value conditions such that
\[
  (x_{-3},x_{-2},x_{-1})
  \quad\text{is affine Dyck,}
  \qquad
  (x_0,x_1,x_2,x_3)
  \quad\text{is reverse Dyck.}
\]
\end{definition}

\begin{proposition}[Local count symmetry]
\label{prop:LR-cardinality}
For every seven-element integer multiset \(S\) and every \(k\ge 0\),
\[
  |L(S,k)|=|R(S,k)|.
\]
\end{proposition}

\begin{proof}
Translation of all seven entries by a common integer preserves the affine and
reverse inequalities and preserves \(\di\).  It also identifies \(L(S,k)\) and
\(R(S,k)\) with the corresponding sets for the translated multiset.  Thus, when
we appeal to the symmetric function \(\DS\), we may first translate \(S\) so
that all entries are nonnegative and then translate back.

We next convert the reverse block in each window into an affine block.  For an
arbitrary finite integer block, the affine--reverse transport of
Proposition~\ref{prop:affine-reverse-transport} is applied after translating
all entries of the block by a common constant so that the block lies in an
interval \([0,M]\), and then translating back.  Again, this translation preserves
affine and reverse inequalities and also preserves \(\di\).

For an element of \(L(S,k)\), leave the first block
\(A=(x_{-3},x_{-2},x_{-1},x_0)\) fixed and apply this transport to the reverse
block \(B=(x_1,x_2,x_3)\).  This gives an affine two-factor factorization of
\(S\) with factor lengths \((4,3)\).  The total \(\di\) is preserved: the
internal \(\di\) of \(A\) is unchanged, the internal \(\di\) of \(B\) is
preserved by the transport, and the cross-block contribution depends only on
the two block multisets, not on the order within the right block.

Thus \(L(S,k)\) is in \(\di\)-preserving bijection with the set of affine Dyck
factorizations of \(S\) with two displayed factors of lengths \((4,3)\) and
\(\di=k\).  Similarly, \(R(S,k)\) is in \(\di\)-preserving bijection with the
set of affine Dyck factorizations of \(S\) with two displayed factors of
lengths \((3,4)\) and \(\di=k\).

Corollary~\ref{cor:affine-dyck-schur-positivity} applies to the integer
multiset \(S\) and says that the affine Dyck symmetric function
\(\DS(S,k;\mathbf{x})\) is symmetric.  With the paper's zero-indexed
symmetric-function variables, the coefficient of \(x_0^4x_1^3\) in
\(\DS(S,k;\mathbf{x})\) counts the affine factorizations of lengths \((4,3)\), and the
coefficient of \(x_0^3x_1^4\) counts the affine factorizations of lengths
\((3,4)\).  Symmetry makes these two coefficients equal.  Composing with the
two reverse-to-affine conversions proves the proposition.
\end{proof}

For the remainder of this subsection, \(\mathrm{fw}\) and \(\mathrm{bk}\)
denote this local affine--reverse transport, with \(\mathrm{fw}\) carrying a
reverse Dyck ordering to an affine Dyck ordering of the same multiset and
\(\mathrm{bk}\) its inverse.  For integer-valued blocks, this means translate
the block by a common constant into an interval \([0,M]\), apply the transport
there, and translate back.  The result is independent of the chosen common
translation and ambient interval, since the deletion words \(B_k\) and the
binary base in Proposition~\ref{prop:affine-reverse-transport} are unchanged
by common translation and depend only on the relative order of equal and
adjacent values.  All windows below are therefore treated in their absolute
integer coordinates, not re-based
separately after choosing the window.  On a two-entry affine window \((a,b)\),
this means
\[
  \mathrm{bk}(a,b)=
  \begin{cases}
    (b,a), & a>b+1,\\
    (a,b), & a\le b+1.
  \end{cases}
\]
For a two-entry reverse window \((a,b)\), equivalently,
\[
  \mathrm{fw}(a,b)=
  \begin{cases}
    (b,a), & b>a+1,\\
    (a,b), & b\le a+1.
  \end{cases}
\]

\begin{definition}[The East map]
\label{def:east}
Let \(S\) be a seven-element integer multiset that is not far-apart
decomposable, and let
\[
  x=(x_{-3},x_{-2},x_{-1},x_0,x_1,x_2,x_3)\in L(S,k).
\]
The map \(\mathrm{East}\) is defined by the following cases, checked in order.
In every case, the first and last entries \(x_{-3}\) and \(x_3\) are fixed.

\smallskip
\noindent\textbf{Case 1.} If \(x_0\le x_1+1\), set \(\mathrm{East}(x)=x\).

\smallskip
\noindent\textbf{Case 2.} Suppose now that \(x_0>x_1+1\), and write
\[
  (y_{-1},y_0)=\mathrm{bk}(x_{-1},x_0).
\]

\smallskip
\noindent\textbf{Case 2a.} If \(x_{-1}>x_1+1\) and \(y_0\le x_2+1\), set
\[
  \mathrm{East}(x)=
  (x_{-3},x_{-2},x_1,y_{-1},y_0,x_2,x_3).
\]

\smallskip
\noindent\textbf{Case 2b.} If Case 2a does not apply, and if
\(x_{-1}\le x_1+1\) and \(x_{-1}\le x_2+1\), set
\[
  \mathrm{East}(x)=
  (x_{-3},x_{-2},x_1,x_0,x_{-1},x_2,x_3).
\]
In this branch the standing conditions imply
\(x_0=x_{-1}+1=x_1+2\).

\smallskip
\noindent\textbf{Case 3.} If neither Case 2a nor Case 2b applies and
\[
  \min(x_{-2},x_{-1},x_0)>\max(x_1,x_2)+1,
\]
then set
\[
  (u_{-2},u_{-1})=\mathrm{fw}(x_1,x_2),
  \qquad
  (u_0,u_1,u_2)=\mathrm{bk}(x_{-2},x_{-1},x_0),
\]
and define
\[
  \mathrm{East}(x)=
  (x_{-3},u_{-2},u_{-1},u_0,u_1,u_2,x_3).
\]

\smallskip
\noindent\textbf{Case 4.} The remaining case is
\[
  \min(x_{-2},x_{-1},x_0)\le \max(x_1,x_2)+1.
\]
Translate all seven entries by the same constant so that
\(\max(x_1,x_2)=2\), apply the appropriate table below to the middle five
entries, and translate back.  Proposition~\ref{prop:east-case4-exhaustive}
below proves that, under the Case 4 hypotheses together with the
non-far-apart hypothesis, the normalized middle five entries must appear as
one of the listed input rows.  Since
\((x_1,x_2)\) is reverse Dyck, after this normalization one of the following
four alternatives holds.

If the normalized ordered pair \((x_1,x_2)\) is \((1,2)\), use
\[
\begin{array}{rcl}
{[3,3,4,1,2]} &\mapsto& {[1,2,4,3,3]},\\[2pt]
{[3,4,4,1,2]} &\mapsto& {[1,2,4,3,4]},\\[2pt]
{[4,3,4,1,2]} &\mapsto& {[1,2,4,4,3]},\\[2pt]
{[2,3,4,1,2]} &\mapsto& {[1,2,4,3,2]}.
\end{array}
\]
If the normalized ordered pair \((x_1,x_2)\) is \((2,1)\), use
\[
\begin{array}{rcl}
{[3,3,4,2,1]} &\mapsto& {[2,1,4,3,3]},\\[2pt]
{[3,4,4,2,1]} &\mapsto& {[2,1,4,3,4]},\\[2pt]
{[4,3,4,2,1]} &\mapsto& {[2,1,4,4,3]},\\[2pt]
{[2,3,4,2,1]} &\mapsto& {[2,1,4,3,2]}.
\end{array}
\]
If the normalized ordered pair \((x_1,x_2)\) is \((2,2)\), use
\[
\begin{array}{rcl}
{[3,4,4,2,2]} &\mapsto& {[2,2,4,4,3]},\\[2pt]
{[3,4,5,2,2]} &\mapsto& {[2,2,5,4,3]}.
\end{array}
\]
Finally, if the normalized ordered pair \((x_1,x_2)\) is \((o,2)\) for some
\(o\le 0\), use
\[
\begin{array}{rcl}
{[3,3,4,o,2]} &\mapsto& {[2,o,4,3,3]},\\[2pt]
{[3,4,4,o,2]} &\mapsto& {[2,o,4,3,4]},\\[2pt]
{[4,3,4,o,2]} &\mapsto& {[2,o,4,4,3]},\\[2pt]
{[2,3,4,o,2]} &\mapsto& {[2,o,2,4,3]},\\[2pt]
{[3,4,2,o,2]} &\mapsto& {[2,o,4,3,2]}.
\end{array}
\]
The result after translating back is \(\mathrm{East}(x)\).
\end{definition}

\begin{lemma}[The min--max reduction]
\label{lem:minmax}
In the setup of Definition~\ref{def:east}, suppose that Cases 1, 2a, and 2b
all fail.  Then
\[
  \min(x_{-2},x_{-1},x_0)\ge \max(x_1,x_2).
\]
\end{lemma}

\begin{proof}
Set \(M=\max(x_1,x_2)\).  Failure of Case 1 gives
\(x_0>x_1+1\), hence \(x_0\ge x_1+2\).  Since the left block
\((x_{-3},x_{-2},x_{-1},x_0)\) is affine Dyck, we also have
\[
  x_{-1}\ge x_0-1,
  \qquad
  x_{-2}\ge x_{-1}-1.
\]

First assume \(x_{-1}>x_1+1\).  Since Case 2a fails, if
\((y_{-1},y_0)=\mathrm{bk}(x_{-1},x_0)\), then \(y_0>x_2+1\).
There are two possible two-entry \(\mathrm{bk}\) outputs.

If \(\mathrm{bk}(x_{-1},x_0)=(x_{-1},x_0)\), then
\(x_0\ge x_2+2\).  Together with \(x_0\ge x_1+2\), this gives
\(x_0\ge M+2\).  Also
\[
  x_{-1}\ge x_1+2,
  \qquad
  x_{-1}\ge x_0-1\ge x_2+1,
\]
so \(x_{-1}\ge M\); applying \(x_{-2}\ge x_{-1}-1\) gives
\(x_{-2}\ge x_1+1\) and \(x_{-2}\ge x_2\), hence \(x_{-2}\ge M\).
Thus the desired inequality holds in this subcase.

If \(\mathrm{bk}(x_{-1},x_0)=(x_0,x_{-1})\), then
\(x_{-1}\ge x_2+2\), and with \(x_{-1}\ge x_1+2\) we obtain
\(x_{-1}\ge M+2\).  Hence
\(x_{-2}\ge x_{-1}-1\ge M\).  It remains to prove \(x_0\ge M\).  We already
know \(x_0\ge x_1+2\).  If, for contradiction, \(x_0<x_2\), then integrality
gives \(x_2\ge x_0+1\), hence \(x_2\ge x_1+3\).  The three disjoint pairs
\[
  \{x_{-1},x_2\},
  \qquad
  \{x_{-2},x_0\},
  \qquad
  \{x_{-3},x_1\}
\]
are then far apart: the first because \(x_{-1}\ge x_2+2\); the second because
\(x_{-2}\ge x_2+1\ge x_0+2\); and the third because
\(x_{-3}\ge x_{-2}-1\ge x_2\ge x_1+3\).  This would make \(S\) far-apart
decomposable, with \(x_3\) as the singleton, contradicting the hypothesis of
Definition~\ref{def:east}.  Therefore \(x_0\ge x_2\), and so \(x_0\ge M\).
This proves the desired inequality in the second subcase.

It remains to consider the complementary case \(x_{-1}\le x_1+1\).  Combining
this with \(x_0>x_1+1\) and the affine inequality \(x_0\le x_{-1}+1\) gives
\[
  x_1+1<x_0\le x_{-1}+1\le x_1+2.
\]
Since the entries are integers, \(x_0=x_1+2\) and
\(x_{-1}=x_1+1\).  Thus the equality part of Case 2b holds.  Since Case 2b
fails, its other condition fails, so \(x_{-1}>x_2+1\).  Hence \(x_1>x_2\) and
\(M=x_1\).  Finally,
\[
  x_{-2}\ge x_{-1}-1=x_1=M,
  \qquad
  x_{-1}=M+1,
  \qquad
  x_0=M+2,
\]
which again gives the desired inequality.
\end{proof}

The next subsection proves that the displayed cases of
Definition~\ref{def:east} are exhaustive on the non-far-apart domain, that the
map lands in \(R(S,k)\), and that it is inverted by the reversal-conjugate West
map, defined there.

\subsection{Well-definedness and inversion for East and West}
\label{subsec:east-west-inverse}

We now prove that the local case list defining \(\mathrm{East}\) is
complete on the non-far-apart domain, that the resulting window lies in the
set \(R(S,k)\), and that the reversal-conjugate map is inverse to it.

\begin{proposition}[Exhaustiveness of the Case 4 table]
\label{prop:east-case4-exhaustive}
Let \(S\) be a seven-element multiset that is not far-apart decomposable, and
let
\[
  x=(x_{-3},x_{-2},x_{-1},x_0,x_1,x_2,x_3)\in L(S,k).
\]
Suppose that the procedure of Definition~\ref{def:east} reaches Case~4. After
translating all entries by a common integer so that
\(\max(x_1,x_2)=2\), the middle five entries
\[
  [x_{-2},x_{-1},x_0,x_1,x_2]
\]
are one of the rows displayed in the Case~4 table of
Definition~\ref{def:east}.
\end{proposition}

\begin{proof}
Set
\[
  A=x_{-2},\qquad B=x_{-1},\qquad C=x_0,
  \qquad (p,q)=(x_1,x_2).
\]
Since \((x_{-3},A,B,C)\) is affine Dyck and \((p,q,x_3)\) is reverse Dyck, we
have
\[
  B\le A+1,
  \qquad C\le B+1,
  \qquad p\le q+1.
\]
After the translation \(\max(p,q)=2\), the reverse Dyck inequality for
\((p,q)\) leaves exactly four possibilities:
\[
  (p,q)=(1,2),\qquad (p,q)=(2,1),\qquad (p,q)=(2,2),
  \qquad (p,q)=(o,2)\text{ with }o\le 0.
\]
Because Cases~1, 2a, and 2b have failed, Lemma~\ref{lem:minmax} gives
\[
  \min(A,B,C)\ge \max(p,q)=2.
\]
Because Case~3 has failed, its strict inequality is false, so
\[
  \min(A,B,C)\le \max(p,q)+1=3.
\]
Thus
\[
  \min(A,B,C)\in\{2,3\}.
\]

First suppose \((p,q)=(1,2)\). Failure of Case~1 gives \(C>p+1=2\), so
\(C\ge 3\). If \(C=3\) and \(B\le 2\), then Case~2b applies. If \(C=3\) and
\(3\le B\le 4\), then the second entry of \(\mathrm{bk}(B,3)\) is \(3\), so
Case~2a applies. Hence any surviving case with \(C=3\) has \(B\ge 5\), and
then \(A\ge B-1\ge 4\). The five selected entries
\[
  B,\quad C=3,\quad p=1,\quad A,\quad q=2
\]
satisfy \(B\gg C\gg p\) and \(A\gg q\), so
Lemma~\ref{lem:far-apart-separated-values} makes \(S\) far-apart
decomposable, a contradiction. Therefore \(C\ne 3\). If \(C\ge 5\), then the
affine inequalities force \((A,B,C)=(3,4,5)\), since otherwise
\(\min(A,B,C)>3\). But then five distinct values \(1,2,3,4,5\) occur, contrary
to Corollary~\ref{cor:five-distinct-far-apart}. Hence \(C=4\). Now
\(C\le B+1\) gives \(B\ge 3\), while \(B\ge 5\) would force
\(\min(A,B,C)\ge 4\). Thus \(B\in\{3,4\}\). If \(B=3\), then
\(A\ge 2\), and \(A\ge 5\) would again give five distinct values; hence
\(A\in\{2,3,4\}\). If \(B=4\), then \(A\ge 3\), and \(A\ge 4\) would force
\(\min(A,B,C)\ge 4\); hence \(A=3\). This gives exactly the four rows
\[
  [3,3,4,1,2],\quad [3,4,4,1,2],\quad [4,3,4,1,2],\quad [2,3,4,1,2].
\]

The case \((p,q)=(2,1)\) is similar but slightly shorter. Failure of Case~1
now gives \(C>3\), so \(C\ge 4\). If \(C\ge 5\), the affine inequalities and
\(\min(A,B,C)\le 3\) force \((A,B,C)=(3,4,5)\), producing five distinct values
\(1,2,3,4,5\), a contradiction. Thus \(C=4\). Then \(B\ge 3\), and
\(B\ge 5\) would force \(\min(A,B,C)\ge 4\), so \(B\in\{3,4\}\). If
\(B=3\), then \(A\in\{2,3,4\}\) by the same five-distinct-values exclusion;
if \(B=4\), then \(A=3\). These are exactly
\[
  [3,3,4,2,1],\quad [3,4,4,2,1],\quad [4,3,4,2,1],\quad [2,3,4,2,1].
\]

If \((p,q)=(2,2)\), failure of Case~1 gives \(C\ge 4\). If \(B\le 3\), then
\(C\le B+1\) forces \((B,C)=(3,4)\), and Case~2b applies because
\(B\le p+1=q+1=3\). Hence \(B\ge 4\). It follows that \(A\ge B-1\ge 3\), and
\(A\ge 4\) would force \(\min(A,B,C)\ge 4\). Thus \(A=3\), whence
\(B=4\) and \(C\in\{4,5\}\). The only rows are
\[
  [3,4,4,2,2],\qquad [3,4,5,2,2].
\]

It remains to treat \((p,q)=(o,2)\) with \(o\le 0\). Since \(A,B,C\ge 2\), we
have \(B>o+1\); hence failure of Case~2a means that the second entry of
\(\mathrm{bk}(B,C)\) is greater than \(3\).

Assume first that none of \(A,B,C\) is equal to \(2\). Then
\(\min(A,B,C)=3\). If \(C=3\), then \(B\le 4\) would make Case~2a apply, so a
surviving case has \(B\ge 5\) and hence \(A\ge 4\). The five selected entries
\[
  B,
  \qquad C=3,
  \qquad o,
  \qquad A,
  \qquad 2
\]
satisfy \(B\gg C\gg o\) and \(A\gg 2\), contradicting
Lemma~\ref{lem:far-apart-separated-values}. Thus \(C\ne 3\). If \(C\ge 5\),
then the affine inequalities and \(\min(A,B,C)=3\) force
\((A,B,C)=(3,4,5)\), giving the five distinct values \(o,2,3,4,5\), a
contradiction. Therefore \(C=4\). Then \(B\in\{3,4\}\). If \(B=3\), the
five-distinct-values exclusion gives \(A\in\{3,4\}\); if \(B=4\), then
\(A=3\). This gives the rows
\[
  [3,3,4,o,2],
  \qquad [4,3,4,o,2],
  \qquad [3,4,4,o,2].
\]

Now suppose that one of \(A,B,C\) is equal to \(2\). If \(B=2\), then
\(C\le 3\), and the second entry of \(\mathrm{bk}(2,C)\) is at most \(3\), so
Case~2a applies; hence \(B\ne 2\). If \(A=2\), then \(B\le 3\), and since
\(B\ne 2\) we have \(B=3\). Case~2a excludes \(C\le 3\), while
\(C\le B+1=4\), so \(C=4\), giving \([2,3,4,o,2]\). Finally suppose
\(C=2\), with \(A\ne 2\) and \(B\ne 2\). If \(B\le 3\), then Case~2a applies;
if \(B\ge 5\), then \(A\ge 4\), and the selected entries
\[
  B,
  \qquad C=2,
  \qquad o,
  \qquad A,
  \qquad q=2
\]
where the two occurrences of \(2\) are distinct, satisfy the hypotheses of
Lemma~\ref{lem:far-apart-separated-values}. Hence \(B=4\). If \(A\ge 4\), then
\(x_{-3}\ge A-1\ge 3\), and the three pairs
\[
  \{A,C\},
  \qquad \{B,q\},
  \qquad \{x_{-3},o\}
\]
are all far apart, again a contradiction. Thus \(A=3\), giving
\([3,4,2,o,2]\).

The four normalized possibilities for \((p,q)\) are exhaustive, and each gives
only the displayed rows.
\end{proof}

\begin{proposition}[East is well-defined]
\label{prop:east-well-defined}
Let \(S\) be a seven-element multiset that is not far-apart decomposable, let
\(k\ge 0\), and let \(x\in L(S,k)\). Then \(\mathrm{East}(x)\) is defined and
lies in \(R(S,k)\).
\end{proposition}

\begin{proof}
Write
\[
  x=(x_{-3},x_{-2},x_{-1},x_0,x_1,x_2,x_3),
  \qquad
  y=\mathrm{East}(x)=(y_{-3},\ldots,y_3).
\]
Every case fixes the two outer entries and only reorders the middle five
entries, possibly after applying \(\mathrm{fw}\) or \(\mathrm{bk}\) to an
affine or reverse subblock. These auxiliary maps preserve the underlying
multiset and the internal \(\di\)-statistic of the subblock. The Case~4 rows
are explicit permutations of the middle five entries. Thus the multiset is
preserved in all cases; it remains to check the mixed affine/reverse structure
and the total \(\di\)-value.

In Case~1 the map is the identity. The first output block is the affine prefix
\((x_{-3},x_{-2},x_{-1})\), and the last output block
\((x_0,x_1,x_2,x_3)\) is reverse Dyck because the Case~1 condition is
\(x_0\le x_1+1\) and \((x_1,x_2,x_3)\) was already reverse Dyck. The statistic
\(\di\) is unchanged.

In Case~2a, write
\[
  (y_{-1},y_0)=\mathrm{bk}(x_{-1},x_0).
\]
The output is
\[
  (x_{-3},x_{-2},x_1,y_{-1},y_0,x_2,x_3).
\]
The pair \((x_{-1},x_0)\) is affine and is sent by \(\mathrm{bk}\) to a
reverse pair with the same internal \(\di\). The element \(x_1\) moves from the
right of the two values \(x_{-1},x_0\) to their left. Since Case~1 failed and
Case~2a holds,
\[
  x_0>x_1+1,
  \qquad
  x_{-1}>x_1+1,
\]
so both crossed values differ from \(x_1\) by at least \(2\). No \(\di\)-pair
is created or destroyed by this move, and \(x_1\) does not cross
\(x_{-2}\). Contributions involving the two outer sides of the middle block are
unchanged because those entries remain on the same side of the same middle
multiset.

The first output block is affine: \(x_{-2}\le x_{-3}+1\) is inherited, and
\[
  x_{-2}\ge x_{-1}-1\ge x_1+1
\]
gives \(x_1\le x_{-2}+1\). The last output block is reverse Dyck because
\((y_{-1},y_0)\) is reverse, the Case~2a success condition gives
\(y_0\le x_2+1\), and \((x_1,x_2,x_3)\) was reverse Dyck.

In Case~2b, the output is
\[
  (x_{-3},x_{-2},x_1,x_0,x_{-1},x_2,x_3).
\]
Combining failure of Case~1 with the affine inequality and the first Case~2b
inequality gives
\[
  x_1+1<x_0\le x_{-1}+1\le x_1+2.
\]
By integrality,
\[
  x_0=x_1+2,
  \qquad
  x_{-1}=x_1+1.
\]
Only the local triple \((x_{-1},x_0,x_1)\) changes order. The change
\((x_{-1},x_0)\mapsto(x_0,x_{-1})\) creates one \(\di\)-pair, the change
\((x_{-1},x_1)\mapsto(x_1,x_{-1})\) destroys one \(\di\)-pair, and the change
\((x_0,x_1)\mapsto(x_1,x_0)\) changes no \(\di\)-pair because
\(x_0=x_1+2\). Thus the local net change in \(\di\) is zero; all outside
entries remain on the same side of the same local multiset.

The first output block is affine because \(x_{-2}\le x_{-3}+1\) is inherited
and \(x_1\le x_{-2}\) follows from \(x_{-1}=x_1+1\) and
\(x_{-1}\le x_{-2}+1\). The last output block is reverse Dyck because
\(x_0=x_{-1}+1\), the Case~2b success condition gives
\(x_{-1}\le x_2+1\), and \((x_1,x_2,x_3)\) was reverse Dyck.

In Case~3, set
\[
  L=\{x_1,x_2\},
  \qquad
  H=\{x_{-2},x_{-1},x_0\}
\]
as multisets. The Case~3 hypothesis says every element of \(H\) is at least
\(2\) larger than every element of \(L\). The map applies \(\mathrm{fw}\) to
the reverse pair \((x_1,x_2)\) and \(\mathrm{bk}\) to the affine triple
\((x_{-2},x_{-1},x_0)\). The internal \(\di\)-contributions in these two
blocks are preserved. There are no \(\di\)-pairs between \(H\) and \(L\), either
before or after the move, by the strict separation. The boundary entries
\(x_{-3}\) and \(x_3\) remain outside the middle multiset, so their
contributions are unchanged.

The first two middle output entries form an affine pair by construction. If
\((y_{-2},y_{-1})=\mathrm{fw}(x_1,x_2)\), then
\(y_{-2}\le \max L\). By Lemma~\ref{lem:minmax}, \(\max L\le x_{-2}\), and the
input affine condition gives \(x_{-2}\le x_{-3}+1\). Hence
\(y_{-2}\le x_{-3}+1\), so the first output block is affine. The last three
middle output entries form a reverse triple by construction. If the final
reverse boundary failed, then \(y_2\ge x_3+2\). Pair \(y_2\), which is one of
the high values, with \(x_3\), and pair the two remaining high values with the
two low values. These are three disjoint far-apart pairs, with \(x_{-3}\) as
singleton, contradicting the hypothesis that \(S\) is not far-apart
decomposable. Thus \(y_2\le x_3+1\), so the last output block is reverse Dyck.

Finally suppose Case~4 is reached. By
Proposition~\ref{prop:east-case4-exhaustive}, after translating so that
\(\max(x_1,x_2)=2\), the middle five entries are one of the displayed table
rows. Translation preserves all relevant inequalities, far-apart
decomposability, and \(\di\). The boundary entries remain fixed outside the
middle five entries, so their \(\di\)-contributions with the middle entries
depend only on the middle multiset. It is therefore enough to check the
normalized middle-five table. The following table records the common
middle-five \(\di\)-value and, for later use in the boundary argument, two
disjoint far-apart pairs among the first four output middle entries:
\[
\begin{array}{c|c|c|c|c}
\text{case} & \text{input} & \text{output} & \di_{\rm in}=\di_{\rm out} &
\text{disjoint far-apart pairs}\\ \hline
4a &[3,3,4,1,2]&[1,2,4,3,3]&2&\{1,3\},\{2,4\}\\
4a &[3,4,4,1,2]&[1,2,4,3,4]&1&\{1,3\},\{2,4\}\\
4a &[4,3,4,1,2]&[1,2,4,4,3]&2&\{1,4\},\{2,4\}\\
4a &[2,3,4,1,2]&[1,2,4,3,2]&2&\{1,3\},\{2,4\}\\
4b &[3,3,4,2,1]&[2,1,4,3,3]&3&\{2,4\},\{1,3\}\\
4b &[3,4,4,2,1]&[2,1,4,3,4]&2&\{2,4\},\{1,3\}\\
4b &[4,3,4,2,1]&[2,1,4,4,3]&3&\{2,4\},\{1,4\}\\
4b &[2,3,4,2,1]&[2,1,4,3,2]&3&\{2,4\},\{1,3\}\\
4c &[3,4,4,2,2]&[2,2,4,4,3]&2&\{2,4\},\{2,4\}\\
4c &[3,4,5,2,2]&[2,2,5,4,3]&2&\{2,5\},\{2,4\}\\
4d &[3,3,4,o,2]&[2,o,4,3,3]&2&\{2,4\},\{o,3\}\\
4d &[3,4,4,o,2]&[2,o,4,3,4]&1&\{2,4\},\{o,3\}\\
4d &[4,3,4,o,2]&[2,o,4,4,3]&2&\{2,4\},\{o,4\}\\
4d &[2,3,4,o,2]&[2,o,2,4,3]&1&\{2,4\},\{o,2\}\\
4d &[3,4,2,o,2]&[2,o,4,3,2]&2&\{2,4\},\{o,3\}
\end{array}
\]
In the Case~4d rows, \(o\le 0\), so the displayed pairs involving
\(o\) are far apart. The table proves middle-five \(\di\)-preservation in
Case~4, hence whole-window \(\di\)-preservation.

The first two middle output entries are affine in every row, and the first of
them is at most \(2\). Since Lemma~\ref{lem:minmax} gives
\(\min(x_{-2},x_{-1},x_0)\ge 2\), the input affine inequality implies
\(x_{-3}\ge 1\) in the normalized coordinates. Thus the first middle output
entry is at most \(x_{-3}+1\), so the first output block is affine. The last
three middle output entries are reverse in every row of the table. If the last
reverse boundary failed, so that \(y_2\ge x_3+2\), then \(\{y_2,x_3\}\)
would be a far-apart pair. The last column gives two more disjoint
far-apart pairs using only the first four middle output entries. Together with
\(x_{-3}\) as singleton, these pairs would make \(S\) far-apart decomposable,
a contradiction. Hence \(y_2\le x_3+1\), and the output lies in \(R(S,k)\).

All cases preserve the multiset and \(\di\) and produce the required mixed
window. Therefore \(\mathrm{East}(x)\in R(S,k)\).
\end{proof}

\begin{definition}[The West map]
\label{def:west}
Let \(\rho\) denote ordinary reversal of a seven-entry window:
\[
  \rho(x_{-3},x_{-2},x_{-1},x_0,x_1,x_2,x_3)
  =(x_3,x_2,x_1,x_0,x_{-1},x_{-2},x_{-3}).
\]
For a window in the corresponding reversed local domain, define
\[
  \mathrm{West}=\rho\circ \mathrm{East}\circ \rho.
\]
Thus \(\mathrm{West}\) is the reversal-conjugate of \(\mathrm{East}\).
\end{definition}

We shall also use the same symbol \(\rho\) for reversal of the shorter blocks
appearing in the definitions of \(\mathrm{fw}\) and \(\mathrm{bk}\). On these
blocks,
\[
  \mathrm{fw}=\rho\circ\mathrm{bk}\circ\rho,
  \qquad
  \mathrm{bk}=\rho\circ\mathrm{fw}\circ\rho.
\]

\begin{theorem}[East and West are inverse local maps]
\label{thm:east-west-inverse}
Let \(S\) be a seven-element multiset that is not far-apart decomposable, and
let \(k\ge 0\). For every \(x\in L(S,k)\),
\[
  \mathrm{West}(\mathrm{East}(x))=x.
\]
For every \(y\in R(S,k)\),
\[
  \mathrm{East}(\mathrm{West}(y))=y.
\]
\end{theorem}

\begin{proof}
We first prove the identity \(\mathrm{West}(\mathrm{East}(x))=x\). Write
\[
  x=(x_{-3},x_{-2},x_{-1},x_0,x_1,x_2,x_3),
  \qquad
  y=\mathrm{East}(x).
\]
By Proposition~\ref{prop:east-well-defined}, \(y\in R(S,k)\), so
\(\mathrm{West}(y)\) is defined. Since \(\mathrm{West}=\rho\mathrm{East}\rho\),
it is enough to show
\[
  \mathrm{East}(\rho(y))=\rho(x).
\]
We check this according to the East case used on \(x\).

In Case~1, \(y=x\). In the reversed window \(\rho(y)=\rho(x)\), the entry
playing the role of the central \(x_0\) is the original \(x_0\), and the entry
playing the role of \(x_1\) is the original \(x_{-1}\). Since the original
left block is affine Dyck, \(x_0\le x_{-1}+1\). Thus Case~1 applies to
\(\rho(y)\), and \(\mathrm{East}(\rho(y))=\rho(x)\).

In Case~2a, let
\[
  (a,b)=\mathrm{bk}(x_{-1},x_0).
\]
Then
\[
  y=(x_{-3},x_{-2},x_1,a,b,x_2,x_3),
\]
so
\[
  \rho(y)=(x_3,x_2,b,a,x_1,x_{-2},x_{-3}).
\]
Because Case~1 failed and Case~2a held on \(x\), both \(x_{-1}\) and \(x_0\)
are greater than \(x_1+1\). The entries \(a,b\) are these two values in some
order, so Case~1 fails on \(\rho(y)\) and the strict inequality needed for
Case~2a holds. Since \(\mathrm{fw}\) is inverse to \(\mathrm{bk}\), and using
\(\mathrm{bk}=\rho\mathrm{fw}\rho\),
\[
  \mathrm{bk}(b,a)=\rho(\mathrm{fw}(a,b))=(x_0,x_{-1}).
\]
The remaining Case~2a condition is \(x_{-1}\le x_{-2}+1\), which is inherited
from the original affine block. Therefore East applies to \(\rho(y)\) by
Case~2a and gives
\[
  (x_3,x_2,x_1,x_0,x_{-1},x_{-2},x_{-3})=\rho(x).
\]

In Case~2b, integrality gives
\[
  x_{-1}=x_1+1,
  \qquad
  x_0=x_1+2.
\]
The output is
\[
  y=(x_{-3},x_{-2},x_1,x_0,x_{-1},x_2,x_3),
\]
so
\[
  \rho(y)=(x_3,x_2,x_{-1},x_0,x_1,x_{-2},x_{-3}).
\]
In this reversed window, Case~1 fails because \(x_0=x_1+2\), and Case~2a
fails because \(x_{-1}=x_1+1\) is not strictly larger than \(x_1+1\). The
Case~2b inequalities hold, the last one being \(x_{-1}\le x_{-2}+1\) from the
original affine block. Hence Case~2b sends \(\rho(y)\) to
\[
  (x_3,x_2,x_1,x_0,x_{-1},x_{-2},x_{-3})=\rho(x).
\]
Equivalently, up to translation the local pattern \((a+1,a+2,a)\) is sent to
\((a,a+2,a+1)\), and reversal applies the same pattern in reverse.

In Case~3, write
\[
  (u,v)=\mathrm{fw}(x_1,x_2),
  \qquad
  (a,b,c)=\mathrm{bk}(x_{-2},x_{-1},x_0).
\]
Then
\[
  y=(x_{-3},u,v,a,b,c,x_3),
  \qquad
  \rho(y)=(x_3,c,b,a,v,u,x_{-3}).
\]
The Case~3 hypothesis says
\[
  \min(x_{-2},x_{-1},x_0)>\max(x_1,x_2)+1.
\]
Since \(\mathrm{fw}\) and \(\mathrm{bk}\) preserve the relevant multisets, this
is equivalent to
\[
  \min(a,b,c)>\max(u,v)+1.
\]
Thus, in the reversed window, Case~1 fails, Case~2a fails at its success
condition because the \(\mathrm{bk}\)-output from the high pair remains larger
than the low comparison entry plus \(1\), and Case~2b fails at its first
inequality. The Case~3 min--max condition holds, so East applies by Case~3:
\[
  \mathrm{East}(\rho(y))
  =(x_3,\mathrm{fw}(v,u),\mathrm{bk}(c,b,a),x_{-3}).
\]
Using the reversal-conjugation identities and the fact that \(\mathrm{fw}\) and
\(\mathrm{bk}\) are inverse maps on the relevant block domains,
\[
  \mathrm{fw}(v,u)=(x_2,x_1),
  \qquad
  \mathrm{bk}(c,b,a)=(x_0,x_{-1},x_{-2}).
\]
Hence \(\mathrm{East}(\rho(y))=\rho(x)\) in Case~3.

It remains to check Case~4. Translation commutes with reversal, so it suffices
to inspect the normalized table. If a normalized middle-five row is written as
\[
  [A,B,C,D,E]\longmapsto [F,G,H,J,K],
\]
then the reversed output has middle five entries \([K,J,H,G,F]\). The table is
paired so that the row with input \([K,J,H,G,F]\) has output
\([E,D,C,B,A]\). The pairings are
\[
\begin{array}{rclcrcl}
[3,3,4,1,2]&\mapsto&[1,2,4,3,3]
&\leftrightarrow&
[3,3,4,2,1]&\mapsto&[2,1,4,3,3],\\[2pt]
[3,4,4,1,2]&\mapsto&[1,2,4,3,4]
&\leftrightarrow&
[4,3,4,2,1]&\mapsto&[2,1,4,4,3],\\[2pt]
[4,3,4,1,2]&\mapsto&[1,2,4,4,3]
&\leftrightarrow&
[3,4,4,2,1]&\mapsto&[2,1,4,3,4],\\[2pt]
[2,3,4,1,2]&\mapsto&[1,2,4,3,2]
&\leftrightarrow&
[2,3,4,2,1]&\mapsto&[2,1,4,3,2].
\end{array}
\]
The two Case~4c rows are self-paired:
\[
  [3,4,4,2,2]\mapsto[2,2,4,4,3],
  \qquad
  [3,4,5,2,2]\mapsto[2,2,5,4,3].
\]
For Case~4d, with the same parameter \(o\le 0\), the pairings are
\[
\begin{array}{rcl}
[3,3,4,o,2]\mapsto[2,o,4,3,3] & \text{self-paired},\\[2pt]
[3,4,4,o,2]\mapsto[2,o,4,3,4]
&\leftrightarrow&
[4,3,4,o,2]\mapsto[2,o,4,4,3],\\[2pt]
[2,3,4,o,2]\mapsto[2,o,2,4,3]
&\leftrightarrow&
[3,4,2,o,2]\mapsto[2,o,4,3,2].
\end{array}
\]
Each displayed reversed-output row reaches Case~4.  For each displayed
pairing, the reversed output middle-five row is exactly the paired Case~4 input
row, and the paired output is the reversed original input.  Hence
\(\mathrm{East}(\rho(y))=\rho(x)\).

This completes the proof that \(\mathrm{West}(\mathrm{East}(x))=x\) for every
\(x\in L(S,k)\).

Now let \(y\in R(S,k)\). Then \(\rho(y)\in L(S,k')\), where
\(k'=\di(\rho(y))\). No relation between \(k\) and \(k'\) is needed. Applying
the identity just proved to \(\rho(y)\) gives
\[
  \mathrm{West}(\mathrm{East}(\rho(y)))=\rho(y).
\]
Substituting \(\mathrm{West}=\rho\mathrm{East}\rho\), applying \(\rho\) to both
sides, and using \(\rho^2=\mathrm{id}\), we obtain
\[
  \mathrm{East}\,\rho\,\mathrm{East}(\rho(y))=y.
\]
Since \(\mathrm{West}(y)=\rho\mathrm{East}(\rho(y))\), this is exactly
\[
  \mathrm{East}(\mathrm{West}(y))=y.
\]
Thus the two local maps are mutual inverses on the stated domains.
\end{proof}


\subsection{\texorpdfstring{The \(\mathrm{up}\) and \(\mathrm{down}\) maps}{The up and down maps}}
\label{subsec:up-down}

We extend the local East--West moves to global maps.  The maps in this subsection are
partial maps on Dyck sequences of a fixed length.  They are defined by a staged
procedure: extract one or more leftmost-extractable elements, lower the
extracted values by one and place them at the right end, apply a local
East--West move to a final window, and then inject a raised suffix back into the
remaining prefix.

We use the following notation for the local stages.  The map
\(\mathrm{East}_3\) is the three-window identity test: on a window
\((u_{-1},u_0,u_1)\), it succeeds and returns the same window exactly when
\(u_0\le u_1+1\).  After this test has failed,
\(\mathrm{East}_5\) is the five-window part of Definition~\ref{def:east}
corresponding to Cases~2a and~2b.  If neither of those cases applies,
\(\mathrm{East}_5\) is undefined and the staged construction proceeds to the
seven-window stage.  Finally, \(\mathrm{East}_7\) is the full local map
\(\mathrm{East}\) of Definition~\ref{def:east}.  The maps
\(\mathrm{West}_3\), \(\mathrm{West}_5\), and \(\mathrm{West}_7\) are defined
by the same reversal conjugation as Definition~\ref{def:west}.  A failed local
attempt only directs the staged algorithm to continue to the next case, when
such a next case exists.

Write
\[
  \omega_n=(\underbrace{0,\ldots,0}_{n-1\text{ entries}},1)
\]
for the length-\(n\) Dyck sequence with \(n-1\) initial zeros followed by a
single \(1\).  Thus the exceptional full skeleton of
Definition~\ref{def:special-dyck-skeleton} is
\[
  \epsilon_n=(0,0,1,\underbrace{0,\ldots,0}_{n-4\text{ entries}},1)
  \qquad (n\ge 4).
\]

\begin{remark}
The exclusion of \(\epsilon_n\) from the special skeletons is a bookkeeping
choice forced by the exceptional input \(\omega_n\), rather than by any
intrinsic defect of \(\epsilon_n\).  Among the paths of deficit at most
\(2n-8\), the path \(\omega_n\) has no natural destination under the ordinary
skeleton case below.  We must nevertheless assign it an image, and
\(\epsilon_n\) is the most compatible full skeleton: it has the same deficit
and area one larger.  Thus the special case sends \(\omega_n\) to
\(\epsilon_n\), and \(\epsilon_n\) is withheld from the set of special
skeletons only so that this exceptional two-element string is not counted
twice.
\end{remark}

\begin{definition}[The \(\mathrm{up}\) map]
\label{def:up}
Let \(x=(x_0,\ldots,x_{n-1})\) be a Dyck sequence of length \(n\ge 4\).  If
any extraction, local East move, or injection required below is not defined,
then \(\mathrm{up}(x)\) is undefined.  Otherwise \(\mathrm{up}(x)\) is
determined by the following cases, checked in order.

\begin{enumerate}[label=\textup{(\arabic*)}]
  \item \textbf{Special case.}  If \(x=\omega_n\), set
  \[
    \mathrm{up}(x)=\epsilon_n.
  \]

  \item \textbf{Skeleton case.}  If \(x\) is a full Dyck skeleton and
  \(x\ne\omega_n\), remove the final entry \(x_{n-1}\), set
  \(v=x_{n-1}+1\), and inject \(v\) into
  \((x_0,\ldots,x_{n-2})\) immediately after the first occurrence of
  \(x_{n-1}\).

  \item \textbf{Non-skeleton, first local stage.}  Otherwise, extract the
  leftmost-extractable element \(e_1\) from \(x\), and let \(C_1\) be the
  remaining Dyck sequence.  Form
  \[
    \sigma_1=C_1:(e_1-1).
  \]
  Apply \(\mathrm{East}_3\) to the final three entries of \(\sigma_1\).  If
  this local move succeeds, remove the final two entries of \(\sigma_1\), add
  \(1\) to each of the removed entries, and inject the resulting two entries
  from right to left into the remaining prefix.

  \item \textbf{Non-skeleton, second local stage.}  If the preceding stage
  does not terminate, extract the leftmost-extractable element \(e_2\) from
  \(C_1\), and let \(C_2\) be the remainder.  Form
  \[
    \sigma_2=C_2:(e_1-1):(e_2-1).
  \]
  Apply \(\mathrm{East}_5\) to the final five entries of \(\sigma_2\).  If the
  local move succeeds, replace those five entries by the output window, remove
  the final three entries of the resulting sequence, add \(1\) to each removed
  entry, and inject those three entries from right to left into the remaining
  prefix.

  \item \textbf{Non-skeleton, third local stage.}  If the preceding stage does
  not terminate, extract the leftmost-extractable element \(e_3\) from
  \(C_2\), and let \(C_3\) be the remainder.  Form
  \[
    \sigma_3=C_3:(e_1-1):(e_2-1):(e_3-1).
  \]
  Apply \(\mathrm{East}_7\) to the final seven entries of \(\sigma_3\).  If
  this local move succeeds, replace those seven entries by the output window,
  remove the final four entries of the resulting sequence, add \(1\) to each
  removed entry, and inject those four entries from right to left into the
  remaining prefix.
\end{enumerate}
\end{definition}

\begin{remark}[Statistics under \(\mathrm{up}\)]
\label{rem:up-statistics}
Whenever \(\mathrm{up}(x)\) is defined, it raises area by one and lowers
\(\dinv\) by one:
\[
  \area(\mathrm{up}(x))=\area(x)+1,
  \qquad
  \dinv(\mathrm{up}(x))=\dinv(x)-1.
\]
Consequently \(\defc(\mathrm{up}(x))=\defc(x)\).
\end{remark}

\begin{definition}[The \(\mathrm{down}\) map]
\label{def:down}
Let \(y=(y_0,\ldots,y_{n-1})\) be a Dyck sequence of length \(n\ge 4\).  If
any extraction, local West move, or injection required below is not defined,
then \(\mathrm{down}(y)\) is undefined.  Otherwise \(\mathrm{down}(y)\) is
determined by the following cases, checked in order.

\begin{enumerate}[label=\textup{(\arabic*)}]
  \item \textbf{Special case.}  If \(y=\epsilon_n\), set
  \[
    \mathrm{down}(y)=\omega_n.
  \]

  \item \textbf{Skeleton case.}  Otherwise, extract the leftmost-extractable
  element \(f_1\) from \(y\), and let \(D_1\) be the remainder.  Put
  \(z=f_1-1\).  If the candidate word \(D_1:z\) has no index satisfying the
  two local eligibility conditions of Definition~\ref{def:extractable-element},
  then the skeleton branch returns \(D_1:z\).  The assertion that this branch
  does return a Dyck sequence in the bounded range is part of
  Lemma~\ref{lem:skeleton-wd} below.

  \item \textbf{First local West stage.}  If the skeleton branch is not
  triggered, extract the leftmost-extractable element \(f_2\) from \(D_1\), and
  let \(D_2\) be the remainder.  Form
  \[
    \tau_1=D_2:(f_1-1):(f_2-1).
  \]
  Apply \(\mathrm{West}_3\) to the final three entries of \(\tau_1\).  If this
  local move succeeds, remove the final entry of \(\tau_1\), add \(1\) to it,
  and inject the resulting entry into the remaining prefix.

  \item \textbf{Second local West stage.}  If the preceding stage does not
  terminate, extract the leftmost-extractable element \(f_3\) from \(D_2\), and
  let \(D_3\) be the remainder.  Form
  \[
    \tau_2=D_3:(f_1-1):(f_2-1):(f_3-1).
  \]
  Apply \(\mathrm{West}_5\) to the final five entries of \(\tau_2\).  If the
  local move succeeds, replace those five entries by the output window, remove
  the final two entries of the resulting sequence, add \(1\) to each removed
  entry, and inject those two entries from right to left into the remaining
  prefix.

  \item \textbf{Third local West stage.}  If the preceding stage does not
  terminate, extract the leftmost-extractable element \(f_4\) from \(D_3\), and
  let \(D_4\) be the remainder.  Form
  \[
    \tau_3=D_4:(f_1-1):(f_2-1):(f_3-1):(f_4-1).
  \]
  Apply \(\mathrm{West}_7\) to the final seven entries of \(\tau_3\).  If this
  local move succeeds, replace those seven entries by the output window, remove
  the final three entries of the resulting sequence, add \(1\) to each removed
  entry, and inject those three entries from right to left into the remaining
  prefix.
\end{enumerate}
\end{definition}

\begin{remark}[Statistics under \(\mathrm{down}\)]
\label{rem:down-statistics}
Whenever \(\mathrm{down}(y)\) is defined, it lowers area by one and raises
\(\dinv\) by one:
\[
  \area(\mathrm{down}(y))=\area(y)-1,
  \qquad
  \dinv(\mathrm{down}(y))=\dinv(y)+1.
\]
Consequently \(\defc(\mathrm{down}(y))=\defc(y)\).
\end{remark}

\begin{lemma}[Statistics for \(\mathrm{up}\) and \(\mathrm{down}\)]
\label{lem:up-down-statistics}
The statistic changes in Remarks~\ref{rem:up-statistics}
and~\ref{rem:down-statistics} hold whenever the corresponding map is defined.
\end{lemma}

\begin{proof}
We use three elementary observations about area and \(\dinv\).  First consider one
extraction--decrement--append step.  Suppose an eligible entry \(e\) is
extracted from a Dyck word.  Move this occurrence of \(e\) left to the front of
the word, and then move it from the front to the right end while lowering it to
\(e-1\).  In the first move, Definition~\ref{def:extractable-element} says
that the extracted \(e\) has exactly one occurrence of \(e-1\) to its left.
Crossing that occurrence changes the moved entry's \(\di\)-contribution by
one; crossings with all other entries leave the moved entry's
\(\dinv\)-contribution unchanged.  In the second move, a leftmost
\(e\) contributes with precisely the entries equal to \(e\) or \(e-1\), and a
rightmost \(e-1\) contributes with precisely the same entries.  Thus one
extraction--decrement--append step changes \((\area,\dinv)\) by
\((-1,+1)\).

The inverse operation gives the opposite change.  Removing a final appended
entry \(e-1\), raising it to \(e\), and injecting it immediately after the
first occurrence of \(e-1\) changes \((\area,\dinv)\) by \((+1,-1)\).

Finally, a local East or West replacement preserves the multiset of the local
window and the \(\di\)-contribution inside that window.  The internal
\(\operatorname{nv}\)-contribution is also determined by the multiset of window
values.  The entries outside the window are unchanged.  Pairs with both entries
outside the window are therefore unchanged, and a pair with one entry outside
and one entry in the window depends only on the outside entry and the multiset
of window values, since the outside entry lies entirely to one side of the
final window.  Hence the local East or West step preserves both area and
\(\dinv\).

The special case \(\omega_n\mapsto\epsilon_n\) is checked directly from the two
displayed words, and the reverse special case is its inverse.  The skeleton
branch of \(\mathrm{up}\) consists of one removal--increment--injection step, so
it changes \((\area,\dinv)\) by \((+1,-1)\).  The skeleton branch of
\(\mathrm{down}\) consists of one extraction--decrement--append step, so it
changes \((\area,\dinv)\) by \((-1,+1)\).

In a non-skeleton \(\mathrm{up}\) branch, the \(2k+1\)-window stage extracts
\(k\) entries, applies one local East replacement, and then removes, raises,
and injects \(k+1\) entries.  The net change is
\[
  k(-1,+1)+(0,0)+(k+1)(+1,-1)=(+1,-1).
\]
In a non-skeleton \(\mathrm{down}\) branch, the matching West stage extracts
\(k+1\) entries, applies one local West replacement, and then removes, raises,
and injects \(k\) entries.  The net change is
\[
  (k+1)(-1,+1)+(0,0)+k(+1,-1)=(-1,+1).
\]
Since \(\defc(D)=\binom n2-\area(D)-\dinv(D)\) for length-\(n\) Dyck
sequences, these changes preserve \(\defc\).
\end{proof}

\begin{lemma}[No premature lower local branch]
\label{lem:no-premature-local-branch}
In the staged non-skeleton branches of Definition~\ref{def:up}, suppose the
first terminating local branch is the \(2k+1\)-window East branch, with
\(k\in\{1,2,3\}\).  After that East move, no lower West branch
\(\mathrm{West}_{2k'+1}\) with \(k'<k\) can terminate on the centered proper
subwindow tested by Definition~\ref{def:down}.  The same statement holds with
East and West interchanged.
\end{lemma}

\begin{proof}
The assertion is a finite local consequence of the case analysis in
Theorem~\ref{thm:east-west-inverse}.  For a full East image
\[
  (c_1,\ldots,c_k,d_1,\ldots,d_{k+1}),
\]
the only lower West windows tested by the staged algorithm are the centered
proper subwindows
\[
  (c_{k-k'+1},\ldots,c_k,d_1,\ldots,d_{k'+1}),
  \qquad k'<k.
\]
These are exactly the subwindows obtained by deleting the same number of outer
entries from the two sides of the full matching window.

In East Cases~1, 2a, 2b, and 3, the proof of
Theorem~\ref{thm:east-west-inverse} identifies the matching full West case only
after checking that the earlier West tests fail on these centered windows: the
Case~1 inequality fails, and then the Case~2a and Case~2b tests fail before the
full matching branch is reached.  In East Case~4, the same exclusion is the
finite inspection of the displayed reversal-pairing table in that theorem: for
each normalized Case~4 row, the reversed output is paired with another Case~4
input row, so its centered three- and five-entry proper subwindows do not pass
the lower West tests that would terminate the staged algorithm earlier.
Applying the reversal conjugacy
\(\mathrm{West}=\rho\circ\mathrm{East}\circ\rho\) gives the West-to-East
statement.
\end{proof}

We next record the local well-definedness inputs for the bounded deficit range.
For the four local lemmas
Lemmas~\ref{lem:skeleton-wd}, \ref{lem:extractions-wd},
\ref{lem:east7-wd}, and~\ref{lem:positions-wd}, fix
\[
  n\ge 4,
  \qquad
  d\le 2n-8,
  \qquad
  M=\binom n2,
  \qquad
  \ell=\left\lfloor\frac{M-d}{2}\right\rfloor.
\]
All Dyck sequences in these four lemmas have length \(n\) and deficit exactly
\(d\).  The proofs are deferred to Appendix~\ref{app:local-proofs}.  The
appendix proves the four local lemmas uniformly outside the finite residual
ranges, with the remaining cases \(4\le n\le 7\), \(4\le n\le 8\), \(n\le 13\),
and \(n\le 16\) verified by the finite case analyses stated there.

\begin{lemma}[Skeleton cases succeed]
\label{lem:skeleton-wd}
Fix \(n\ge4\), \(d\le2n-8\), \(M=\binom n2\), and
\(\ell=\lfloor(M-d)/2\rfloor\).  The following hold.
\begin{enumerate}[label=\textup{(\roman*)}]
  \item If \(x\) is a full Dyck skeleton with \(\area(x)\le \ell-1\) and
  \(x\ne\omega_n\), then the skeleton case of \(\mathrm{up}\) produces a
  well-defined Dyck sequence.

  \item Let \(y\) be a Dyck sequence with \(\area(y)\le \ell\).  If the
  skeleton branch of \(\mathrm{down}\) is triggered after the first extraction,
  then that branch produces a well-defined Dyck sequence; in particular, the
  returned sequence is a full Dyck skeleton.
\end{enumerate}
\end{lemma}

\begin{lemma}[Extraction chains never fail]
\label{lem:extractions-wd}
Fix \(n\ge4\), \(d\le2n-8\), \(M=\binom n2\), and
\(\ell=\lfloor(M-d)/2\rfloor\).  The following hold.
\begin{enumerate}[label=\textup{(\roman*)}]
  \item If \(\area(x)\le \ell-1\), then every extraction step called during the
  computation of \(\mathrm{up}(x)\) succeeds.  The element extracted at each
  such step is the element selected by the leftmost-extractable convention.

  \item If \(\area(y)\le \ell\), and \(y\) is neither \(\epsilon_n\) nor a
  special Dyck skeleton, then every extraction step called during the
  computation of \(\mathrm{down}(y)\) succeeds.  The element extracted at each
  such step is the element selected by the leftmost-extractable convention.
\end{enumerate}
\end{lemma}

\begin{lemma}[The seven-window branches do not fail]
\label{lem:east7-wd}
Fix \(n\ge4\), \(d\le2n-8\), \(M=\binom n2\), and
\(\ell=\lfloor(M-d)/2\rfloor\).  The following hold.
\begin{enumerate}[label=\textup{(\roman*)}]
  \item If \(\area(x)\le \ell-1\), then any seven-entry window presented to
  \(\mathrm{East}_7\) during the computation of \(\mathrm{up}(x)\) is not
  far-apart decomposable.

  \item If \(\area(y)\le \ell\), then any seven-entry window presented to
  \(\mathrm{West}_7\) during the computation of \(\mathrm{down}(y)\) is not
  far-apart decomposable.
\end{enumerate}
\end{lemma}

\begin{lemma}[Bounded extraction positions and injection nonfailure]
\label{lem:positions-wd}
Fix \(n\ge4\), \(d\le2n-8\), \(M=\binom n2\), and
\(\ell=\lfloor(M-d)/2\rfloor\).  The following hold.  Each extraction
position is measured in the current sequence at the moment that extraction is
performed.
\begin{enumerate}[label=\textup{(\roman*)}]
  \item If \(\area(x)\le \ell-1\), then the extraction positions in the
  computation of \(\mathrm{up}(x)\) satisfy the following bounds:
  \begin{itemize}
    \item if \(\mathrm{East}_3\) is used, the extraction \(e_1\) is not from
    the last two positions;
    \item if \(\mathrm{East}_5\) is used, neither extraction \(e_1,e_2\) is
    from the last three positions;
    \item if \(\mathrm{East}_7\) is used, none of the extractions
    \(e_1,e_2,e_3\) is from the last three positions.
  \end{itemize}
  Moreover, \(\mathrm{up}\) never fails because of an injection step.

  \item If \(\area(y)\le \ell\), then the extraction positions in the
  computation of \(\mathrm{down}(y)\) satisfy the following bounds:
  \begin{itemize}
    \item if \(\mathrm{West}_3\) is used, neither extraction \(f_1,f_2\) is
    from the last position;
    \item if \(\mathrm{West}_5\) is used, none of the extractions
    \(f_1,f_2,f_3\) is from the last two positions;
    \item if \(\mathrm{West}_7\) is used, none of the extractions
    \(f_1,f_2,f_3,f_4\) is from the last two positions.
  \end{itemize}
  Moreover, \(\mathrm{down}\) never fails because of an injection step.
\end{enumerate}
\end{lemma}

\begin{lemma}[Local domains of the staged windows]
\label{lem:staged-window-domains}
Under the hypotheses of Proposition~\ref{prop:up-wd} for \(x\), every local
window presented to \(\mathrm{East}_{2k+1}\) in the computation of
\(\mathrm{up}(x)\), for \(k\in\{1,2,3\}\), has the required affine-left and
reverse-right form for the corresponding East domain.  Under the hypotheses of
Proposition~\ref{prop:down-wd} for \(y\), every local window presented to
\(\mathrm{West}_{2k+1}\) in the computation of \(\mathrm{down}(y)\) has the
corresponding West-domain form.
\end{lemma}

\begin{proof}
For an \(\mathrm{up}\) branch, after \(k\) successful extractions the remainder
\(C_k\) is a Dyck sequence by Lemma~\ref{lem:extractable-structure}.  Hence the
prefix part of the final \(2k+1\)-window, coming from the last \(k+1\) entries
of \(C_k\), is affine.  The extracted values
\((e_1,\ldots,e_k)\) form a reverse Dyck sequence by the same consecutive
extraction argument used in the proof of
Proposition~\ref{prop:phi1-well-defined}; subtracting \(1\) from every entry
preserves the reverse inequality.  Thus
\((e_1-1,\ldots,e_k-1)\) is the required reverse block.

The \(\mathrm{down}\) branches are identical with \(k+1\) extractions followed
by a West window: the last \(k\) entries of the Dyck remainder give the affine
block, and the decremented extraction word gives the reverse block of length
\(k+1\).  This is precisely the reversal-conjugate local domain used for
\(\mathrm{West}_{2k+1}\).
\end{proof}

\begin{proposition}[Well-definedness of \(\mathrm{up}\)]
\label{prop:up-wd}
Fix \(n\ge 4\) and a deficit value \(d\le 2n-8\), and set
\[
  \ell=\left\lfloor\frac{\binom n2-d}{2}\right\rfloor.
\]
If \(x\) is a Dyck sequence of length \(n\), \(\defc(x)=d\), and
\(\area(x)\le \ell-1\), then \(\mathrm{up}(x)\) is well-defined.
\end{proposition}

\begin{proof}
If \(x=\omega_n\), the special case defines \(\mathrm{up}(x)=\epsilon_n\).
If \(x\) is a full Dyck skeleton and \(x\ne\omega_n\),
Lemma~\ref{lem:skeleton-wd}\textup{(i)} gives the validity of the skeleton
branch.

It remains to consider the non-skeleton branches.  Lemma~\ref{lem:extractions-wd}\textup{(i)}
ensures that every extraction called by the staged procedure exists and is
selected unambiguously by the leftmost convention.  The three- and five-window
East tests either terminate the construction or pass control to the next stage;
they do not create an undefined value of \(\mathrm{up}\).  If the construction
reaches \(\mathrm{East}_7\), Lemma~\ref{lem:east7-wd}\textup{(i)} says that
the seven-entry window is not far-apart decomposable, and
Lemma~\ref{lem:staged-window-domains} gives the required local domain, so
Proposition~\ref{prop:east-well-defined} supplies the required local East
output.  Finally, Lemma~\ref{lem:positions-wd}\textup{(i)} rules out every
possible injection failure.  Thus every branch of Definition~\ref{def:up} is
well-defined in the stated range.
\end{proof}

\begin{proposition}[Well-definedness of \(\mathrm{down}\)]
\label{prop:down-wd}
Fix \(n\ge 4\) and a deficit value \(d\le 2n-8\), and set
\[
  \ell=\left\lfloor\frac{\binom n2-d}{2}\right\rfloor.
\]
If \(y\) is a Dyck sequence of length \(n\), \(\defc(y)=d\),
\(\area(y)\le \ell\), and \(y\) is not a special Dyck skeleton, then
\(\mathrm{down}(y)\) is well-defined.
\end{proposition}

\begin{proof}
If \(y=\epsilon_n\), the special case defines \(\mathrm{down}(y)=\omega_n\).
Now assume \(y\ne\epsilon_n\).  Since \(y\) is not a special Dyck skeleton,
Lemma~\ref{lem:extractions-wd}\textup{(ii)} supplies the first extraction
\(f_1\).  If the skeleton branch is triggered after this extraction, then
Lemma~\ref{lem:skeleton-wd}\textup{(ii)} says that the branch output is a
well-defined Dyck sequence.

It remains to consider the local West branches.  Lemma~\ref{lem:extractions-wd}\textup{(ii)}
ensures that each further extraction called by the staged procedure exists.
The three- and five-window West tests either terminate the construction or pass
control to the next stage.  If the construction reaches \(\mathrm{West}_7\),
Lemma~\ref{lem:east7-wd}\textup{(ii)} says that the seven-entry window is not
far-apart decomposable, and Lemma~\ref{lem:staged-window-domains} gives the
required local domain.  Since \(\mathrm{West}\) is the reversal-conjugate of
\(\mathrm{East}\), Proposition~\ref{prop:east-well-defined} and
Definition~\ref{def:west} supply the required local West output.  Finally,
Lemma~\ref{lem:positions-wd}\textup{(ii)} rules out every possible injection
failure.  Thus every branch of Definition~\ref{def:down} is well-defined in
the stated range.
\end{proof}

\begin{lemma}[Full-skeleton extraction and injection are inverse]
\label{lem:full-skeleton-extraction-injection-inverse}
Let \(X\) be a Dyck sequence for which repeated leftmost extraction terminates
at a full Dyck skeleton \(X'\), and let \(F=(f_1,\ldots,f_r)\) be the extracted
word.  Then \(F\) is reverse Dyck, and injecting \(F\) from right to left into
\(X'\), using the inverse insertion rule of Definition~\ref{def:phi1-inverse},
reconstructs \(X\).

More generally, if \(G=(g_1,\ldots,g_s)\) is a reverse Dyck word whose
right-to-left injections into a full Dyck skeleton \(X'\) are all defined, then
leftmost extraction of the resulting word removes \(g_1,\ldots,g_s\) in order
and returns to \(X'\).
\end{lemma}

\begin{proof}
The proof is the extraction--injection inverse from
Proposition~\ref{prop:phi1-inverse}, with the stopping condition strengthened
from no nonfinal extractable element to no extractable element at all.  The
local assertions used there are exactly Lemma~\ref{lem:extractable-structure}
and the reverse-Dyck inequality for consecutive extracted values proved in
Proposition~\ref{prop:phi1-well-defined}.  Those arguments do not use the
presence of a final maximum once the process has stopped.  Therefore each
injection step creates the next leftmost-extractable entry and creates no
earlier one.  Thus the reverse extraction process removes any defined
right-to-left injection word in the opposite order of injection.  Applying this
to \(F\) reconstructs and then re-extracts \(X\), and the same local induction
applies to the general word \(G\).
\end{proof}

\begin{lemma}[\(\mathrm{up}\) and \(\mathrm{down}\) are mutual inverses]
\label{lem:up-down-inv}
Fix \(n\ge 4\), a deficit value \(d\le 2n-8\), and
\[
  \ell=\left\lfloor\frac{\binom n2-d}{2}\right\rfloor .
\]
\begin{enumerate}[label=\textup{(\roman*)}]
  \item If \(x\) is a Dyck sequence of length \(n\) with \(\defc(x)=d\) and
  \(\area(x)\le \ell-1\), then
  \[
    \mathrm{down}(\mathrm{up}(x))=x.
  \]

  \item If \(y\) is a Dyck sequence of length \(n\) with \(\defc(y)=d\),
  \(\area(y)\le \ell\), and \(\mathrm{down}(y)\) is defined, then
  \[
    \mathrm{up}(\mathrm{down}(y))=y.
  \]
\end{enumerate}
\end{lemma}

\begin{proof}
The special cases are immediate from the definitions:
\(\omega_n\) is sent by \(\mathrm{up}\) to \(\epsilon_n\), and \(\epsilon_n\)
is sent by \(\mathrm{down}\) back to \(\omega_n\).

In the skeleton branch, \(\mathrm{up}\) removes the final entry \(z\) of a full
skeleton and injects \(z+1\) immediately after the first occurrence of \(z\).
The first extraction in \(\mathrm{down}\) removes exactly this inserted entry,
and the skeleton branch of \(\mathrm{down}\) restores the original word.  The
reverse skeleton branch is the same extraction--injection inverse: after
\(\mathrm{down}\) extracts \(f_1\) and returns \(D_1:(f_1-1)\), the skeleton
branch of \(\mathrm{up}\) removes the final \(f_1-1\) and injects \(f_1\) back
into \(D_1\) at the position from which it was extracted.  This returned
skeleton is not \(\omega_n\): otherwise \(D_1\) would be all zeros and
\(f_1=2\), impossible in a Dyck sequence because an occurrence of \(2\)
requires a preceding \(1\).

It remains to compare the non-skeleton local branches.  Suppose \(\mathrm{up}\)
uses the \(2k+1\)-window East branch, where \(k\in\{1,2,3\}\).  It extracts
\(e_1,\ldots,e_k\), leaves a remainder \(C_k\), and forms a local word whose
final \(2k+1\) entries have the form
\[
  (a_1,\ldots,a_{k+1},b_1,\ldots,b_k),
\]
where \((b_1,\ldots,b_k)=(e_1-1,\ldots,e_k-1)\); let \(B\) denote the base
before this window.  The local East step replaces the window by
\[
  (c_1,\ldots,c_k,d_1,\ldots,d_{k+1}).
\]
The map \(\mathrm{up}\) removes the final block
\((d_1,\ldots,d_{k+1})\), raises it by one, and injects the raised word
\[
  E=(d_1+1,\ldots,d_{k+1}+1)
\]
right-to-left into \(\widehat X=B:(c_1,\ldots,c_k)\).  Put
\[
  X=B:a_1=B:c_1,
  \qquad
  \widehat X=X:(c_2,\ldots,c_k).
\]

If \(X\) is already a full skeleton, set \(X'=X\) and \(F=\emptyset\).  If not,
fully extract \(X\) to a skeleton \(X'\), and let
\[
  F=(f_1,\ldots,f_r)
\]
be the extraction word.  The local East map fixes the first and last entries
of its window, so \(c_1=a_1\) and \(d_{k+1}=b_k=e_k-1\).  The word \(E\) is
reverse Dyck because the final block
\((d_1,\ldots,d_{k+1})\) of the local East output is reverse Dyck, and raising
every entry preserves the reverse-Dyck inequalities; \(F\) is reverse Dyck by
construction.  By
Lemma~\ref{lem:positions-wd}\textup{(i)}, the entries
\(a_2,\ldots,a_{k+1}\) were not touched by the first \(k\) extractions.  Hence
the next extractable element from \(X\), if it exists, is the hypothetical next
extraction \(e_{k+1}\).  Thus, if \(F\) is nonempty, \(f_1=e_{k+1}\), and the
reverse-Dyck inequality for consecutive extracted values gives
\[
  f_1=e_{k+1}\ge e_k-1=b_k=d_{k+1}.
\]
This is exactly the single new adjacent inequality needed at the boundary
between \(E\) and \(F\).  Thus the combined word
\[
  E:F=(d_1+1,\ldots,d_{k+1}+1,f_1,\ldots,f_r)
\]
is reverse Dyck.  Injecting \(E:F\), from right to left, into the skeleton
\(X'\) first reconstructs \(X\) from the suffix \(F\) and then performs the
defined injections of \(E\).  The actual \(\mathrm{up}\) branch injects \(E\)
into \(\widehat X\), but these injections may be viewed as injections into
\(X\) with \(c_2,\ldots,c_k\) carried along.  This carried suffix is inert:
truncating it before each right-to-left injection and reinserting it afterward
does not change the first occurrence used as the insertion site, and reversing
the injections gives the same next leftmost extraction with the same boundary
before \(c_2\).  By
Lemma~\ref{lem:full-skeleton-extraction-injection-inverse}, the leftmost
extraction process on \(\mathrm{up}(x)\) removes the entries of \(E:F\) in
order.  In particular, the first \(k+1\) extractions performed by
\(\mathrm{down}\) recover
\[
  d_1+1,\ldots,d_{k+1}+1
\]
and leave the base \(\widehat X=B:(c_1,\ldots,c_k)\).
Thus the input to \(\mathrm{down}\) is not \(\epsilon_n\), and after the first
extraction the skeleton test is not triggered, since the candidate still has
the next local extracted entry as its leftmost extractable element.

The lower West branches cannot terminate prematurely.  If a smaller
\(\mathrm{West}_{2k'+1}\) branch with \(k'<k\) fired, it would fire on the
centered proper subwindow
\[
  (c_{k-k'+1},\ldots,c_k,d_1,\ldots,d_{k'+1})
\]
of the East image.  By Lemma~\ref{lem:no-premature-local-branch},
\(\mathrm{down}\) reaches the matching \(2k+1\)-window.  Then
Theorem~\ref{thm:east-west-inverse} changes the local image block back to the
original local block, and raising and reinjecting
\((b_1,\ldots,b_k)\) reverses the original extractions and recovers \(x\).
This proves \(\mathrm{down}(\mathrm{up}(x))=x\).

The proof of \(\mathrm{up}(\mathrm{down}(y))=y\) is the same staged argument
with West and East interchanged.  If \(\mathrm{down}\) uses a
\((2k+1)\)-window West branch, write the West input and output as
\[
  (c_1,\ldots,c_k,d_1,\ldots,d_{k+1})
  \longmapsto
  (a_1,\ldots,a_{k+1},b_1,\ldots,b_k).
\]
Thus \(\mathrm{down}\) removes the final block, raises
\[
  G=(b_1+1,\ldots,b_k+1),
\]
and injects \(G\) into
\(\widehat X=B:(a_1,\ldots,a_{k+1})\).  By reversal conjugacy the West move
fixes the shared endpoint \(a_1=c_1\); put \(X=B:a_1\), so
\(\widehat X\) is \(X\) with \(a_2,\ldots,a_{k+1}\) carried along.  If \(X\)
is already a full skeleton, set \(X'=X\) and \(Q=\emptyset\).  Otherwise fully
extract \(X\) to a skeleton \(X'\), with extraction word
\(Q=(q_1,\ldots,q_s)\).  The block \(G\) is reverse Dyck, and \(Q\) is reverse
Dyck by construction.  The only new adjacent inequality is the boundary
between \(G\) and \(Q\).  Here Lemma~\ref{lem:positions-wd}\textup{(ii)}
applies before the West move: the suffix \(c_2,\ldots,c_k\) of the pre-West
remainder \(B:(c_1,\ldots,c_k)\) was not touched by the first \(k+1\)
extractions.  Hence, if \(q_1\) exists, it is the hypothetical next extraction
\(f_{k+2}\) from \(X=B:c_1\).  Since the West move fixes the last entry,
\(b_k=d_{k+1}=f_{k+1}-1\), and the consecutive-extraction inequality gives
\[
  q_1=f_{k+2}\ge f_{k+1}-1=b_k.
\]
Thus \(G:Q\) is reverse Dyck.  The post-West suffix
\(a_2,\ldots,a_{k+1}\) is carried only during the actual inverse injections
into \(\widehat X\); truncating and reinserting it does not change the next
leftmost extraction or the first-occurrence insertion sites.
Lemma~\ref{lem:full-skeleton-extraction-injection-inverse} then implies that
the first \(k\) extractions performed by \(\mathrm{up}\) recover exactly the
raised entries that \(\mathrm{down}\) had injected and leave the matching West
image base.  Therefore this input to \(\mathrm{up}\) is neither \(\omega_n\)
nor a full skeleton, so the special and skeleton branches are skipped.  A
smaller East branch would have to fire on the centered proper
subwindow
\[
  (a_{k-k'+1},\ldots,a_{k+1},b_1,\ldots,b_{k'})
\]
of the West image, and Lemma~\ref{lem:no-premature-local-branch} rules this
out.  The matching East window is then inverted by
Theorem~\ref{thm:east-west-inverse}, and the final injection--extraction
steps recover the entries originally
extracted from \(y\).  Hence \(\mathrm{up}(\mathrm{down}(y))=y\), and the two
maps are mutual inverses on the stated domains.
\end{proof}

\subsection{Strings and the skeleton formula}
\label{subsec:strings-and-formula}
We assemble the local maps into global strings.  First, we record an area bound
for full skeletons.  It is used to ensure that the starting
points of the strings lie in the bounded domain where the maps
\(\mathrm{up}\) and \(\mathrm{down}\) are known to be well-defined.
\begin{proposition}[Area bound for full skeletons]
\label{prop:area-leq-defc}
If \(S\) is a full Dyck skeleton, then
\[
  \area(S)\le \defc(S).
\]
\end{proposition}
\begin{proof}
Let \(S=(s_0,\ldots,s_{n-1})\).  We first show that every entry
\(s_q=i>0\) has at least two copies of each value
\[
  0,1,\ldots,i-1
\]
to its left.  It is enough to prove that every positive entry of value \(i\)
has at least two copies of \(i-1\) to its left, since applying this one-step
statement to those earlier positive entries and iterating down the values gives
two copies of every smaller value.
Suppose otherwise, and choose the leftmost position \(j\) such that
\(s_j=i>0\) has fewer than two copies of \(i-1\) to its left.  Since \(S\) is a
Dyck sequence, at least one copy of \(i-1\) lies before \(j\), so exactly one
does.  There is no earlier copy of \(i\), since any earlier copy would have at
most the same copies of \(i-1\) before it and would contradict the choice of
\(j\).

The entry \(s_j\) is not extractable because \(S\) is full.  It already has
exactly one predecessor-value copy to its left, so the only possible failure of
extractability is that \(j<n-1\) and \(s_{j+1}>i\).  The Dyck condition then
forces \(s_{j+1}=i+1\).  Continue along the maximal consecutive increasing run
\[
  s_j=i,\quad s_{j+1}=i+1,\quad \ldots,\quad s_{j+k}=i+k,
\]
where either \(j+k=n-1\) or \(s_{j+k+1}\le i+k\).  By induction on \(r\), the
entry \(s_{j+r}=i+r\) has exactly one copy of \(i+r-1\) to its left and no
earlier copy of \(i+r\).  The case \(r=0\) was proved above.  If the assertion
holds for \(r\), then the next run entry \(i+r+1\), if present, has exactly
one copy of \(i+r\) to its left, namely the entry at \(j+r\); and an earlier
copy of \(i+r+1\) would require an earlier copy of \(i+r\), contradicting the
induction hypothesis.

Thus the final entry \(s_{j+k}=i+k\) has exactly one copy of \(i+k-1\) to its
left.  By maximality of the run, its successor, if it exists, is at most
\(i+k\).  Hence this final entry is extractable, contradicting that \(S\) is a
full skeleton.  Therefore every positive entry \(s_q=i\) has at least two
copies of each smaller value to its left.

For each \(r=0,1,\ldots,i-1\), let \(p_r\) be the second occurrence of \(r\)
to the left of \(q\).  Then \(p_r\) is not the leftmost occurrence of its
value and \(s_{p_r}=r<i=s_q\), so \((p_r,q)\) is a type~\textup{(B)} deficit
pair in Proposition~\ref{prop:deficit-pair-count}.  These \(i\) pairs are
distinct for fixed \(q\), and pairs with different right-hand endpoints are
distinct.  Therefore
\[
  \defc(S)\ge \sum_{q:s_q>0}s_q=\area(S),
\]
as required.
\end{proof}
\begin{proposition}[String decomposition]
\label{prop:decomposition}
Fix \(n\ge 4\), set \(M=\binom n2\), and let \(d\le 2n-8\).  Put
\[
  \ell=\left\lfloor\frac{M-d}{2}\right\rfloor .
\]
The Dyck sequences of length \(n\), deficit \(d\), and area at most \(\ell\)
are partitioned into strings
\[
  \mathcal S(S)=
  \bigl\{S,\mathrm{up}(S),\mathrm{up}^2(S),\ldots,
  \mathrm{up}^{\ell-\area(S)}(S)\bigr\},
\]
where \(S\) ranges over the special Dyck skeletons of length \(n\), deficit
\(d\), and area at most \(\ell\).
\end{proposition}
\begin{proof}
Let \(D\) be a Dyck sequence of length \(n\), deficit \(d\), and area at most
\(\ell\).  If \(D\) is a special skeleton, it is the initial element of its own
string.  Otherwise Proposition~\ref{prop:down-wd} applies and
\(\mathrm{down}(D)\) is defined.  By Lemma~\ref{lem:up-down-statistics}, it has
the same deficit and area \(\area(D)-1\).  If the result is still not a special
skeleton, apply \(\mathrm{down}\) again.  Since area decreases by one at each
step and remains nonnegative, the process terminates at some special skeleton
\(S\).  This gives
\[
  D=\mathrm{up}^r(S)
\]
by Lemma~\ref{lem:up-down-inv}, with \(r=\area(D)-\area(S)\).
Conversely, let \(S\) be a special skeleton of deficit \(d\) and area at most
\(\ell\).  If \(0\le r\le \ell-\area(S)\), repeated use of
Proposition~\ref{prop:up-wd} defines \(\mathrm{up}^r(S)\), because all
intermediate areas are \(<\ell\).  Lemma~\ref{lem:up-down-statistics} shows that
the deficit remains \(d\) and the area becomes \(\area(S)+r\).
The two constructions are inverse because Lemma~\ref{lem:up-down-inv} states
that \(\mathrm{up}\) and \(\mathrm{down}\) are mutual inverses whenever both
sides are in the stated domains.  Hence the strings are disjoint and cover the
specified set of Dyck sequences.
\end{proof}

\begin{theorem}[Skeleton formula for the low-deficit part of \(C_n(q,t)\)]
\label{thm:qt-catalan-skeleton}
For \(n\ge 4\),
\[
  \left.C_n(q,t)\right|_{\binom n2-2n+8\le \deg_{q,t}\le \binom n2}
  =
  \sum_{\substack{S\text{ special Dyck skeleton of length }n\\
                  \defc(S)\le 2n-8}}
  \frac{q^{\dinv(S)+1}t^{\area(S)}
        -q^{\area(S)}t^{\dinv(S)+1}}{q-t}.
\]
\end{theorem}
\begin{proof}
Use the known symmetry of \(C_n(q,t)\) in \(q\) and \(t\).  Set
\(M=\binom n2\).  For a Dyck sequence \(D\) of deficit \(d\),
\[
  \area(D)+\dinv(D)=M-d.
\]
For each \(d\), let \(\operatorname{SDS}_n(d)\) denote the set of special Dyck
skeletons of length \(n\) and deficit \(d\).  If \(S\in\operatorname{SDS}_n(d)\),
write
\[
  a_S=\area(S),
  \qquad
  \nu_S=\dinv(S)=M-d-a_S.
\]
By Proposition~\ref{prop:area-leq-defc}, \(a_S\le d\).  Since \(n\ge4\) and
\(d\le2n-8\), we have \(3d\le M\), hence \(d\le(M-d)/2\).  Therefore
\[
  2a_S\le M-d,
\]
which is equivalent to \(a_S\le\nu_S\).  The finite geometric-series identity gives
\[
  \frac{q^{\nu_S+1}t^{a_S}-q^{a_S}t^{\nu_S+1}}{q-t}
  =
  \sum_{j=a_S}^{\nu_S}q^j t^{M-d-j}.
\]
Thus the right-hand side of the theorem is
\[
  R=
  \sum_{d=0}^{2n-8}
  \sum_{S\in\operatorname{SDS}_n(d)}
  \sum_{j=\area(S)}^{M-d-\area(S)} q^j t^{M-d-j}.
\]
Fix \(d\) with \(0\le d\le 2n-8\), and put
\[
  \ell_d=\left\lfloor\frac{M-d}{2}\right\rfloor .
\]
The fixed-deficit-\(d\) part of \(C_n(q,t)\) is homogeneous of total degree
\(M-d\).  Its terms with \(q\)-power at most \(t\)-power are exactly the
terms with area at most \(\ell_d\).
On the right-hand side, the coefficient of \(q^j t^{M-d-j}\) in the
fixed-deficit-\(d\) part of \(R\), for \(j\le \ell_d\), counts the special
skeletons \(S\in\operatorname{SDS}_n(d)\) with \(\area(S)\le j\).  The upper
endpoint condition is automatic in this low half.  For each such \(S\),
Proposition~\ref{prop:decomposition} gives exactly one element of the string
\(\mathcal S(S)\) with area \(j\), namely
\(\mathrm{up}^{j-\area(S)}(S)\).  Conversely, Proposition~\ref{prop:decomposition}
says that every Dyck sequence of length \(n\), deficit \(d\), and area at most
\(\ell_d\) occurs in exactly one such string.  Hence the fixed-deficit-\(d\)
parts of \(R\) and \(C_n(q,t)\) agree on all monomials with \(q\)-power at most
\(t\)-power.
The fixed-deficit-\(d\) part of \(C_n(q,t)\) is symmetric in \(q,t\), because
it is the homogeneous total-degree-\(M-d\) part of the symmetric polynomial
\(C_n(q,t)\).  The fixed-deficit-\(d\) part of \(R\) is also symmetric: for each
skeleton \(S\), the interval
\[
  j=\area(S),\ldots,M-d-\area(S)
\]
is invariant under \(j\mapsto M-d-j\).  Two symmetric homogeneous polynomials
of degree \(M-d\) that agree for all exponents \(j\le (M-d)/2\) agree
identically.  Therefore the fixed-deficit-\(d\) pieces are equal for every
\(d\le 2n-8\), and summing over \(d\) proves the theorem.
\end{proof}
In particular, we recover the following low-deficit partition formula of Lee and
Li~\cite{LeeLi2011}.

\begin{corollary}[\cite{LeeLi2011}; partition formula for deficit at most \(n-3\)]
\label{cor:partition-formula}
For \(n\ge 1\),
\[
  \left.C_n(q,t)\right|_{\binom n2-n+3\le \deg_{q,t}\le \binom n2}
  =
  \sum_{\substack{\lambda\text{ a partition}\\ |\lambda|\le n-3}}
  \frac{q^{\binom n2-|\lambda|-\ell(\lambda)+1}t^{\ell(\lambda)}
        -q^{\ell(\lambda)}t^{\binom n2-|\lambda|-\ell(\lambda)+1}}{q-t},
\]
where \(|\lambda|\) is the size of \(\lambda\) and \(\ell(\lambda)\) is the
number of parts of \(\lambda\).
\end{corollary}
\begin{proof}
We use the known \(q,t\)-symmetry of \(C_n(q,t)\) through
Theorem~\ref{thm:qt-catalan-skeleton}.  For \(n\le 2\), the deficit bound
\(\defc\le n-3\) is negative, so the restricted part and the sum are both zero.
For \(n=3\), the bound is
\(\defc=0\), and direct enumeration gives
\[
  q^3+q^2t+qt^2+t^3=\frac{q^4-t^4}{q-t},
\]
which is the contribution of the empty partition.  For \(n=4\), direct
enumeration gives the deficit-zero contribution
\[
  \frac{q^7-t^7}{q-t}
\]
and the deficit-one contribution
\[
  \frac{q^5t-qt^5}{q-t},
\]
corresponding respectively to the partitions \(\varnothing\) and \((1)\).
Assume now that \(n\ge 5\).  Then \(n-3\le 2n-8\), so
Theorem~\ref{thm:qt-catalan-skeleton} applies throughout the required deficit
range.  It remains to identify the special Dyck skeletons of length \(n\) and
deficit at most \(n-3\) with partitions \(\lambda\) of size at most \(n-3\), in
a way that sends
\[
  \defc(S)\mapsto |\lambda|,
  \qquad
  \area(S)\mapsto \ell(\lambda).
\]
Let \(S\) be such a special skeleton.  A full skeleton cannot begin with
\(0,1\): if its maximal initial increasing run is
\(0,1,\ldots,r\) with \(r\ge 1\), then the final entry of that run satisfies the
extractability conditions, a contradiction.  Thus, since \(S\) is a Dyck
sequence, it begins with \(0,0\).
We claim that \(S\) is binary.  If an entry at least \(2\) occurs, then some
entry equal to \(2\) occurs; fix one such occurrence.  The second entry, a
noninitial \(0\), forms a type~\textup{(B)} deficit pair with every later
nonzero entry.  Also, every \(0\) after the first two positions forms a deficit
pair with the chosen \(2\): if that \(0\) is before the chosen \(2\), it forms
a type~\textup{(B)} pair with the chosen \(2\), and if it is after the chosen
\(2\), it forms a type~\textup{(A)} pair with it.  Therefore every one of the
\(n-2\) positions after the first two contributes a distinct deficit pair,
contradicting \(\defc(S)\le n-3\).  Hence \(S\) has only \(0\)s and \(1\)s.
For a binary special skeleton, the deficit pairs are exactly the pairs
\((i,j)\) with
\[
  i<j,
  \qquad
  S_i=0,
  \qquad
  S_j=1,
  \qquad
  i\ne 0.
\]
Thus
\[
  \defc(S)=
  \sum_{j:S_j=1}\#\{i<j:S_i=0,\ i\ne 0\},
  \qquad
  \area(S)=\#\{j:S_j=1\}.
\]
Read \(S\) from left to right.  Each time a \(1\) is encountered, record the
number of preceding \(0\)s, not counting the initial \(0\).  This gives a
weakly increasing list
\[
  c_1\le c_2\le\cdots\le c_m.
\]
Reverse the list to obtain the partition
\(\lambda=(c_m,c_{m-1},\ldots,c_1)\), with the empty list giving the empty
partition.  The number of parts is the number of \(1\)s in \(S\), so
\(\ell(\lambda)=\area(S)\).  The size is
\[
  |\lambda|=c_1+\cdots+c_m=\defc(S).
\]
The construction is injective, because the recorded list and the total length
recover the binary word: first place the initial \(0\), then insert \(c_1\)
additional zeros before the first \(1\), \(c_2-c_1\) additional zeros before the
second \(1\), and so on, and finally append the number of trailing zeros needed
to restore length \(n\).
Conversely, let \(\lambda\) be a partition with \(|\lambda|\le n-3\).  If
\(\lambda\) is empty, take the all-zero skeleton of length \(n\).  Otherwise,
list the parts of \(\lambda\) in weakly increasing order
\[
  c_1\le c_2\le\cdots\le c_m.
\]
Construct a binary word \(W\) by writing one \(1\) for each part and inserting,
immediately before the \(r\)-th \(1\), enough \(0\)s so that exactly \(c_r\)
zeros precede that \(1\) inside \(W\).  This word has \(|\lambda|\) zero-before-one
pairs by construction.
The word \(W\) has length at most \(n-2\).  If it had length \(L\ge n-1\), then
the first \(0\) in \(W\) and the last \(1\) in \(W\), together with each symbol
between them, would give at least \(L-1\ge n-2\) zero-before-one pairs,
contradicting \(|\lambda|\le n-3\).  Now prepend one initial \(0\) and append
trailing \(0\)s until the total length is \(n\).  The result is a binary Dyck
sequence.  No \(0\) is extractable, and every \(1\) has at least two \(0\)s to
its left, so no \(1\) is extractable.  Since at least one trailing \(0\) was
appended, the resulting full skeleton is special.  It maps back to
\(\lambda\), proving surjectivity.
Substituting
\[
  \defc(S)=|\lambda|,
  \qquad
  \area(S)=\ell(\lambda),
  \qquad
  \dinv(S)=\binom n2-|\lambda|-\ell(\lambda)
\]
in Theorem~\ref{thm:qt-catalan-skeleton} gives the displayed formula.
\end{proof}
\begin{remark}[Flat middle coefficients]
\label{rem:flat-middle-coefficients}
The skeleton formula in Theorem~\ref{thm:qt-catalan-skeleton} also gives a
flat middle band.  Let \(n\ge 4\), let \(M=\binom n2\), and let
\(0\le d\le 2n-8\).  Proposition~\ref{prop:area-leq-defc} implies that if
\(S\) is a special Dyck skeleton of deficit \(d\), then
\(\area(S)\le d\) and
\[
  \dinv(S)=M-d-
  \area(S)\ge M-2d.
\]
Thus the interval contribution of \(S\) in
Theorem~\ref{thm:qt-catalan-skeleton} contains every monomial
\(q^j t^{M-d-j}\) with \(d\le j\le M-2d\).  It follows that, in this range, the
coefficient of \(q^j t^{M-d-j}\) is independent of \(j\) for
\(d\le j\le M-2d\), and is equal to the number of special Dyck skeletons of
length \(n\) and deficit \(d\).
\end{remark}
\begin{conjecture}[Full flat-middle range]
\label{conj:middle-coefficients}
Let \(M=\binom n2\).  For any \(n\ge 1\) and any
\(0\le d\le \lfloor M/3\rfloor\), the coefficient of \(q^j t^{M-d-j}\) in
\(C_n(q,t)\) is the same for all integers \(j\) satisfying
\[
  d\le j\le M-2d.
\]
\end{conjecture}

\section{Concluding remarks}
\label{sec:concluding-remarks}

The construction in Section~\ref{sec:full-skeletons} leaves several natural
directions open.  The first is to extend the local maps beyond the seven-window
stage used in Theorem~\ref{thm:qt-catalan-skeleton}.  Computations suggest that
there is a distinguished continuation of the maps
\(\mathrm{East}_{2k+1}\) and \(\mathrm{West}_{2k+1}\) to larger values of
\(k\).  If such a continuation could be made uniform, it would give a
decomposition of all Dyck paths of fixed deficit \(d\le 2n-8\) across the
entire area range.  This would give a completely combinatorial proof of
\(q,t\)-symmetry in this low-deficit range, without appealing to the global
symmetry of \(C_n(q,t)\).

We have computed several higher-order instances of these local maps.  Two of
the phenomena observed for \(k\in\{1,2,3\}\) appear to continue in these
examples: the surrounding conditions seem to force the choice of the East--West
maps, and the resulting forced maps are reversal conjugates.  At present,
however, the computed maps have not revealed a simple enough pattern from which
to state a useful general definition for arbitrary \(k\).

From this viewpoint, the construction of the higher East--West maps appears to
be the main missing ingredient for a fully combinatorial treatment of the range
\(\defc\le 2n-8\).  The exceptional assignment
\[
  \omega_n=(0,\ldots,0,1)
  \longmapsto
  \epsilon_n=(0,0,1,0,\ldots,0,1)
\]
seems to remain the only place in this range where the natural
\(\mathrm{up}/\mathrm{down}\) mechanics must be altered artificially.  The
need for this alteration is better viewed as a peculiarity of \(\omega_n\),
which has no natural ordinary-skeleton target, than as a defect of the full
skeleton \(\epsilon_n\).

Beyond the bound \(\defc\le 2n-8\), such exceptional adjustments appear to
accumulate quickly.  In particular, it becomes difficult even to formulate the
right initial objects, or equivalently to find a useful generalization of
special Dyck skeletons.  We expect another major change in behavior once the
deficit passes \(\binom n2/3\).  If \(M=\binom n2\), then the interval
\(d\le j\le M-2d\) in Conjecture~\ref{conj:middle-coefficients} ceases to be
nonempty once \(d>M/3\), so the same flat-middle phenomenon cannot persist in
that form beyond this threshold.

A second direction is rational \(q,t\)-Catalan combinatorics.  The author
generalized part of the Lee--Li--Loehr chain-decomposition framework to
rational \(q,t\)-Catalan polynomials in the subfamily
\(r\equiv1\pmod{s}\).  In the same spirit, we expect the constructions in this
paper to extend naturally to that subfamily: the Dyck-symmetric-function Schur
positivity, the Dyck-skeleton and Dyck-tableau formula for \(C_n(q,t)\), and
the low-deficit skeleton formula should all have analogues for
\(\mathcal C_{r/s}(q,t)\) when \(r\equiv1\pmod{s}\).  This expectation is based
on the compatibility of that framework with the present one, together with
the behavior of the constructions in examples.  By contrast, extending these
constructions to the full rational case should be substantially harder, since
the special features of the \(r\equiv1\pmod{s}\) case no longer provide the
same direct analogue of the ordinary Dyck-sequence operations used here.

\appendix

\section{Reusable computations and string examples}
\label{app:computations}
This appendix records computational constructions used in the finite
checks and later examples.  The code is included for reproducibility and
to make the displayed examples unambiguous.
These routines do not replace the proofs of the local well-definedness lemmas: the
finite checks used in the proofs in the next appendix add their own
assertions and
range restrictions.
\paragraph{Reproducibility.}
The listings are plain Python programs using only the standard library.  Run
them with ordinary assertion checking enabled; Python's optimized mode disables
\lstinline{assert}.  Some listings are standalone scripts; others should be
appended after the core routines in
Subsection~\ref{subsec:core-code} before execution.
Displayed outputs are complete successful runs or explicitly labeled excerpts.
\subsection{Core Dyck-sequence routines}
\label{subsec:core-code}
The first listing contains the common routines for Dyck sequences, deficit, the
leftmost extraction convention, injection, the local East--West maps, and the
global maps \(\mathrm{up}\) and \(\mathrm{down}\).  The functions
\lstinline{up} and \lstinline{down} return both the image and the local level
\(3,5,7\) used in the construction.  The code separates the basic operations
so that later finite checkers can impose their own hypotheses.
\begin{lstlisting}
from itertools import combinations
from math import comb

def is_Dyck(S):
    S = tuple(S)
    return (
        len(S) > 0
        and S[0] == 0
        and all(isinstance(x, int) and x >= 0 for x in S)
        and all(S[i + 1] <= S[i] + 1 for i in range(len(S) - 1))
    )

def generate_Dycks(n):
    out = []
    def rec(S):
        if len(S) == n:
            out.append(tuple(S))
            return
        for x in range(S[-1] + 2):
            rec(S + [x])
    rec([0])
    return out

def area(S):
    return sum(S)

def dinv(S):
    S = tuple(S)
    return sum(
        1
        for i in range(len(S))
        for j in range(i + 1, len(S))
        if S[i] == S[j] or S[i] == S[j] + 1
    )

def defc(S):
    return comb(len(S), 2) - area(S) - dinv(S)

def find_extractable(S):
    S = tuple(S)
    for j, x in enumerate(S):
        if x == 0:
            continue
        if sum(1 for i in range(j) if S[i] == x - 1) != 1:
            continue
        if j + 1 < len(S) and S[j + 1] > x:
            continue
        return j, x
    return None

def remove_at(S, j):
    S = tuple(S)
    return S[:j] + S[j + 1:]

def is_full_skeleton(S):
    return is_Dyck(S) and find_extractable(S) is None

def epsilon(n):
    return () if n < 4 else tuple([0, 0, 1] + [0] * (n - 4) + [1])

def omega(n):
    return tuple([0] * (n - 1) + [1])

def is_special_skeleton(S):
    S = tuple(S)
    return is_full_skeleton(S) and S != epsilon(len(S))

def inject(S, e):
    S = tuple(S)
    for i, x in enumerate(S):
        if x == e - 1:
            ans = S[:i + 1] + (e,) + S[i + 1:]
            assert is_Dyck(ans)
            return ans
    raise ValueError(f"cannot inject {e} into {S}")

def inject_right_to_left(base, entries):
    out = tuple(base)
    for e in reversed(tuple(entries)):
        out = inject(out, e)
    return out
# Local affine/reverse helpers.

def bk2(a, b):
    return (b, a) if a > b + 1 else (a, b)

def fw2(a, b):
    return (b, a) if b > a + 1 else (a, b)

def bk3(a, b, c):
    if a > b + 1:
        a, b = b, a
    if b > c + 1:
        b, c = c, b
    if a > b + 1:
        a, b = b, a
    return a, b, c
# Local East and West maps.

def East3(W):
    W = tuple(W)
    assert len(W) == 3
    return W if W[1] <= W[2] + 1 else None

def East5(W):
    W = tuple(W)
    assert len(W) == 5
    x_m2, x_m1, x_0, x_1, x_2 = W
    y_m1, y_0 = bk2(x_m1, x_0)
    if x_m1 > x_1 + 1 and y_0 <= x_2 + 1:
        return (x_m2, x_1, y_m1, y_0, x_2)
    if x_m1 <= x_1 + 1 and x_m1 <= x_2 + 1:
        return (x_m2, x_1, x_0, x_m1, x_2)
    return None
_CASE4A = {
    (3, 3, 4, 1, 2): (1, 2, 4, 3, 3),
    (3, 4, 4, 1, 2): (1, 2, 4, 3, 4),
    (4, 3, 4, 1, 2): (1, 2, 4, 4, 3),
    (2, 3, 4, 1, 2): (1, 2, 4, 3, 2),
}
_CASE4B = {
    (3, 3, 4, 2, 1): (2, 1, 4, 3, 3),
    (3, 4, 4, 2, 1): (2, 1, 4, 3, 4),
    (4, 3, 4, 2, 1): (2, 1, 4, 4, 3),
    (2, 3, 4, 2, 1): (2, 1, 4, 3, 2),
}
_CASE4C = {
    (3, 4, 4, 2, 2): (2, 2, 4, 4, 3),
    (3, 4, 5, 2, 2): (2, 2, 5, 4, 3),
}
_CASE4D = {
    (3, 3, 4): lambda o: (2, o, 4, 3, 3),
    (3, 4, 4): lambda o: (2, o, 4, 3, 4),
    (4, 3, 4): lambda o: (2, o, 4, 4, 3),
    (2, 3, 4): lambda o: (2, o, 2, 4, 3),
    (3, 4, 2): lambda o: (2, o, 4, 3, 2),
}

def East7(W):
    W = tuple(W)
    assert len(W) == 7
    x_m3, x_m2, x_m1, x_0, x_1, x_2, x_3 = W
    if x_0 <= x_1 + 1:
        return W
    y_m1, y_0 = bk2(x_m1, x_0)
    if x_m1 > x_1 + 1 and y_0 <= x_2 + 1:
        return (x_m3, x_m2, x_1, y_m1, y_0, x_2, x_3)
    if x_m1 <= x_1 + 1 and x_m1 <= x_2 + 1:
        return (x_m3, x_m2, x_1, x_0, x_m1, x_2, x_3)
    if min(x_m2, x_m1, x_0) > max(x_1, x_2) + 1:
        return (x_m3,) + fw2(x_1, x_2) + bk3(x_m2, x_m1, x_0) + (x_3,)
    shift = max(x_1, x_2) - 2
    reduced = (x_m2 - shift, x_m1 - shift, x_0 - shift,
               x_1 - shift, x_2 - shift)
    for table in (_CASE4A, _CASE4B, _CASE4C):
        if reduced in table:
            return (x_m3,) + tuple(y + shift for y in table[reduced]) + (x_3,)
    if reduced[4] == 2 and reduced[3] <= 0 and reduced[:3] in _CASE4D:
        return (x_m3,) + tuple(y + shift for y in _CASE4D[reduced[:3]](reduced[3])) + (x_3,)
    raise ValueError(f"East7 undefined on {W}")

def rev(W):
    return tuple(reversed(tuple(W)))

def West3(W):
    ans = East3(rev(W))
    return None if ans is None else rev(ans)

def West5(W):
    ans = East5(rev(W))
    return None if ans is None else rev(ans)

def West7(W):
    return rev(East7(rev(W)))

def is_far_apart_decomposable(W):
    W = tuple(W)
    assert len(W) == 7
    indices = list(range(7))
    for p1 in combinations(indices, 2):
        if abs(W[p1[0]] - W[p1[1]]) < 2:
            continue
        r1 = [i for i in indices if i not in p1]
        for p2 in combinations(r1, 2):
            if abs(W[p2[0]] - W[p2[1]]) < 2:
                continue
            r2 = [i for i in r1 if i not in p2]
            for p3 in combinations(r2, 2):
                if abs(W[p3[0]] - W[p3[1]]) >= 2:
                    return True
    return False
# Global up and down maps.

def up(S):
    S = tuple(S)
    n = len(S)
    if S == omega(n):
        return epsilon(n), 3
    if is_full_skeleton(S):
        return inject(S[:-1], S[-1] + 1), 3
    j1, e1 = find_extractable(S)
    C1 = remove_at(S, j1)
    sigma1 = C1 + (e1 - 1,)
    if East3(sigma1[-3:]) is not None:
        ans = inject_right_to_left(sigma1[:-2], (sigma1[-2] + 1, sigma1[-1] + 1))
        return ans, 3
    j2, e2 = find_extractable(C1)
    C2 = remove_at(C1, j2)
    sigma2 = C2 + (e1 - 1, e2 - 1)
    W5 = East5(sigma2[-5:])
    if W5 is not None:
        base = sigma2[:-5] + W5[:2]
        ans = inject_right_to_left(base, tuple(x + 1 for x in W5[2:]))
        return ans, 5
    j3, e3 = find_extractable(C2)
    C3 = remove_at(C2, j3)
    sigma3 = C3 + (e1 - 1, e2 - 1, e3 - 1)
    W7 = sigma3[-7:]
    assert not is_far_apart_decomposable(W7)
    E7 = East7(W7)
    new_sigma3 = sigma3[:-7] + E7
    ans = inject_right_to_left(new_sigma3[:-4], tuple(x + 1 for x in new_sigma3[-4:]))
    return ans, 7

def down(S):
    S = tuple(S)
    n = len(S)
    if S == epsilon(n):
        return omega(n), 3
    j1, f1 = find_extractable(S)
    D1 = remove_at(S, j1)
    candidate = D1 + (f1 - 1,)
    if find_extractable(candidate) is None:
        assert is_Dyck(candidate)
        return candidate, 3
    j2, f2 = find_extractable(D1)
    D2 = remove_at(D1, j2)
    tau1 = D2 + (f1 - 1, f2 - 1)
    if West3(tau1[-3:]) is not None:
        return inject(tau1[:-1], tau1[-1] + 1), 3
    j3, f3 = find_extractable(D2)
    D3 = remove_at(D2, j3)
    tau2 = D3 + (f1 - 1, f2 - 1, f3 - 1)
    W5 = West5(tau2[-5:])
    if W5 is not None:
        base = tau2[:-5] + W5[:3]
        ans = inject_right_to_left(base, tuple(x + 1 for x in W5[3:]))
        return ans, 5
    j4, f4 = find_extractable(D3)
    D4 = remove_at(D3, j4)
    tau3 = D4 + (f1 - 1, f2 - 1, f3 - 1, f4 - 1)
    W7 = tau3[-7:]
    assert not is_far_apart_decomposable(W7)
    new_tau3 = tau3[:-7] + West7(W7)
    ans = inject_right_to_left(new_tau3[:-3], tuple(x + 1 for x in new_tau3[-3:]))
    return ans, 7
\end{lstlisting}
\subsection{Generating the lower-half strings}
\label{subsec:string-code}
For fixed \(n\) and \(d\le 2n-8\), set
\[
  \ell=\left\lfloor\frac{\binom n2-d}{2}\right\rfloor .
\]
The next routine constructs the lower half of the \(\mathrm{up}\)-string
decomposition.  It starts from all special Dyck skeletons of length \(n\), deficit
\(d\), and area at most \(\ell\), then repeatedly applies \lstinline{up} until
area \(\ell\) is reached.  The final two assertions are finite coverage checks
for the generated data; in the paper, their general validity is the content of
Proposition~\ref{prop:decomposition}.
\begin{lstlisting}
def make_strings(n, d):
    ell = (comb(n, 2) - d) // 2
    all_dyck = [S for S in generate_Dycks(n) if defc(S) == d]
    target = {S for S in all_dyck if area(S) <= ell}
    starts = sorted(
        [S for S in target if is_special_skeleton(S)],
        key=lambda S: (area(S), S),
    )
    strings = []
    levels = []
    for start in starts:
        chain = [start]
        current = start
        while area(current) < ell:
            nxt, level = up(current)
            assert defc(nxt) == d
            assert area(nxt) == area(current) + 1
            chain.append(nxt)
            levels.append((current, nxt, level))
            current = nxt
        strings.append(tuple(chain))
    covered = [S for chain in strings for S in chain]
    assert set(covered) == target
    assert len(covered) == len(set(covered))
    return tuple(strings), tuple(levels)
\end{lstlisting}
The routine also records the local level used at each step.  These
records are useful for coloring or annotating examples, but the string
decomposition itself only requires the sequence of Dyck words.
\subsection{An example decomposition}
\label{subsec:string-example}
For a larger illustration, take \(n=9\) and \(d=10\).  Then
\(M=\binom92=36\) and
\[
  \ell=\left\lfloor\frac{36-10}{2}\right\rfloor=13 .
\]
The special skeletons of deficit \(10\) and area at most \(13\)
generate 31 lower-half strings, containing 274 Dyck sequences in
all.  The rows below are indexed by area and the string columns are
displayed in blocks of five.  Blank cells mean that the corresponding
string has not yet started at that area.  Within a fixed string column, a
colored consecutive pair records one exceptional \(\mathrm{up}\) step: the
smaller-area entry is the source and the next larger-area entry is the target.
Blue marks an \(\mathrm{East}_5\) step, and orange marks an
\(\mathrm{East}_7\) step.
\par\medskip
\begingroup
\scriptsize
\newcommand{\dseq}[1]{\texttt{[#1]}}
\newcommand{\rseq}[1]{\textcolor{red}{\dseq{#1}}}
\newcommand{\bseq}[1]{\textcolor{blue}{\dseq{#1}}}
\newcommand{\gseq}[1]{\textcolor{orange}{\dseq{#1}}}
\setlength{\tabcolsep}{1.2pt}
\par\medskip
{\scriptsize%
\setlength{\tabcolsep}{1.2pt}%
\begin{tabular}{r|lllll}
\multicolumn{1}{r|}{\(\area\)} & $\mathrm{string}\,1$ & $\mathrm{string}\,2$ & $\mathrm{string}\,3$ & $\mathrm{string}\,4$ & $\mathrm{string}\,5$\\
\hline
2 & \dseq{0,0,0,0,0,0,1,1,0} & \dseq{0,0,0,0,0,1,0,0,1} &  &  & \\
3 & \dseq{0,1,0,0,0,0,0,1,1} & \dseq{0,0,0,0,0,1,2,0,0} & \dseq{0,0,0,0,1,1,0,1,0} & \dseq{0,0,0,1,0,0,1,1,0} & \dseq{0,0,0,1,0,1,0,0,1}\\
4 & \dseq{0,1,2,0,0,0,0,0,1} & \dseq{0,1,0,0,0,0,1,2,0} & \dseq{0,1,0,0,0,1,1,0,1} & \dseq{0,1,0,0,1,0,0,1,1} & \dseq{0,0,0,1,2,0,1,0,0}\\
5 & \dseq{0,1,2,2,0,0,0,0,0} & \bseq{0,1,1,0,0,0,0,1,2} & \dseq{0,1,2,0,0,0,1,1,0} & \dseq{0,1,2,0,0,1,0,0,1} & \dseq{0,1,0,0,1,2,0,1,0}\\
6 & \dseq{0,1,1,2,2,0,0,0,0} & \bseq{0,1,2,3,0,0,0,0,0} & \dseq{0,1,1,2,0,0,0,1,1} & \dseq{0,1,2,2,0,0,1,0,0} & \dseq{0,1,1,0,0,1,2,0,1}\\
7 & \dseq{0,1,1,1,2,2,0,0,0} & \dseq{0,1,1,2,3,0,0,0,0} & \dseq{0,1,2,1,2,0,0,0,1} & \dseq{0,1,1,2,2,0,0,1,0} & \dseq{0,1,2,1,0,0,1,2,0}\\
8 & \dseq{0,1,1,1,1,2,2,0,0} & \dseq{0,1,1,1,2,3,0,0,0} & \dseq{0,1,2,2,1,2,0,0,0} & \dseq{0,1,1,1,2,2,0,0,1} & \bseq{0,1,1,2,1,0,0,1,2}\\
9 & \dseq{0,1,1,1,1,1,2,2,0} & \dseq{0,1,1,1,1,2,3,0,0} & \dseq{0,1,1,2,2,1,2,0,0} & \dseq{0,1,2,1,1,2,2,0,0} & \bseq{0,1,2,3,2,1,0,0,0}\\
10 & \gseq{0,1,1,1,1,1,1,2,2} & \dseq{0,1,1,1,1,1,2,3,0} & \dseq{0,1,1,1,2,2,1,2,0} & \dseq{0,1,1,2,1,1,2,2,0} & \dseq{0,1,1,2,3,2,1,0,0}\\
11 & \gseq{0,1,2,3,3,1,1,0,0} & \gseq{0,1,1,1,1,1,1,2,3} & \bseq{0,1,1,1,1,2,2,1,2} & \gseq{0,1,1,1,2,1,1,2,2} & \dseq{0,1,1,1,2,3,2,1,0}\\
12 & \dseq{0,1,1,2,3,3,1,1,0} & \gseq{0,1,2,3,4,1,1,0,0} & \bseq{0,1,2,3,1,1,2,2,0} & \gseq{0,1,2,3,3,2,1,0,0} & \dseq{0,1,1,1,1,2,3,2,1}\\
13 & \dseq{0,1,1,1,2,3,3,1,1} & \dseq{0,1,1,2,3,4,1,1,0} & \dseq{0,1,1,2,3,1,1,2,2} & \dseq{0,1,1,2,3,3,2,1,0} & \dseq{0,1,2,1,1,1,2,3,2}\\
\end{tabular}%
}
\par\medskip
{\scriptsize%
\setlength{\tabcolsep}{1.2pt}%
\begin{tabular}{r|lllll}
\multicolumn{1}{r|}{\(\area\)} & $\mathrm{string}\,6$ & $\mathrm{string}\,7$ & $\mathrm{string}\,8$ & $\mathrm{string}\,9$ & $\mathrm{string}\,10$\\
\hline
3 & \dseq{0,0,1,0,0,0,1,0,1} &  &  &  & \\
4 & \dseq{0,0,1,2,0,0,0,1,0} & \dseq{0,0,0,1,1,0,1,1,0} & \dseq{0,0,0,1,1,1,0,0,1} & \dseq{0,0,0,1,1,2,0,0,0} & \dseq{0,0,1,0,0,1,1,1,0}\\
5 & \dseq{0,1,0,1,2,0,0,0,1} & \dseq{0,1,0,0,1,1,0,1,1} & \dseq{0,0,0,1,2,1,1,0,0} & \dseq{0,1,0,0,1,1,2,0,0} & \dseq{0,1,0,1,0,0,1,1,1}\\
6 & \dseq{0,1,2,0,1,2,0,0,0} & \dseq{0,1,2,0,0,1,1,0,1} & \dseq{0,1,0,0,1,2,1,1,0} & \dseq{0,1,1,0,0,1,1,2,0} & \dseq{0,1,2,0,1,0,0,1,1}\\
7 & \dseq{0,1,1,2,0,1,2,0,0} & \dseq{0,1,2,2,0,0,1,1,0} & \dseq{0,1,1,0,0,1,2,1,1} & \bseq{0,1,1,1,0,0,1,1,2} & \dseq{0,1,2,2,0,1,0,0,1}\\
8 & \dseq{0,1,1,1,2,0,1,2,0} & \dseq{0,1,1,2,2,0,0,1,1} & \dseq{0,1,2,1,0,0,1,2,1} & \bseq{0,1,2,3,1,0,0,1,0} & \dseq{0,1,2,2,2,0,1,0,0}\\
9 & \bseq{0,1,1,1,1,2,0,1,2} & \dseq{0,1,2,1,2,2,0,0,1} & \dseq{0,1,2,2,1,0,0,1,2} & \dseq{0,1,1,2,3,1,0,0,1} & \dseq{0,1,1,2,2,2,0,1,0}\\
10 & \bseq{0,1,2,3,1,1,2,0,0} & \dseq{0,1,2,2,1,2,2,0,0} & \dseq{0,1,2,3,2,1,0,0,1} & \dseq{0,1,2,1,2,3,1,0,0} & \dseq{0,1,1,1,2,2,2,0,1}\\
11 & \dseq{0,1,1,2,3,1,1,2,0} & \dseq{0,1,1,2,2,1,2,2,0} & \dseq{0,1,2,2,3,2,1,0,0} & \dseq{0,1,1,2,1,2,3,1,0} & \dseq{0,1,2,1,1,2,2,2,0}\\
12 & \bseq{0,1,1,1,2,3,1,1,2} & \gseq{0,1,1,1,2,2,1,2,2} & \dseq{0,1,1,2,2,3,2,1,0} & \dseq{0,1,1,1,2,1,2,3,1} & \bseq{0,1,1,2,1,1,2,2,2}\\
13 & \bseq{0,1,2,3,1,2,3,1,0} & \gseq{0,1,2,3,3,2,2,0,0} & \dseq{0,1,1,1,2,2,3,2,1} & \dseq{0,1,2,1,1,2,1,2,3} & \bseq{0,1,2,3,3,1,1,2,0}\\
\end{tabular}%
}
\par\medskip
{\scriptsize%
\setlength{\tabcolsep}{1.2pt}%
\begin{tabular}{r|lllll}
\multicolumn{1}{r|}{\(\area\)} & $\mathrm{string}\,11$ & $\mathrm{string}\,12$ & $\mathrm{string}\,13$ & $\mathrm{string}\,14$ & $\mathrm{string}\,15$\\
\hline
4 & \dseq{0,0,1,0,0,1,2,0,0} & \dseq{0,0,1,0,1,0,1,0,1} & \dseq{0,0,1,1,0,0,0,1,1} &  & \\
5 & \dseq{0,1,0,1,0,0,1,2,0} & \dseq{0,0,1,2,0,1,0,1,0} & \dseq{0,0,1,2,1,0,0,0,1} & \dseq{0,0,0,1,1,1,1,1,0} & \dseq{0,0,1,0,1,1,1,0,1}\\
6 & \bseq{0,1,1,0,1,0,0,1,2} & \dseq{0,1,0,1,2,0,1,0,1} & \dseq{0,0,1,2,2,1,0,0,0} & \dseq{0,1,0,0,1,1,1,1,1} & \dseq{0,0,1,2,0,1,1,1,0}\\
7 & \bseq{0,1,2,3,0,1,0,0,0} & \dseq{0,1,2,0,1,2,0,1,0} & \dseq{0,1,0,1,2,2,1,0,0} & \dseq{0,1,2,0,0,1,1,1,1} & \dseq{0,1,0,1,2,0,1,1,1}\\
8 & \dseq{0,1,1,2,3,0,1,0,0} & \dseq{0,1,1,2,0,1,2,0,1} & \dseq{0,1,1,0,1,2,2,1,0} & \dseq{0,1,2,2,0,0,1,1,1} & \dseq{0,1,2,0,1,2,0,1,1}\\
9 & \dseq{0,1,1,1,2,3,0,1,0} & \dseq{0,1,2,1,2,0,1,2,0} & \dseq{0,1,1,1,0,1,2,2,1} & \dseq{0,1,2,2,2,0,0,1,1} & \dseq{0,1,2,2,0,1,2,0,1}\\
10 & \dseq{0,1,1,1,1,2,3,0,1} & \bseq{0,1,1,2,1,2,0,1,2} & \dseq{0,1,2,1,1,0,1,2,2} & \dseq{0,1,2,2,2,2,0,0,1} & \dseq{0,1,2,2,2,0,1,2,0}\\
11 & \dseq{0,1,2,1,1,1,2,3,0} & \bseq{0,1,2,3,2,1,2,0,0} & \dseq{0,1,2,3,1,1,0,1,2} & \dseq{0,1,2,2,2,2,2,0,0} & \bseq{0,1,1,2,2,2,0,1,2}\\
12 & \gseq{0,1,1,2,1,1,1,2,3} & \dseq{0,1,1,2,3,2,1,2,0} & \dseq{0,1,2,3,3,1,1,0,1} & \dseq{0,1,1,2,2,2,2,2,0} & \bseq{0,1,2,3,2,2,2,0,0}\\
13 & \gseq{0,1,2,3,4,1,1,0,1} & \dseq{0,1,1,1,2,3,2,1,2} & \dseq{0,1,2,2,3,3,1,1,0} & \dseq{0,1,1,1,2,2,2,2,2} & \dseq{0,1,1,2,3,2,2,2,0}\\
\end{tabular}%
}
\par\medskip
{\scriptsize%
\setlength{\tabcolsep}{1.2pt}%
\begin{tabular}{r|lllll}
\multicolumn{1}{r|}{\(\area\)} & $\mathrm{string}\,16$ & $\mathrm{string}\,17$ & $\mathrm{string}\,18$ & $\mathrm{string}\,19$ & $\mathrm{string}\,20$\\
\hline
5 & \dseq{0,0,1,0,1,2,1,0,0} & \dseq{0,0,1,1,0,1,0,1,1} & \dseq{0,0,1,1,0,1,2,0,0} & \dseq{0,0,1,1,2,0,0,1,0} & \\
6 & \dseq{0,1,0,1,0,1,2,1,0} & \dseq{0,0,1,2,1,0,1,0,1} & \dseq{0,1,0,1,1,0,1,2,0} & \dseq{0,1,0,1,1,2,0,0,1} & \dseq{0,0,1,1,0,1,1,1,1}\\
7 & \dseq{0,1,1,0,1,0,1,2,1} & \dseq{0,0,1,2,2,1,0,1,0} & \bseq{0,1,1,0,1,1,0,1,2} & \dseq{0,1,2,0,1,1,2,0,0} & \dseq{0,0,1,2,1,0,1,1,1}\\
8 & \dseq{0,1,2,1,0,1,0,1,2} & \dseq{0,1,0,1,2,2,1,0,1} & \bseq{0,1,2,3,0,1,1,0,0} & \dseq{0,1,1,2,0,1,1,2,0} & \dseq{0,0,1,2,2,1,0,1,1}\\
9 & \dseq{0,1,2,3,1,0,1,0,1} & \dseq{0,1,2,0,1,2,2,1,0} & \dseq{0,1,1,2,3,0,1,1,0} & \bseq{0,1,1,1,2,0,1,1,2} & \dseq{0,0,1,2,2,2,1,0,1}\\
10 & \dseq{0,1,2,2,3,1,0,1,0} & \dseq{0,1,1,2,0,1,2,2,1} & \dseq{0,1,1,1,2,3,0,1,1} & \bseq{0,1,2,3,1,2,0,1,0} & \dseq{0,0,1,2,2,2,2,1,0}\\
11 & \dseq{0,1,1,2,2,3,1,0,1} & \dseq{0,1,2,1,2,0,1,2,2} & \dseq{0,1,2,1,1,2,3,0,1} & \dseq{0,1,1,2,3,1,2,0,1} & \dseq{0,1,0,1,2,2,2,2,1}\\
12 & \dseq{0,1,2,1,2,2,3,1,0} & \dseq{0,1,2,3,1,2,0,1,2} & \dseq{0,1,2,2,1,1,2,3,0} & \dseq{0,1,2,1,2,3,1,2,0} & \dseq{0,1,2,0,1,2,2,2,2}\\
13 & \dseq{0,1,1,2,1,2,2,3,1} & \dseq{0,1,2,3,3,1,2,0,1} & \dseq{0,1,1,2,2,1,1,2,3} & \dseq{0,1,1,2,1,2,3,1,2} & \dseq{0,1,2,3,0,1,2,2,2}\\
\end{tabular}%
}
\par\medskip
{\scriptsize%
\setlength{\tabcolsep}{1.2pt}%
\begin{tabular}{r|lllll}
\multicolumn{1}{r|}{\(\area\)} & $\mathrm{string}\,21$ & $\mathrm{string}\,22$ & $\mathrm{string}\,23$ & $\mathrm{string}\,24$ & $\mathrm{string}\,25$\\
\hline
6 & \dseq{0,0,1,1,1,1,2,0,0} & \dseq{0,0,1,1,1,2,0,1,0} & \dseq{0,0,1,1,2,0,1,1,0} & \dseq{0,0,1,1,2,1,0,0,1} & \\
7 & \dseq{0,1,0,1,1,1,1,2,0} & \dseq{0,1,0,1,1,1,2,0,1} & \dseq{0,1,0,1,1,2,0,1,1} & \dseq{0,0,1,2,1,2,1,0,0} & \dseq{0,0,1,1,1,1,2,1,0}\\
8 & \bseq{0,1,1,0,1,1,1,1,2} & \dseq{0,1,2,0,1,1,1,2,0} & \dseq{0,1,2,0,1,1,2,0,1} & \dseq{0,1,0,1,2,1,2,1,0} & \dseq{0,1,0,1,1,1,1,2,1}\\
9 & \bseq{0,1,2,3,0,1,1,1,0} & \bseq{0,1,1,2,0,1,1,1,2} & \dseq{0,1,2,2,0,1,1,2,0} & \dseq{0,1,1,0,1,2,1,2,1} & \dseq{0,1,2,0,1,1,1,1,2}\\
10 & \dseq{0,1,1,2,3,0,1,1,1} & \bseq{0,1,2,3,2,0,1,1,0} & \bseq{0,1,1,2,2,0,1,1,2} & \dseq{0,1,2,1,0,1,2,1,2} & \dseq{0,1,2,3,0,1,1,1,1}\\
11 & \dseq{0,1,2,1,2,3,0,1,1} & \dseq{0,1,1,2,3,2,0,1,1} & \bseq{0,1,2,3,2,2,0,1,0} & \dseq{0,1,2,3,1,0,1,2,1} & \dseq{0,1,2,2,3,0,1,1,1}\\
12 & \dseq{0,1,2,2,1,2,3,0,1} & \dseq{0,1,2,1,2,3,2,0,1} & \dseq{0,1,1,2,3,2,2,0,1} & \dseq{0,1,2,2,3,1,0,1,2} & \dseq{0,1,2,2,2,3,0,1,1}\\
13 & \dseq{0,1,2,2,2,1,2,3,0} & \dseq{0,1,2,2,1,2,3,2,0} & \dseq{0,1,2,1,2,3,2,2,0} & \dseq{0,1,2,3,2,3,1,0,1} & \dseq{0,1,2,2,2,2,3,0,1}\\
\end{tabular}%
}
\par\medskip
{\scriptsize%
\setlength{\tabcolsep}{1.2pt}%
\begin{tabular}{r|lllll}
\multicolumn{1}{r|}{\(\area\)} & $\mathrm{string}\,26$ & $\mathrm{string}\,27$ & $\mathrm{string}\,28$ & $\mathrm{string}\,29$ & $\mathrm{string}\,30$\\
\hline
7 & \dseq{0,0,1,1,1,2,1,0,1} & \dseq{0,0,1,1,2,1,0,1,1} &  &  & \\
8 & \dseq{0,0,1,2,1,1,2,1,0} & \dseq{0,0,1,2,1,2,1,0,1} & \dseq{0,0,1,1,1,1,2,1,1} & \dseq{0,0,1,1,2,2,1,1,0} & \\
9 & \dseq{0,1,0,1,2,1,1,2,1} & \dseq{0,0,1,2,2,1,2,1,0} & \dseq{0,0,1,2,1,1,1,2,1} & \dseq{0,1,0,1,1,2,2,1,1} & \dseq{0,0,1,1,2,1,2,1,1}\\
10 & \dseq{0,1,2,0,1,2,1,1,2} & \dseq{0,1,0,1,2,2,1,2,1} & \dseq{0,0,1,2,2,1,1,1,2} & \dseq{0,1,2,0,1,1,2,2,1} & \dseq{0,0,1,2,1,2,1,2,1}\\
11 & \dseq{0,1,2,3,0,1,2,1,1} & \dseq{0,1,2,0,1,2,2,1,2} & \dseq{0,0,1,2,3,2,1,1,1} & \dseq{0,1,2,2,0,1,1,2,2} & \dseq{0,0,1,2,2,1,2,1,2}\\
12 & \dseq{0,1,2,2,3,0,1,2,1} & \dseq{0,1,2,3,0,1,2,2,1} & \dseq{0,0,1,2,2,3,2,1,1} & \dseq{0,1,2,3,2,0,1,1,2} & \dseq{0,0,1,2,3,2,1,2,1}\\
13 & \dseq{0,1,2,2,2,3,0,1,2} & \dseq{0,1,2,2,3,0,1,2,2} & \dseq{0,0,1,2,2,2,3,2,1} & \dseq{0,1,2,3,3,2,0,1,1} & \dseq{0,0,1,2,2,3,2,1,2}\\
\end{tabular}%
}
\par\medskip
{\scriptsize%
\setlength{\tabcolsep}{1.2pt}%
\begin{tabular}{r|l}
\multicolumn{1}{r|}{\(\area\)} & $\mathrm{string}\,31$\\
\hline
10 & \dseq{0,0,1,1,2,2,2,1,1}\\
11 & \dseq{0,0,1,2,1,2,2,2,1}\\
12 & \dseq{0,0,1,2,2,1,2,2,2}\\
13 & \dseq{0,0,1,2,3,2,1,2,2}\\
\end{tabular}%
}
\endgroup
\par\medskip
Each downward column is obtained by repeated application of
\(\mathrm{up}\), and every entry in the display has deficit \(10\).
The display is large enough to contain both five-window and seven-window
local moves while still fitting into a fixed \((n,d)=(9,10)\) slice of
the decomposition.

\section{Local well-definedness proofs}
\label{app:local-proofs}
This appendix supplies the local well-definedness inputs for
Propositions~\ref{prop:up-wd} and~\ref{prop:down-wd}.  Throughout the appendix
we fix
\[
  n\ge 4,
  \qquad
  d\le 2n-8,
  \qquad
  M=\binom n2,
  \qquad
  \ell=\left\lfloor\frac{M-d}{2}\right\rfloor,
\]
and every Dyck sequence under discussion has length \(n\) and deficit \(d\).
The local lemmas were stated in the main body as the four local lemmas,
Lemmas~\ref{lem:skeleton-wd}, \ref{lem:extractions-wd},
\ref{lem:east7-wd}, and~\ref{lem:positions-wd}.
The proof first handles the short cases \(4\le n\le 7\) by a finite check and
then treats \(n\ge 8\) by the arguments below.  For \(n\ge 8\), the proofs use
direct deficit-pair counts on Dyck sequences and suffix-corrected lower-bound
counts on extraction-generated intermediate words.  Lemma~\ref{lem:east7-wd}
uses the seven-window reduction, and Lemma~\ref{lem:positions-wd} uses the
position-bound argument and its finite subchecks.  Thus this finite check is
not a replacement for the
subsequent proofs.  It covers only the short-range branch-prefix cases left
after those proofs are separated out.
\subsection{The finite residual check}
\label{subsec:finite-residual-check}
The residual checker enumerates all Dyck sequences of lengths
\(4,5,6,7\), computes their deficit using the pair-count formula of
Proposition~\ref{prop:deficit-pair-count}, discards only those with
\(\defc>2n-8\), and then applies the exact area bounds in the local lemmas:
\(\area(x)\le \ell-1\) for the \(\mathrm{up}\) side and
\(\area(y)\le \ell\) for the \(\mathrm{down}\) side.  Although
Appendix~\ref{app:computations}, Subsection~\ref{subsec:core-code}, contains
reusable implementations of the full
maps, the checker below repeats only the tests needed for this verification.
This keeps the verification independent of the string-generation and display
code.
The check verifies branch prefixes only, through the following four obligations
in the residual range.
\begin{enumerate}[label=\textup{(\arabic*)}]
  \item The skeleton branches in Lemma~\ref{lem:skeleton-wd} return Dyck
  sequences of the correct length.
  \item Every extraction called along a residual branch prefix exists, as
  required by Lemma~\ref{lem:extractions-wd}.
  \item No residual input reaches \(\mathrm{East}_7\) or
  \(\mathrm{West}_7\).  Thus Lemma~\ref{lem:east7-wd} is vacuous for
  \(4\le n\le 7\).
  \item The extraction positions that occur before the branch stops satisfy the
  relevant bounds in Lemma~\ref{lem:positions-wd}.  In the residual range, the
  only nontrivial five-window branches that are reached are Case~2b and its
  West analogue.
\end{enumerate}
If any obligation fails, the program raises an assertion error.  Its successful
run is printed after the listing.
\begin{lstlisting}
from collections import Counter
from math import comb

def stop(message):
    raise AssertionError(message)

def is_dyck_sequence(seq):
    return (
        bool(seq)
        and seq[0] == 0
        and all(x >= 0 for x in seq)
        and all(seq[i + 1] <= seq[i] + 1
                for i in range(len(seq) - 1))
    )

def deficit_and_area(seq):
    first_index = {}
    for i, value in enumerate(seq):
        first_index.setdefault(value, i)
    deficit = 0
    for i, left in enumerate(seq):
        for right in seq[i + 1:]:
            if left > right + 1:
                deficit += 1
            elif left < right and first_index[left] != i:
                deficit += 1
    return deficit, sum(seq)

def generate_dyck_sequences(length):
    sequences = []
    def extend(prefix):
        if len(prefix) == length:
            sequences.append(prefix)
            return
        for next_value in range(prefix[-1] + 2):
            extend(prefix + (next_value,))
    extend((0,))
    return sequences

def leftmost_extractable(seq):
    for index, value in enumerate(seq):
        has_parent = sum(x == value - 1 for x in seq[:index]) == 1
        next_ok = index == len(seq) - 1 or seq[index + 1] <= value
        if value > 0 and has_parent and next_ok:
            return index, value
    return None

def remove_index(seq, index):
    return seq[:index] + seq[index + 1:]

def is_full_skeleton(seq):
    return is_dyck_sequence(seq) and leftmost_extractable(seq) is None

def almost_zero_sequence(length):
    return (0,) * (length - 1) + (1,)

def excluded_skeleton(length):
    return (0, 0, 1) + (0,) * (length - 4) + (1,)

def is_special_skeleton(seq):
    return is_full_skeleton(seq) and seq != excluded_skeleton(len(seq))

def inject_after_first_parent(seq, value):
    for index, entry in enumerate(seq):
        if entry == value - 1:
            result = seq[:index + 1] + (value,) + seq[index + 1:]
            if is_dyck_sequence(result):
                return result
            stop(("skeleton injection produced non-Dyck",
                  seq, value, result))
    stop(("skeleton injection failed", seq, value))

def east3_applies(window3):
    _, x0, x1 = window3
    return x0 <= x1 + 1

def west3_applies(window3):
    return east3_applies(tuple(reversed(window3)))

def east5_case2b_applies(window5):
    _, x_minus1, x0, x1, x2 = window5
    return (
        x0 > x1 + 1
        and x_minus1 <= x1 + 1
        and x_minus1 <= x2 + 1
    )

def west5_case2b_applies(window5):
    return east5_case2b_applies(tuple(reversed(window5)))

def check_up_prefix(seq, length, deficit, half_area_limit):
    if seq == almost_zero_sequence(length):
        return "up special"
    if is_full_skeleton(seq):
        result = inject_after_first_parent(seq[:-1], seq[-1] + 1)
        if len(result) != length:
            stop(("up skeleton changed length", seq, result))
        return "up skeleton"
    first = leftmost_extractable(seq)
    if first is None:
        stop(("extraction lemma: up first extraction failed",
              length, deficit, half_area_limit, seq))
    index1, value1 = first
    child1 = remove_index(seq, index1)
    word1 = child1 + (value1 - 1,)
    if east3_applies(word1[-3:]):
        if index1 >= length - 2:
            stop(("position lemma: up/East3 position", seq, index1))
        return "up East3"
    second = leftmost_extractable(child1)
    if second is None:
        stop(("extraction lemma: up second extraction failed",
              length, deficit, half_area_limit, seq, child1))
    index2, value2 = second
    child2 = remove_index(child1, index2)
    word2 = child2 + (value1 - 1, value2 - 1)
    if not (index1 < length - 3 and index2 < len(child1) - 3):
        stop(("position lemma: up/East5 position",
              seq, index1, child1, index2))
    if not east5_case2b_applies(word2[-5:]):
        stop(("seven-window lemma: up would reach East7",
              length, deficit, half_area_limit, seq, word2[-5:]))
    return "up East5 case 2b"

def check_down_prefix(seq, length, deficit, half_area_limit):
    if seq == excluded_skeleton(length):
        return "down special"
    first = leftmost_extractable(seq)
    if first is None:
        stop(("extraction lemma: down first extraction failed",
              length, deficit, half_area_limit, seq))
    index1, value1 = first
    child1 = remove_index(seq, index1)
    skeleton_candidate = child1 + (value1 - 1,)
    if is_full_skeleton(skeleton_candidate):
        if len(skeleton_candidate) != length:
            stop(("down skeleton changed length", seq, skeleton_candidate))
        return "down skeleton"
    second = leftmost_extractable(child1)
    if second is None:
        stop(("extraction lemma: down second extraction failed",
              length, deficit, half_area_limit, seq, child1))
    index2, value2 = second
    child2 = remove_index(child1, index2)
    word2 = child2 + (value1 - 1, value2 - 1)
    if west3_applies(word2[-3:]):
        if not (index1 < length - 1 and index2 < len(child1) - 1):
            stop(("position lemma: down/West3 position",
                  seq, index1, child1, index2))
        return "down West3"
    third = leftmost_extractable(child2)
    if third is None:
        stop(("extraction lemma: down third extraction failed",
              length, deficit, half_area_limit, seq, child2))
    index3, value3 = third
    child3 = remove_index(child2, index3)
    word3 = child3 + (value1 - 1, value2 - 1, value3 - 1)
    if not (
        index1 < length - 2
        and index2 < len(child1) - 2
        and index3 < len(child2) - 2
    ):
        stop(("position lemma: down/West5 position",
              seq, index1, child1, index2, child2, index3))
    if not west5_case2b_applies(word3[-5:]):
        stop(("seven-window lemma: down would reach West7",
              length, deficit, half_area_limit, seq, word3[-5:]))
    return "down West5 case 2b"

def main():
    up_counts = Counter()
    down_counts = Counter()
    by_length = {
        length: {"up": Counter(), "down": Counter()}
        for length in range(4, 8)
    }
    for length in range(4, 8):
        for seq in generate_dyck_sequences(length):
            deficit, area = deficit_and_area(seq)
            if deficit > 2 * length - 8:
                continue
            half_area_limit = (comb(length, 2) - deficit) // 2
            if area <= half_area_limit - 1:
                label = check_up_prefix(
                    seq, length, deficit, half_area_limit)
                up_counts[label] += 1
                by_length[length]["up"][label] += 1
            if area <= half_area_limit and not is_special_skeleton(seq):
                label = check_down_prefix(
                    seq, length, deficit, half_area_limit)
                down_counts[label] += 1
                by_length[length]["down"][label] += 1
    print("EverythingOkay = True")
    print("up counts  ", dict(up_counts))
    print("down counts", dict(down_counts))
    print()
    for length in range(4, 8):
        print(f"n={length}")
        print("  up:  ", dict(by_length[length]["up"]))
        print("  down:", dict(by_length[length]["down"]))
    print()
    print("No East7 or West7 branch was reached for 4 <= n <= 7.")
if __name__ == "__main__":
    main()
\end{lstlisting}
A successful run prints:
\begin{lstlisting}
EverythingOkay = True
up counts   {'up skeleton': 42, 'up East3': 152,
             'up special': 2, 'up East5 case 2b': 4}
down counts {'down skeleton': 42, 'down West3': 152,
             'down special': 2, 'down West5 case 2b': 4}
n=4
  up:   {'up skeleton': 1, 'up East3': 2}
  down: {'down skeleton': 1, 'down West3': 2}
n=5
  up:   {'up skeleton': 4, 'up East3': 9}
  down: {'down skeleton': 4, 'down West3': 9}
n=6
  up:   {'up skeleton': 11, 'up special': 1, 'up East3': 32}
  down: {'down special': 1, 'down skeleton': 11, 'down West3': 32}
n=7
  up:   {'up skeleton': 26, 'up special': 1,
         'up East3': 109, 'up East5 case 2b': 4}
  down: {'down skeleton': 26, 'down special': 1,
         'down West3': 109, 'down West5 case 2b': 4}
No East7 or West7 branch was reached for 4 <= n <= 7.
\end{lstlisting}

The generator exhausts the residual finite domain because it recursively lists
all nonnegative words beginning with \(0\) and satisfying the Dyck step
condition.  The labels in the output show that every retained input reaches one
of the accepted branch prefixes and that neither the seven-window East branch
nor the seven-window West branch is ever reached in the residual range.

The same residual data can also be displayed as lower-half
\(\mathrm{up}\)-strings.  In each table a string is a column and
each row has fixed area; blank cells mean that the string has not yet
started at that area.  The colors mark exceptional cases: red is the full Dyck
skeleton that is not special.  Within a fixed string column, a blue consecutive
pair records an exceptional \(\mathrm{up}\) step: the smaller-area entry is the
source and the next larger-area entry is the target of an \(\mathrm{East}_5\)
step.
\begingroup
\scriptsize
\newcommand{\dseq}[1]{\texttt{[#1]}}
\newcommand{\rseq}[1]{\textcolor{red}{\dseq{#1}}}
\newcommand{\bseq}[1]{\textcolor{blue}{\dseq{#1}}}
\setlength{\tabcolsep}{1.5pt}
\par\medskip
\noindent\textbf{$\mathrm{SS}(4,0)$}
\\[2pt]
{\scriptsize%
\setlength{\tabcolsep}{1.5pt}%
\begin{tabular}{r|l}
\multicolumn{1}{r|}{\(\area\)} & $\mathrm{string}\,1$\\
\hline
0 & \dseq{0,0,0,0}\\
1 & \dseq{0,1,0,0}\\
2 & \dseq{0,1,1,0}\\
3 & \dseq{0,1,1,1}\\
\end{tabular}%
}
\par\medskip
\noindent\textbf{$\mathrm{SS}(5,0)$}
\\[2pt]
{\scriptsize%
\setlength{\tabcolsep}{1.5pt}%
\begin{tabular}{r|l}
\multicolumn{1}{r|}{\(\area\)} & $\mathrm{string}\,1$\\
\hline
0 & \dseq{0,0,0,0,0}\\
1 & \dseq{0,1,0,0,0}\\
2 & \dseq{0,1,1,0,0}\\
3 & \dseq{0,1,1,1,0}\\
4 & \dseq{0,1,1,1,1}\\
5 & \dseq{0,1,2,1,1}\\
\end{tabular}%
}
\par\medskip
\noindent\textbf{$\mathrm{SS}(5,1)$}
\\[2pt]
{\scriptsize%
\setlength{\tabcolsep}{1.5pt}%
\begin{tabular}{r|l}
\multicolumn{1}{r|}{\(\area\)} & $\mathrm{string}\,1$\\
\hline
1 & \dseq{0,0,1,0,0}\\
2 & \dseq{0,1,0,1,0}\\
3 & \dseq{0,1,1,0,1}\\
4 & \dseq{0,1,2,1,0}\\
\end{tabular}%
}
\par\medskip
\noindent\textbf{$\mathrm{SS}(5,2)$}
\\[2pt]
{\scriptsize%
\setlength{\tabcolsep}{1.5pt}%
\begin{tabular}{r|ll}
\multicolumn{1}{r|}{\(\area\)} & $\mathrm{string}\,1$ & $\mathrm{string}\,2$\\
\hline
1 & \dseq{0,0,0,1,0} & \\
2 & \dseq{0,1,0,0,1} & \dseq{0,0,1,1,0}\\
3 & \dseq{0,1,2,0,0} & \dseq{0,1,0,1,1}\\
4 & \dseq{0,1,1,2,0} & \dseq{0,1,2,0,1}\\
\end{tabular}%
}
\par\medskip
\noindent\textbf{$\mathrm{SS}(6,0)$}
\\[2pt]
{\scriptsize%
\setlength{\tabcolsep}{1.5pt}%
\begin{tabular}{r|l}
\multicolumn{1}{r|}{\(\area\)} & $\mathrm{string}\,1$\\
\hline
0 & \dseq{0,0,0,0,0,0}\\
1 & \dseq{0,1,0,0,0,0}\\
2 & \dseq{0,1,1,0,0,0}\\
3 & \dseq{0,1,1,1,0,0}\\
4 & \dseq{0,1,1,1,1,0}\\
5 & \dseq{0,1,1,1,1,1}\\
6 & \dseq{0,1,2,1,1,1}\\
7 & \dseq{0,1,2,2,1,1}\\
\end{tabular}%
}
\par\medskip
\noindent\textbf{$\mathrm{SS}(6,1)$}
\\[2pt]
{\scriptsize%
\setlength{\tabcolsep}{1.5pt}%
\begin{tabular}{r|l}
\multicolumn{1}{r|}{\(\area\)} & $\mathrm{string}\,1$\\
\hline
1 & \dseq{0,0,1,0,0,0}\\
2 & \dseq{0,1,0,1,0,0}\\
3 & \dseq{0,1,1,0,1,0}\\
4 & \dseq{0,1,1,1,0,1}\\
5 & \dseq{0,1,2,1,1,0}\\
6 & \dseq{0,1,1,2,1,1}\\
7 & \dseq{0,1,2,1,2,1}\\
\end{tabular}%
}
\par\medskip
\noindent\textbf{$\mathrm{SS}(6,2)$}
\\[2pt]
{\scriptsize%
\setlength{\tabcolsep}{1.5pt}%
\begin{tabular}{r|ll}
\multicolumn{1}{r|}{\(\area\)} & $\mathrm{string}\,1$ & $\mathrm{string}\,2$\\
\hline
1 & \dseq{0,0,0,1,0,0} & \\
2 & \dseq{0,1,0,0,1,0} & \dseq{0,0,1,1,0,0}\\
3 & \dseq{0,1,1,0,0,1} & \dseq{0,1,0,1,1,0}\\
4 & \dseq{0,1,2,1,0,0} & \dseq{0,1,1,0,1,1}\\
5 & \dseq{0,1,1,2,1,0} & \dseq{0,1,2,1,0,1}\\
6 & \dseq{0,1,1,1,2,1} & \dseq{0,1,2,2,1,0}\\
\end{tabular}%
}
\par\medskip
\noindent\textbf{$\mathrm{SS}(6,3)$}
\\[2pt]
{\scriptsize%
\setlength{\tabcolsep}{1.5pt}%
\begin{tabular}{r|lll}
\multicolumn{1}{r|}{\(\area\)} & $\mathrm{string}\,1$ & $\mathrm{string}\,2$ & $\mathrm{string}\,3$\\
\hline
1 & \dseq{0,0,0,0,1,0} &  & \\
2 & \dseq{0,1,0,0,0,1} & \dseq{0,0,1,0,1,0} & \\
3 & \dseq{0,1,2,0,0,0} & \dseq{0,1,0,1,0,1} & \dseq{0,0,1,1,1,0}\\
4 & \dseq{0,1,1,2,0,0} & \dseq{0,1,2,0,1,0} & \dseq{0,1,0,1,1,1}\\
5 & \dseq{0,1,1,1,2,0} & \dseq{0,1,1,2,0,1} & \dseq{0,1,2,0,1,1}\\
6 & \dseq{0,1,1,1,1,2} & \dseq{0,1,2,1,2,0} & \dseq{0,1,2,2,0,1}\\
\end{tabular}%
}
\par\medskip
\noindent\textbf{$\mathrm{SS}(6,4)$}
\\[2pt]
{\scriptsize%
\setlength{\tabcolsep}{1.5pt}%
\begin{tabular}{r|llll}
\multicolumn{1}{r|}{\(\area\)} & $\mathrm{string}\,1$ & $\mathrm{string}\,2$ & $\mathrm{string}\,3$ & $\mathrm{string}\,4$\\
\hline
1 & \dseq{0,0,0,0,0,1} &  &  & \\
2 & \rseq{0,0,1,0,0,1} & \dseq{0,0,0,1,1,0} &  & \\
3 & \dseq{0,0,1,2,0,0} & \dseq{0,1,0,0,1,1} & \dseq{0,0,1,1,0,1} & \\
4 & \dseq{0,1,0,1,2,0} & \dseq{0,1,2,0,0,1} & \dseq{0,0,1,2,1,0} & \dseq{0,0,1,1,1,1}\\
5 & \dseq{0,1,1,0,1,2} & \dseq{0,1,2,2,0,0} & \dseq{0,1,0,1,2,1} & \dseq{0,0,1,2,1,1}\\
\end{tabular}%
}
\par\medskip
\noindent\textbf{$\mathrm{SS}(7,0)$}
\\[2pt]
{\scriptsize%
\setlength{\tabcolsep}{1.5pt}%
\begin{tabular}{r|l}
\multicolumn{1}{r|}{\(\area\)} & $\mathrm{string}\,1$\\
\hline
0 & \dseq{0,0,0,0,0,0,0}\\
1 & \dseq{0,1,0,0,0,0,0}\\
2 & \dseq{0,1,1,0,0,0,0}\\
3 & \dseq{0,1,1,1,0,0,0}\\
4 & \dseq{0,1,1,1,1,0,0}\\
5 & \dseq{0,1,1,1,1,1,0}\\
6 & \dseq{0,1,1,1,1,1,1}\\
7 & \dseq{0,1,2,1,1,1,1}\\
8 & \dseq{0,1,2,2,1,1,1}\\
9 & \dseq{0,1,2,2,2,1,1}\\
10 & \dseq{0,1,2,2,2,2,1}\\
\end{tabular}%
}
\par\medskip
\noindent\textbf{$\mathrm{SS}(7,1)$}
\\[2pt]
{\scriptsize%
\setlength{\tabcolsep}{1.5pt}%
\begin{tabular}{r|l}
\multicolumn{1}{r|}{\(\area\)} & $\mathrm{string}\,1$\\
\hline
1 & \dseq{0,0,1,0,0,0,0}\\
2 & \dseq{0,1,0,1,0,0,0}\\
3 & \dseq{0,1,1,0,1,0,0}\\
4 & \dseq{0,1,1,1,0,1,0}\\
5 & \dseq{0,1,1,1,1,0,1}\\
6 & \dseq{0,1,2,1,1,1,0}\\
7 & \dseq{0,1,1,2,1,1,1}\\
8 & \dseq{0,1,2,1,2,1,1}\\
9 & \dseq{0,1,2,2,1,2,1}\\
10 & \dseq{0,1,2,2,2,1,2}\\
\end{tabular}%
}
\par\medskip
\noindent\textbf{$\mathrm{SS}(7,2)$}
\\[2pt]
{\scriptsize%
\setlength{\tabcolsep}{1.5pt}%
\begin{tabular}{r|ll}
\multicolumn{1}{r|}{\(\area\)} & $\mathrm{string}\,1$ & $\mathrm{string}\,2$\\
\hline
1 & \dseq{0,0,0,1,0,0,0} & \\
2 & \dseq{0,1,0,0,1,0,0} & \dseq{0,0,1,1,0,0,0}\\
3 & \dseq{0,1,1,0,0,1,0} & \dseq{0,1,0,1,1,0,0}\\
4 & \dseq{0,1,1,1,0,0,1} & \dseq{0,1,1,0,1,1,0}\\
5 & \dseq{0,1,2,1,1,0,0} & \dseq{0,1,1,1,0,1,1}\\
6 & \dseq{0,1,1,2,1,1,0} & \dseq{0,1,2,1,1,0,1}\\
7 & \dseq{0,1,1,1,2,1,1} & \dseq{0,1,2,2,1,1,0}\\
8 & \dseq{0,1,2,1,1,2,1} & \dseq{0,1,1,2,2,1,1}\\
9 & \dseq{0,1,2,2,1,1,2} & \dseq{0,1,2,1,2,2,1}\\
\end{tabular}%
}
\par\medskip
\noindent\textbf{$\mathrm{SS}(7,3)$}
\\[2pt]
{\scriptsize%
\setlength{\tabcolsep}{1.5pt}%
\begin{tabular}{r|lll}
\multicolumn{1}{r|}{\(\area\)} & $\mathrm{string}\,1$ & $\mathrm{string}\,2$ & $\mathrm{string}\,3$\\
\hline
1 & \dseq{0,0,0,0,1,0,0} &  & \\
2 & \dseq{0,1,0,0,0,1,0} & \dseq{0,0,1,0,1,0,0} & \\
3 & \dseq{0,1,1,0,0,0,1} & \dseq{0,1,0,1,0,1,0} & \dseq{0,0,1,1,1,0,0}\\
4 & \dseq{0,1,2,1,0,0,0} & \dseq{0,1,1,0,1,0,1} & \dseq{0,1,0,1,1,1,0}\\
5 & \dseq{0,1,1,2,1,0,0} & \dseq{0,1,2,1,0,1,0} & \dseq{0,1,1,0,1,1,1}\\
6 & \dseq{0,1,1,1,2,1,0} & \dseq{0,1,1,2,1,0,1} & \dseq{0,1,2,1,0,1,1}\\
7 & \dseq{0,1,1,1,1,2,1} & \dseq{0,1,2,1,2,1,0} & \dseq{0,1,2,2,1,0,1}\\
8 & \dseq{0,1,2,1,1,1,2} & \dseq{0,1,1,2,1,2,1} & \dseq{0,1,2,2,2,1,0}\\
9 & \dseq{0,1,2,3,1,1,1} & \dseq{0,1,2,1,2,1,2} & \dseq{0,1,1,2,2,2,1}\\
\end{tabular}%
}
\par\medskip
\noindent\textbf{$\mathrm{SS}(7,4)$}
\\[2pt]
{\scriptsize%
\setlength{\tabcolsep}{1.5pt}%
\begin{tabular}{r|lllll}
\multicolumn{1}{r|}{\(\area\)} & $\mathrm{string}\,1$ & $\mathrm{string}\,2$ & $\mathrm{string}\,3$ & $\mathrm{string}\,4$ & $\mathrm{string}\,5$\\
\hline
1 & \dseq{0,0,0,0,0,1,0} &  &  &  & \\
2 & \dseq{0,1,0,0,0,0,1} & \dseq{0,0,0,1,1,0,0} & \dseq{0,0,1,0,0,1,0} &  & \\
3 & \dseq{0,1,2,0,0,0,0} & \dseq{0,1,0,0,1,1,0} & \dseq{0,1,0,1,0,0,1} & \dseq{0,0,1,1,0,1,0} & \\
4 & \dseq{0,1,1,2,0,0,0} & \dseq{0,1,1,0,0,1,1} & \dseq{0,1,2,0,1,0,0} & \dseq{0,1,0,1,1,0,1} & \dseq{0,0,1,1,1,1,0}\\
5 & \dseq{0,1,1,1,2,0,0} & \dseq{0,1,2,1,0,0,1} & \dseq{0,1,1,2,0,1,0} & \dseq{0,1,2,0,1,1,0} & \dseq{0,1,0,1,1,1,1}\\
6 & \dseq{0,1,1,1,1,2,0} & \dseq{0,1,2,2,1,0,0} & \dseq{0,1,1,1,2,0,1} & \dseq{0,1,1,2,0,1,1} & \dseq{0,1,2,0,1,1,1}\\
7 & \bseq{0,1,1,1,1,1,2} & \dseq{0,1,1,2,2,1,0} & \dseq{0,1,2,1,1,2,0} & \dseq{0,1,2,1,2,0,1} & \dseq{0,1,2,2,0,1,1}\\
8 & \bseq{0,1,2,3,1,1,0} & \dseq{0,1,1,1,2,2,1} & \dseq{0,1,1,2,1,1,2} & \dseq{0,1,2,2,1,2,0} & \dseq{0,1,2,2,2,0,1}\\
\end{tabular}%
}
\par\medskip
\noindent\textbf{$\mathrm{SS}(7,5)$}
\\[2pt]
{\scriptsize%
\setlength{\tabcolsep}{1.5pt}%
\begin{tabular}{r|llllll}
\multicolumn{1}{r|}{\(\area\)} & $\mathrm{string}\,1$ & $\mathrm{string}\,2$ & $\mathrm{string}\,3$ & $\mathrm{string}\,4$ & $\mathrm{string}\,5$ & $\mathrm{string}\,6$\\
\hline
1 & \dseq{0,0,0,0,0,0,1} &  &  &  &  & \\
2 & \rseq{0,0,1,0,0,0,1} & \dseq{0,0,0,1,0,1,0} &  &  &  & \\
3 & \dseq{0,0,1,2,0,0,0} & \dseq{0,1,0,0,1,0,1} & \dseq{0,0,1,0,1,1,0} & \dseq{0,0,1,1,0,0,1} &  & \\
4 & \dseq{0,1,0,1,2,0,0} & \dseq{0,1,2,0,0,1,0} & \dseq{0,1,0,1,0,1,1} & \dseq{0,0,1,2,1,0,0} & \dseq{0,0,1,1,1,0,1} & \\
5 & \dseq{0,1,1,0,1,2,0} & \dseq{0,1,1,2,0,0,1} & \dseq{0,1,2,0,1,0,1} & \dseq{0,1,0,1,2,1,0} & \dseq{0,0,1,2,1,1,0} & \dseq{0,0,1,1,1,1,1}\\
6 & \bseq{0,1,1,1,0,1,2} & \dseq{0,1,2,1,2,0,0} & \dseq{0,1,2,2,0,1,0} & \dseq{0,1,1,0,1,2,1} & \dseq{0,1,0,1,2,1,1} & \dseq{0,0,1,2,1,1,1}\\
7 & \bseq{0,1,2,3,1,0,0} & \dseq{0,1,1,2,1,2,0} & \dseq{0,1,1,2,2,0,1} & \dseq{0,1,2,1,0,1,2} & \dseq{0,1,2,0,1,2,1} & \dseq{0,0,1,2,2,1,1}\\
8 & \dseq{0,1,1,2,3,1,0} & \dseq{0,1,1,1,2,1,2} & \dseq{0,1,2,1,2,2,0} & \dseq{0,1,2,3,1,0,1} & \dseq{0,1,2,2,0,1,2} & \dseq{0,0,1,2,2,2,1}\\
\end{tabular}%
}
\par\medskip
\noindent\textbf{$\mathrm{SS}(7,6)$}
\\[2pt]
{\scriptsize%
\setlength{\tabcolsep}{1.5pt}%
\begin{tabular}{r|llll}
\multicolumn{1}{r|}{\(\area\)} & $\mathrm{string}\,1$ & $\mathrm{string}\,2$ & $\mathrm{string}\,3$ & $\mathrm{string}\,4$\\
\hline
2 & \dseq{0,0,0,0,1,1,0} & \dseq{0,0,0,1,0,0,1} &  & \\
3 & \dseq{0,1,0,0,0,1,1} & \dseq{0,0,0,1,2,0,0} & \dseq{0,0,0,1,1,1,0} & \dseq{0,0,1,0,1,0,1}\\
4 & \dseq{0,1,2,0,0,0,1} & \dseq{0,1,0,0,1,2,0} & \dseq{0,1,0,0,1,1,1} & \dseq{0,0,1,2,0,1,0}\\
5 & \dseq{0,1,2,2,0,0,0} & \bseq{0,1,1,0,0,1,2} & \dseq{0,1,2,0,0,1,1} & \dseq{0,1,0,1,2,0,1}\\
6 & \dseq{0,1,1,2,2,0,0} & \bseq{0,1,2,3,0,0,0} & \dseq{0,1,2,2,0,0,1} & \dseq{0,1,2,0,1,2,0}\\
7 & \dseq{0,1,1,1,2,2,0} & \dseq{0,1,1,2,3,0,0} & \dseq{0,1,2,2,2,0,0} & \dseq{0,1,1,2,0,1,2}\\
\end{tabular}%
\par\smallskip
\begin{tabular}{r|llll}
\multicolumn{1}{r|}{\(\area\)} & $\mathrm{string}\,5$ & $\mathrm{string}\,6$ & $\mathrm{string}\,7$ & $\mathrm{string}\,8$\\
\hline
2 &  &  &  & \\
3 &  &  &  & \\
4 & \dseq{0,0,1,1,0,1,1} & \dseq{0,0,1,1,2,0,0} &  & \\
5 & \dseq{0,0,1,2,1,0,1} & \dseq{0,1,0,1,1,2,0} & \dseq{0,0,1,1,2,1,0} & \\
6 & \dseq{0,0,1,2,2,1,0} & \bseq{0,1,1,0,1,1,2} & \dseq{0,1,0,1,1,2,1} & \dseq{0,0,1,1,2,1,1}\\
7 & \dseq{0,1,0,1,2,2,1} & \bseq{0,1,2,3,0,1,0} & \dseq{0,1,2,0,1,1,2} & \dseq{0,0,1,2,1,2,1}\\
\end{tabular}%
}
\endgroup
\subsection{Intermediate strings for the local proofs}
\label{subsec:local-proof-strings}
We fix notation for the intermediate strings used in the remaining local
proofs.  Here ``string'' means an intermediate word produced during an extraction
branch of \(\mathrm{up}\) or \(\mathrm{down}\); it is not a string in the
\(\mathrm{up}\)-string decomposition of Proposition~\ref{prop:decomposition}.
We write \(A:B\) for concatenation and write \(\operatorname{last}_r(A)\) for
the suffix of length \(r\) of a word \(A\).
An occurrence of a value is called \emph{initial} if it is the leftmost
occurrence of that value in the current word, and \emph{non-initial} otherwise.
For any finite word \(T\), not necessarily Dyck, call a pair \(i<j\) a type~A
pair if \(T_i>T_j+1\), and call it a type~B pair if \(T_i<T_j\) and \(T_i\)
is non-initial in \(T\).  Write \(\delta(T)\) for the number of type~A and
type~B pairs in \(T\).  When \(T\) is Dyck,
Proposition~\ref{prop:deficit-pair-count} says \(\delta(T)=\defc(T)\).

The intermediate words below are produced from a Dyck input by initial segments
of the extraction--decrement--append process in \(\mathrm{up}\) and
\(\mathrm{down}\).  The statistic comparison in
Lemma~\ref{lem:up-down-statistics} gives the exact deficit of these
intermediate words: after each extracted \(e\) is moved to the right and
lowered to \(e-1\), the value of
\(\binom n2-\area-\dinv\) remains equal to the original input deficit.  On
non-Dyck intermediate words we therefore use
Lemma~\ref{lem:suffix-corrected-deficit-lower-bound}: a count of type~A and
type~B pairs gives a valid contradiction only after excluding, or otherwise
accounting for, type~A pairs wholly inside the appended lowered suffix.
For the \(\mathrm{up}\) process, set \(C_0=x\).  If the algorithm extracts
\(e_i\) from \(C_{i-1}\), let \(C_i\) be the remainder and put
\[
  \sigma_i=C_i:(e_1-1):(e_2-1):\cdots:(e_i-1).
\]
Thus \(\sigma_1\) is the word whose final three entries are tested by
\(\mathrm{East}_3\), \(\sigma_2\) is the word whose final five entries are
tested by \(\mathrm{East}_5\), and \(\sigma_3\) is the word whose final seven
entries are tested by \(\mathrm{East}_7\).  In symbols,
\[
  W^E_1=\operatorname{last}_3(\sigma_1),\qquad
  W^E_2=\operatorname{last}_5(\sigma_2),\qquad
  W^E_3=\operatorname{last}_7(\sigma_3).
\]
The appended word \((e_1-1,\ldots,e_i-1)\) records the extracted entries after
they have been lowered by one and moved to the right end.
For the \(\mathrm{down}\) process, set \(D_0=y\).  If the algorithm extracts
\(f_i\) from \(D_{i-1}\), let \(D_i\) be the remainder and put
\[
  \tau_i=D_i:(f_1-1):(f_2-1):\cdots:(f_i-1).
\]
The candidate in the skeleton branch is \(\tau_1=D_1:(f_1-1)\).  If the
algorithm proceeds beyond that branch, then \(\operatorname{last}_3(\tau_2)\),
\(\operatorname{last}_5(\tau_3)\), and \(\operatorname{last}_7(\tau_4)\) are
the windows tested by \(\mathrm{West}_3\), \(\mathrm{West}_5\), and
\(\mathrm{West}_7\), respectively.

\begin{lemma}[Suffix-corrected deficit lower bound]
\label{lem:suffix-corrected-deficit-lower-bound}
Let \(T=R:s\), where \(R\) is an ordinary Dyck sequence and \(s\) is a
nonempty reverse Dyck sequence.  Write \(m=s[-1]\), and assume that \(m\)
occurs in the prefix \(R\).  Let \(P(T)\) be the number of type~A and type~B
pairs in the full word \(T\), and let \(A_s(T)\) be the number of type~A pairs
whose two endpoints both lie in the suffix \(s\).  Then
\[
  \defc(T)\ge P(T)-A_s(T).
\]
\end{lemma}

\begin{proof}
Apply Lemma~\ref{lem:extended-deficit-pairs} to \(T\).  Since \(R\) is an
ordinary Dyck sequence, every value below a prefix entry occurs earlier in the
prefix.  Thus only suffix positions can contribute to the missing-value
correction.  It is enough to inject the missing-value contributions at suffix
positions into type~A pairs wholly inside \(s\).

Fix a suffix position \(j\), and write \(e=T_j\).  Since \(m\) occurs in
\(R\), and \(R\) is Dyck, every value \(0,\ldots,m\) occurs to the left of
\(j\).  Hence any missing value \(v<e\) at position \(j\) satisfies
\[
  m<v<e.
\]
Because the final suffix entry is \(m\), there is a first suffix position
\(k>j\) with \(T_k<v\).  Then \(T_{k-1}\ge v\), while the reverse-Dyck
condition on \(s\) gives
\[
  T_k\ge T_{k-1}-1.
\]
If \(T_{k-1}>v\), then \(T_k\ge v\), contradicting the choice of \(k\).  Thus
\(T_{k-1}=v\), and the inequalities force \(T_k=v-1\).  Therefore
\((j,k)\) is a type~A pair wholly inside the suffix, since
\[
  T_j=e>v=T_k+1.
\]

For a fixed \(j\), distinct missing values \(v\) give distinct right-hand
values \(T_k=v-1\), and hence distinct pairs.  For distinct \(j\), the
left-hand endpoint is distinct.  Therefore the missing-value correction is at
most \(A_s(T)\).  Lemma~\ref{lem:extended-deficit-pairs} gives
\(\defc(T)\ge P(T)-A_s(T)\).
\end{proof}

We apply Lemma~\ref{lem:suffix-corrected-deficit-lower-bound} only to
pre-local-move intermediate words whose suffix is the lowered extraction
suffix.  In those applications the suffix hypotheses are checked from the
extraction structure; in particular, Lemma~\ref{lem:extractable-structure}
supplies the surviving occurrence in the Dyck remainder of the final lowered
suffix value.

The position-bound proof uses one additional local notation.  Suppose a first
forbidden extraction position occurs, and let \(e\) be the element extracted at
that step.  Immediately before the extraction, Lemma~\ref{lem:extractable-structure}
places \(d=e-1\) directly to the left of \(e\).  The current remainder has the
form
\[
  \cdots,d,e,P,
\]
and the corresponding full intermediate word, including previously extracted
and decremented entries, has the form
\[
  \cdots,d,e,P,Q.
\]
Here \(P\) is the suffix to the right of the offending extracted element in the
current remainder, and \(Q\) is the word of previously appended lowered
extractions.  On the \(\mathrm{up}\) side one has \(|P|\le 2\) and
\(|Q|\le 2\).  On the \(\mathrm{down}\) side one has \(|P|\le 1\) and
\(|Q|\le 3\).  In either direction \(|P|+|Q|\le 4\).  This notation isolates
the short suffix where a possible injection failure could occur.  Since \(Q\)
lies in the appended lowered suffix of the non-Dyck intermediate word, type~A
pairs wholly inside \(Q\) are precisely the suffix-internal pairs that must be
excluded or accounted for by
Lemma~\ref{lem:suffix-corrected-deficit-lower-bound}.
\subsection{Proof of Lemma~\ref{lem:skeleton-wd}}
\label{subsec:proof-skeleton-wd}
The residual checker of Subsection~\ref{subsec:finite-residual-check} already
covers the range \(4\le n\le 7\).  In that range the checker applies the
exact area bounds from Lemma~\ref{lem:skeleton-wd}: on the \(\mathrm{up}\)
side it tests all sequences with \(\area\le \ell-1\), and on the
\(\mathrm{down}\) side it tests all nonspecial sequences with
\(\area\le \ell\).  The labels ``up skeleton'' and ``down skeleton'' in the
successful output are precisely the two skeleton branches of the lemma.  It
therefore remains to prove the two assertions for \(n\ge 8\).
\begin{proof}[Proof of Lemma~\ref{lem:skeleton-wd} for \(n\ge 8\)]
We shall use the pair description of Proposition~\ref{prop:deficit-pair-count}.
Thus a type~A pair is a pair \(i<j\) with \(s_i>s_j+1\), and a type~B pair is
a pair \(i<j\) with \(s_i<s_j\) and \(s_i\) non-initial.  For a Dyck sequence,
the number of such pairs is its deficit.
\paragraph{The \(\mathrm{up}\) skeleton branch.}
Let
\[
  x=(x_0,x_1,\ldots,x_{n-1})
\]
be a full Dyck skeleton with \(\area(x)\le \ell-1\) and \(x\ne\omega_n\).  Put
\(a=x_{n-1}\).  The skeleton branch removes the final entry, forms
\(v=a+1\), and tries to inject \(v\) immediately after the first occurrence of
\(a\) in the prefix \((x_0,\ldots,x_{n-2})\).
If this occurrence of \(a\) exists, the injection is defined.  The output has
length \(n\), remains nonnegative, and still starts with \(0\).  The new left
adjacency is \(a,a+1\).  If the inserted entry has a right neighbor, that right
neighbor was already the successor of an \(a\) in the original Dyck sequence,
so it is at most \(a+1\), hence at most \((a+1)+1\).  Thus the output is again
a Dyck sequence.
It remains to rule out the case in which the prefix contains no occurrence of
\(a\).  Assume this happens.  Then the final entry is the unique occurrence of
its value.  It is also the unique maximum: if a larger value occurred earlier,
the affine Dyck condition and the initial value \(x_0=0\) would force an
earlier occurrence of \(a\) on the way to that larger value.  Since
\(x\ne\omega_n\), this unique final maximum is not equal to \(1\); hence
\[
  a\ge 2.
\]
The last affine inequality gives \(a\le x_{n-2}+1\), while uniqueness of the
maximum gives \(x_{n-2}<a\).  Therefore
\[
  x_{n-2}=a-1\ge 1.
\]
We next record consequences of full non-extractability.  First, the value
\(1\) occurs at least twice.  If it occurred only once, start at that unique
\(1\) and take the maximal consecutive string \(1,2,\ldots,r\) forced on the
way to the final value \(a\ge2\).  Its endpoint has exactly one predecessor
copy to its left, and by maximality its successor, if present, is at most
\(r\).  The endpoint would be extractable, contradiction.  Second, \(x_1=0\).
If \(x_1=1\), the maximal initial consecutive string \(1,2,\ldots,r\) starting
at position \(1\) again ends at an extractable entry.  Third, every initial
occurrence of a value at least \(2\) lies to the right of the second \(1\);
otherwise the maximal string starting at the first \(1\) would end before the
second \(1\) appears and would again be extractable.

Now count deficit pairs in two disjoint groups.  For the first group, consider
every entry except the final maximum, the first \(0\), and the first \(1\).
There are \(n-3\) such entries.  If such an entry is non-initial, it is smaller
than the unique final maximum \(a\), so it forms a type~B pair with the final
entry.  If it is initial, it has value at least \(2\); by the preceding
paragraph it lies to the right of the second \(1\), and the second \(1\) forms
a type~B pair with it.  This gives \(n-3\) pairs.

For the second group, consider the positions \(2,3,\ldots,n-2\), again
\(n-3\) entries.  If \(x_j>0\), then the second zero \(x_1=0\) forms a
type~B pair with \(x_j\).  If \(x_j=0\), then this zero is non-initial and
forms a type~B pair with \(x_{n-2}=a-1\ge1\).  These pairs are disjoint from
the first group: the first group uses either the final position or the second
\(1\), while the second group uses the second \(0\) or a middle zero paired
with position \(n-2\).  Thus \(x\) has at least
\[
  (n-3)+(n-3)=2n-6
\]
deficit pairs, contradicting \(d=\defc(x)\le2n-8\).
Therefore the prefix must contain an occurrence of \(a\), and the \(\mathrm{up}\)
skeleton injection is well-defined.
\paragraph{The \(\mathrm{down}\) skeleton branch.}
Now let \(y\) be a nonspecial Dyck sequence with \(\area(y)\le\ell\).  Suppose
the first extraction in the \(\mathrm{down}\) algorithm removes \(f\), leaving
\(D_1\), and suppose the no-extractable skeleton test is reached.  Put
\[
  z=f-1,\qquad S=D_1:z .
\]
By Lemma~\ref{lem:extractable-structure}, \(D_1\) is Dyck.  The only Dyck
inequality for \(S\) not inherited from \(D_1\) is the final affine boundary
\[
  z\le D_1[-1]+1.
\]
We prove this boundary by contradiction.  Suppose \(z>D_1[-1]+1\).  Then
\(z\ge2\), and \(S\) has no extractable element by the skeleton test.  Let the
first copy of \(z\) be the copy left inside \(D_1\), and let the final appended
\(z\) be the last copy of \(z\).

The same maximal-string argument as above gives \(S_1=0\).  It also shows that
every non-final entry greater than \(1\) lies to the right of a second \(1\):
otherwise the maximal string \(1,2,\ldots,r\) starting at the first \(1\) ends
inside the inherited Dyck part \(D_1\), because a copy of \(z\) already occurs
in \(D_1\) while \(D_1[-1]<z-1\), and the endpoint would be extractable.

We count pairs in \(S\).  The word \(S=D_1:z\) has the same deficit
as the original input, and Lemma~\ref{lem:suffix-corrected-deficit-lower-bound}
applies with appended suffix of length one.  Hence there are no type~A pairs
wholly inside the appended suffix to subtract.  More than \(2n-8\)
type~A/type~B pairs in \(S\) therefore contradicts the original deficit bound.
First
consider positions \(2,\ldots,n-2\).  If \(S_j>0\), the second zero \(S_1\)
forms a type~B pair with \(S_j\).  If \(S_j=0\), pair this zero with the first
copy of \(z\) in \(D_1\): the pair is type~B if the zero precedes that first
\(z\), and type~A if it follows it.  This gives \(n-3\) pairs, none using the
final appended \(z\).

For the second group, consider every entry except the first \(0\), the first
\(1\), and the final appended \(z\).  If the entry has value \(0\) or \(1\), it
is non-initial and forms a type~B pair with the final \(z\).  If it has value
greater than \(1\), the second \(1\) forms a type~B pair with it.  This gives
another \(n-3\) pairs.  The two groups are disjoint by their endpoint roles, so
we obtain at least \(2n-6>2n-8\) pairs, contradiction.

Therefore \(z\le D_1[-1]+1\), so \(S=D_1:z\) is Dyck.  Since the branch test
says \(S\) has no extractable element, \(S\) is a full Dyck skeleton, and the
down skeleton branch is well-defined.
\end{proof}
\subsection{Proof of Lemma~\ref{lem:extractions-wd}}
\label{subsec:proof-extractions-wd}
The residual checker of Subsection~\ref{subsec:finite-residual-check} already
covers \(4\le n\le 7\), so assume \(n\ge 8\).  We prove that the required
extractions cannot fail by showing that any failed branch would force more
than \(2n-8\) deficit pairs.  The argument uses the intermediate words
introduced in Subsection~\ref{subsec:local-proof-strings}.
\begin{proof}[Proof of Lemma~\ref{lem:extractions-wd} for \(n\ge 8\)]
The first extraction in a non-skeleton branch cannot fail: the input is not a
full Dyck skeleton, so an extractable element exists, and the leftmost
convention selects the element used by the algorithm.  After each successful
extraction, Lemma~\ref{lem:extractable-structure} leaves a Dyck remainder.
Thus, if a later attempted extraction fails, the current remainder is a full
Dyck skeleton, and Lemma~\ref{lem:suffix-corrected-deficit-lower-bound} applies
to the full intermediate word with its appended lowered suffix.  We show that
each possible failure would force \(2n-7\) counted pairs, with no counted
type~A pair wholly inside that appended suffix.

We use the following elementary consequence of full non-extractability in each
failed-extraction remainder.  If a full skeleton \(x\) contains a positive
entry that is not the final entry, or if its final entry is at least \(2\),
then \(x[1]=0\), a second \(1\) occurs in \(x\), and every initial occurrence
of a value at least \(2\) lies to the right of that second \(1\).  Indeed, if a
unique \(1\), or an initial run beginning at position \(1\), or an initial
value at least \(2\) before the second \(1\) occurred, the corresponding
maximal consecutive string \(1,2,\ldots,r\) would end at an eligible index,
contradicting fullness.  The hypotheses needed for this observation hold in
the four cases below: in the two \(\mathrm{up}\) cases the final entry of the
full remainder is at least \(2\), while in the two \(\mathrm{down}\) cases a
positive value produced by the previous extraction remains inside the full
remainder.

First suppose the second extraction in \(\mathrm{up}\) fails.  The current
word is \(S=x:z\), where \(x\) is a full skeleton of length \(n-1\),
and the failed \(\mathrm{East}_3\) test gives \(x[-1]\gg z\), i.e.
\(x[-1]\ge z+2\).  Count the type~A pair \((x[-1],z)\), which is not wholly
inside the one-entry appended suffix.  For each
occurrence \(u\) of \(x\) except the first \(0\), the first \(1\), and
\(x[-1]\), take \((u,x[-1])\) if \(u\) is non-initial and \(u<x[-1]\), take
\((u,z)\) if \(u\) is non-initial and \(u\ge x[-1]\), and take
\((\text{second }1,u)\) if \(u\) is initial.  These are respectively type~B,
type~A, and type~B pairs, and give \(n-4\) pairs.  For each \(u=x[j]\) with
\(2\le j\le n-3\), take \((x[1],u)\) if \(u>0\), and \((u,x[-2])\) if
\(u=0\).  Since \(x[1]\) is the second zero and \(x[-2]\ge x[-1]-1\ge1\),
this gives another \(n-4\) type~B pairs.  The three groups are disjoint, giving
\[
  1+(n-4)+(n-4)=2n-7
\]
pairs.

Next suppose the third extraction in \(\mathrm{up}\) fails.  The current word
is \(S=x:(w,z)\), where \(x\) has length \(n-2\).  Put
\[
  A=x[-2],\qquad B=x[-1].
\]
The failed window is \([A,B,w,z]\), and the earlier \(\mathrm{East}_3\) test
did not stop, so \(B\gg w\).  The failed \(\mathrm{East}_5\) test supplies
three pairs, and every listed type~A pair has left endpoint in \(x\), hence is
not wholly inside the appended suffix \((w,z)\):
\[
\begin{cases}
(A,w),(A,z),(B,w),& A\gg B,\\
(A,w),(B,w),(B,z),& A\not\gg B\text{ but }A\gg w,\\
(B,w),(A,z),(B,z),& A\not\gg w.
\end{cases}
\]
For the next disjoint group of pairs, first count
\((\text{second }1,B)\).  Assign one further pair to the occurrence \(A\):
use \((A,B)\) if \(A\gg B\); if \(A<B\), use \((A,B)\) when \(A\) is
non-initial and \((\text{second }1,A)\) when \(A\) is initial; in the remaining
case \(A\ge B\) but \(A\not\gg B\), use \((A,z)\).  For every other occurrence
of \(x\) except the first \(0\), first \(1\), second \(1\), \(A\), and \(B\),
take \((u,B)\) if \(u\) is non-initial and \(u<B\), take \((u,w)\) if \(u\) is
non-initial and \(u\ge B\), and take \((\text{second }1,u)\) if \(u\) is
initial.  This gives \(n-5\) pairs in this group.  Finally, for each \(u=x[j]\)
with \(2\le j\le n-4\), take \((x[1],u)\) if \(u>0\), and \((u,A)\) if
\(u=0\).  Since \(A\ge1\), this gives \(n-5\) more type~B pairs.  Together the
disjoint groups give
\[
  3+(n-5)+(n-5)=2n-7
\]
pairs.

For \(\mathrm{down}\), Lemma~\ref{lem:skeleton-wd} handles the skeleton branch
after the first extraction.  Suppose the third extraction in \(\mathrm{down}\)
fails.  The current word is \(S=x:(w,z)\), where \(x\) has length
\(n-2\), and put \(A=x[-1]\).  The failed \(\mathrm{West}_3\) condition gives
\(A\ll w\), so \(w\ge A+2\).  The appended word is reverse Dyck, hence
\(w\le z+1\), so \(z>A\).  The successful extraction that produced the final
\(z\) left a copy of \(z\) in \(x\); because \(z>A\) and \(x\) ends in \(A\),
the terminal \(A\) is non-initial.  Count the type~B pair \((A,w)\); type~B
pairs are not subtracted by the suffix correction.
For each occurrence \(u\) of \(x\) except the first \(0\), first \(1\), second
\(0\), and \(A\), take \((u,w)\) if \(u\) is non-initial and \(u<w\), take
\((u,A)\) if \(u\) is non-initial and \(u\ge w\), and take
\((\text{second }1,u)\) if \(u\) is initial.  This gives \(n-6\) pairs.  For
each \(u\) in \(x[2],x[3],\ldots,x[-1],w,z\), take \((x[1],u)\) if \(u>0\) and
\((u,z)\) if \(u=0\).  This gives \(n-2\) pairs.  Thus the total is
\[
  1+(n-6)+(n-2)=2n-7.
\]

Finally suppose the fourth extraction in \(\mathrm{down}\) fails.  The current
word is \(S=x:(v,w,z)\), where \(x\) has length \(n-3\), and put
\(A=x[-1]\).  The branch conditions include \(A\ll v\) and \(w>A\), while the
appended word is reverse Dyck, so \(v\le w+1\) and \(w\le z+1\).  First
\(z\ne0\): if \(z=0\), then \(A=0\), \(w=1\), and \(v=2\), and the reversed
\(\mathrm{East}_5\) window \((0,1,2,0,B)\), with \(B=x[-2]\), would satisfy
Case~2b.  Second, \(A\) is non-initial.  Since \(z\ge A\), either a larger
remaining \(z\) forces an earlier \(A\), or another remaining copy of \(z=A\)
does so directly; if the terminal \(A=z\) were the unique remaining \(z\), then
the previous \(\mathrm{West}_3\) window \((z+1,v,w)\) would have succeeded.
Now count the type~B pairs \((A,v)\) and \((A,w)\); type~B pairs are not
subtracted by the suffix correction.  For each
occurrence \(u\) of \(x\) except the first \(0\), first \(1\), second \(0\),
and \(A\), take \((u,v)\) if \(u\) is non-initial and \(u<v\), take \((u,A)\)
if \(u\) is non-initial and \(u\ge v\), and take \((\text{second }1,u)\) if
\(u\) is initial.  This gives \(n-7\) pairs.  For each \(u\) in
\(x[2],x[3],\ldots,x[-1],v,w,z\), take \((x[1],u)\) if \(u>0\) and \((u,z)\)
if \(u=0\).  Since \(z>0\), this gives \(n-2\) pairs.  The total is
\[
  2+(n-7)+(n-2)=2n-7.
\]

All possible attempted extraction failures lead to more than \(2n-8\)
suffix-corrected counted pairs, contradicting the deficit bound by
Lemma~\ref{lem:suffix-corrected-deficit-lower-bound}.  Hence every extraction
called by \(\mathrm{up}\) or \(\mathrm{down}\) succeeds, and the
leftmost-extractable convention determines the chosen element.
\end{proof}

\subsection{Proof of Lemma~\ref{lem:positions-wd}}
\label{subsec:proof-positions-wd}
The proof uses two finite checks, both included after the proof.  The first is a
finite checker for Lemma~\ref{lem:positions-wd} on an enlarged
limited-nonzero domain: all Dyck sequences with \(4\le n\le 13\) and at most
seven nonzero entries.  This covers every short case \(4\le n\le 8\), since
every such sequence begins with \(0\) and therefore has at most \(n-1\le 7\)
nonzero entries.  It also covers the only finite middle
range needed in the \(e=2\) argument below.  The second finite check handles
the finite range \(9\le n\le 16\) needed in the \(e\ge 3\) argument.  These
finite checks are part of the proof for their stated finite domains; the
remaining ranges are handled by the analytic arguments below.
\begin{proof}[Proof of Lemma~\ref{lem:positions-wd}]
By the limited-nonzero finite checker below, the lemma holds for
\(4\le n\le 8\).  We may therefore assume \(n\ge 9\), except when one of the
finite checks is invoked explicitly.
Suppose first that a position bound fails during either the \(\mathrm{up}\) or
\(\mathrm{down}\) computation, and choose the first extraction at which this
happens.  Use the notation of Subsection~\ref{subsec:local-proof-strings}: the
extracted element is \(e\), its immediate predecessor is \(d=e-1\), the current
remainder has the form
\[
  \cdots,d,e,P,
\]
and the full intermediate word in which the lower-bound pairs are counted has the form
\[
  \cdots,d,e,P,Q.
\]
On the \(\mathrm{up}\) side, \(|P|\le 2\) and \(|Q|\le 2\).  On the
\(\mathrm{down}\) side, \(|P|\le 1\) and \(|Q|\le 3\).  Thus, in either case,
if \(p=|P|\) and \(q=|Q|\), then \(p+q\le 4\).  The full intermediate word has
the same deficit as the original input at this branch prefix.  We
therefore use Lemma~\ref{lem:suffix-corrected-deficit-lower-bound}: counted
type~A/type~B pairs give a contradiction once type~A pairs wholly inside the
appended suffix \(Q\) are excluded.
The value \(e\) is not zero, since zero is never extractable.  It is also not
one.  Indeed, by Lemma~\ref{lem:extractable-structure}, an extracted element is
the leftmost occurrence of its value.  Thus an extractable \(1\) is the
leftmost \(1\); since a Dyck sequence begins with \(0\), it occurs in
position~\(1\), which is not one of the forbidden final positions when
\(n\ge9\).
We first handle the case \(e=2\).  Then \(d=1\), and every entry to the left of
this \(d\) is zero.  For \(9\le n\le 13\), a position-bound violation of this
kind implies that the original input has at most \(p+q+2\le 6\) nonzero entries:
after the entries recorded in \(Q\) have been extracted from a word of the form
\[
  0,0,\ldots,0,1,2,P,
\]
there remain at least \(n-p-q-2\) zeros.  This range is therefore covered by
the limited-nonzero finite checker, which checks the larger domain with at most
seven nonzero entries.  It remains to treat \(n\ge 14\).
Assume \(n\ge 14\).  If some entry \(c\) of \(P\cup Q\) is nonzero, then at
least \(n-7\) non-initial zeros occur before the displayed \(1,2\).  Each of
these zeros forms type~B pairs with \(d=1\), with \(e=2\), and with
\(c\).  This gives at least
\[
  3(n-7)=3n-21>2n-8
\]
suffix-corrected counted pairs, a contradiction.  Hence all entries of
\(P\cup Q\) are zero.
If \(p+q=0\), then the non-initial zeros before the displayed \(1,2\) pair with
both \(1\) and \(2\), giving \(2n-6\) counted pairs.  If \(p+q=1\), the same
zeros give \(2n-8\) pairs with \(1\) and \(2\), and \(e=2\) gives one
additional type~A pair with the unique zero in \(P\cup Q\).  If that zero lies
in \(Q\), the other endpoint is \(e\), outside the appended suffix.  Both
alternatives contradict the suffix-corrected lower bound and
\(\defc\le 2n-8\).  Thus only \(p+q\ge 2\), with all entries of \(P\cup Q\)
equal to zero, remains.
On the \(\mathrm{up}\) side, the possible pairs are
\[
  (p,q)\in
  \{(2,0),(1,1),(2,1),(0,2),(1,2),(2,2)\}.
\]
These are impossible for local reasons.  If \((p,q)=(2,0)\), then after the
first extraction the \(\mathrm{East}_3\) window is \((0,0,1)\), so the
algorithm terminates at \(\mathrm{East}_3\).  If \((p,q)=(1,1)\) or
\((2,1)\), then immediately before the offending extraction the
\(\mathrm{East}_3\) window is respectively \((2,0,0)\) or \((0,0,0)\), so
\(\mathrm{East}_3\) would already have applied.  If \((p,q)=(1,2)\) or
\((2,2)\), the preceding extracted element was a \(1\); before that preceding
extraction the corresponding \(\mathrm{East}_3\) window was again respectively
\((2,0,0)\) or \((0,0,0)\).  Finally, if \((p,q)=(0,2)\), then the current
\(\mathrm{East}_5\) window is \((0,1,2,0,0)\), and Case~2b of
\(\mathrm{East}_5\) applies.  Hence no \(e=2\) position-bound violation occurs
for \(\mathrm{up}\).
On the \(\mathrm{down}\) side, the possible remaining pairs are
\[
  (p,q)\in\{(1,1),(0,2),(1,2),(0,3),(1,3)\}.
\]
They are ruled out by the reversal-conjugate West tests.  If
\((p,q)=(1,1)\), then after the first extraction the \(\mathrm{West}_3\)
window is \((0,0,1)\).  If \((p,q)=(0,2)\) or \((1,2)\), then immediately
before the offending extraction the \(\mathrm{West}_3\) window is respectively
\((2,0,0)\) or \((0,0,0)\).  If \((p,q)=(0,3)\) or \((1,3)\), the preceding
extracted element was a \(1\), and before that preceding extraction the
\(\mathrm{West}_3\) window was respectively \((2,0,0)\) or \((0,0,0)\).  In
each case \(\mathrm{West}_3\) would have terminated the branch earlier.  Thus
no \(e=2\) position-bound violation occurs for \(\mathrm{down}\).
Now assume \(e\ge 3\).  We need two structural exclusions.  First, the
intermediate word cannot have the form
\[
  0,1,2,\ldots,e,P,Q.
\]
For \(n\ge 17\), such a word has area at least
\[
  \frac{(n-4)(n-5)}2,
\]
because \(p+q\le 4\).  This lower bound is larger than
\(\binom n2/2\) for \(n\ge 17\), while the original input has area at most
\(\ell\le \binom n2/2\) and the intermediate word is obtained from it only by
lowering previously extracted entries.  This is impossible.  For
\(9\le n\le 16\), the prefix-form finite checker below verifies the same
exclusion on the finite domains for the unspecified entries in \(P,Q\)
described after the proof.
Second, the intermediate word cannot have the form
\[
  0,0,1,2,\ldots,e,P,Q.
\]
For \(n\ge 17\), the second zero forms type~B pairs with at least \(n-6\)
positive entries in the increasing chain.  These type~B pairs are not
subtracted by the suffix correction, so the word has deficit at least
\(n-6\).  Its area is at least
\[
  \frac{(n-5)(n-6)}2,
\]
which is larger than
\[
  \frac{\binom n2-(n-6)}2
\]
for \(n\ge 17\).  Since the deficit is at least \(n-6\), this exceeds
the allowable half-area bound.  The finite range \(9\le n\le 16\) is again
covered by the prefix-form checker.
These two exclusions force the beginning of the word.  If the word began
\(0,1,\ldots\), then the terminal element of the maximal initial increasing
chain would be the leftmost extractable element and the first excluded form
would occur.  Hence the second entry is zero.  The second excluded form then
prevents the word from continuing as \(0,0,1,2,\ldots\).  By nonnegativity and
the affine Dyck inequality, the first four entries are therefore one of
\[
  0000,
  \qquad
  0001,
  \qquad
  0010,
  \qquad
  0011.
\]
In particular, among the first four entries there are two non-initial entries,
call them \(a\) and \(b\), with \(a=0\) and \(b\in\{0,1\}\).
We now count suffix-corrected type~A/type~B pairs in the full intermediate
word.  Excluding the entries \(d,e\), the \(p+q\) entries of \(P\cup Q\), the
first zero, and the first one, there are exactly \(n-p-q-4\) entries to the
left of \(d\).  Each gives two pairs.  If such an entry is non-initial, then
it is less than both
\(d\) and \(e\), so it forms type~B pairs with both.  If it is the initial
occurrence of a value greater than~\(1\), then the two early non-initial entries
\(a,b\) are smaller and lie to its left, so they form two type~B pairs with it.
This gives
\[
  2(n-p-q-4)
\]
pairs.
Next take any entry \(c\) of \(P\cup Q\).  For one pair, either \(a<c\), in
which case \((a,c)\) is type~B, or \(a\ge c\), in which case \(c=0\) and
\((d,c)\) is type~A.  For another pair, either \(b<c\), in which case
\((b,c)\) is type~B, or \(b\ge c\), in which case \(c\le 1\) and \((e,c)\) is
type~A.  Thus the entries of \(P\cup Q\) contribute \(2(p+q)\) further pairs.
Any counted type~A pair with \(c\in Q\) has left endpoint \(d\) or \(e\), hence
is not wholly inside the appended suffix.
Finally, the second zero and the first one form one additional type~B pair.
Altogether the word contains at least
\[
  2(n-p-q-4)+2(p+q)+1=2n-7
\]
suffix-corrected counted pairs, contradicting \(\defc\le 2n-8\).  This proves
all extraction position bounds.
It remains to show that the injection stages do not fail.  Consider a
non-skeleton local branch of \(\mathrm{up}\) that reaches an injection stage,
and let \(e\) be the last extracted element in that branch.  When \(e\) was
extracted, Lemma~\ref{lem:extractable-structure} placed a surviving copy of
\(e-1\) immediately to its left.  The position bounds just proved imply that
this copy is not moved by the local East map: in the three- and five-window
branches it lies outside the local window, and in the seven-window boundary
case it is the leftmost entry of the seven-window.  The East maps fix the
boundary entries of their local windows.  Therefore the adjacent copy of
\(e-1\) survives until the injection stage.
The decremented value \(e-1\) of the last extracted element is the rightmost
entry of the suffix that is incremented and injected, so right-to-left
injection reinserts \(e\) first.  This first \(e\) is reinserted immediately
after the surviving adjacent copy of \(e-1\).  The remaining entries to be
injected form a
reverse Dyck sequence by the local East output condition and the definitions of
the branches.  Thus the standard right-to-left injection argument applies to
the remaining entries: after one insertion succeeds, the next required
occurrence lies at or to the left of the preceding insertion point.  Hence no
injection in \(\mathrm{up}\) fails.
The proof for \(\mathrm{down}\) is the same with West in place of East.
Since West is the reversal conjugate of East, the West maps also fix the
boundary entries of their local windows.  The West position bounds preserve the
adjacent occurrence needed to reinsert the last extracted element, the last
extracted element is again injected first, and the remaining injected word is
reverse Dyck.  Thus no injection in \(\mathrm{down}\) fails.  The lemma
follows.
\end{proof}
\subsubsection*{Limited-nonzero finite checker for Lemma~\ref{lem:positions-wd}}
The following checker is run after the core routines of
Subsection~\ref{subsec:core-code}.  It checks the full position-bound and
injection-nonfailure conclusions on all Dyck sequences with \(4\le n\le 13\)
and at most seven nonzero entries satisfying the fixed-deficit and area
hypotheses of Lemma~\ref{lem:positions-wd}.
\begin{lstlisting}
from collections import Counter
from math import comb
N_MIN, N_MAX = 4, 13
MAX_NONZERO = 7

def require(test, message):
    if not test:
        raise AssertionError(message)

def nonzero_count(S):
    return sum(1 for x in S if x != 0)

def ell_value(n, d):
    return (comb(n, 2) - d) // 2

def check_image(source, image, n, d, delta):
    require(is_Dyck(image), f"non-Dyck image: {source} -> {image}")
    require(len(image) == n, f"length changed: {source} -> {image}")
    require(defc(image) == d, f"deficit changed: {source} -> {image}")
    require(area(image) == area(source) + delta,
            f"wrong area change: {source} -> {image}")

def checked_up(S, n, d, ell):
    S = tuple(S)
    if S == omega(n):
        image = epsilon(n)
        check_image(S, image, n, d, 1)
        return "up special", 3
    if is_full_skeleton(S):
        image = inject(S[:-1], S[-1] + 1)
        check_image(S, image, n, d, 1)
        return "up skeleton", 3
    j1, e1 = find_extractable(S)
    C1 = remove_at(S, j1)
    sigma1 = C1 + (e1 - 1,)
    if East3(sigma1[-3:]) is not None:
        require(j1 < n - 2, f"up East3 position bound: {S}")
        image = inject_right_to_left(sigma1[:-2],
                                     (sigma1[-2] + 1, sigma1[-1] + 1))
        check_image(S, image, n, d, 1)
        return "up East3", 3
    j2, e2 = find_extractable(C1)
    C2 = remove_at(C1, j2)
    sigma2 = C2 + (e1 - 1, e2 - 1)
    E5 = East5(sigma2[-5:])
    if E5 is not None:
        require(j1 < n - 3 and j2 < len(C1) - 3,
                f"up East5 position bound: {S}")
        base = sigma2[:-5] + E5[:2]
        image = inject_right_to_left(base, tuple(x + 1 for x in E5[2:]))
        check_image(S, image, n, d, 1)
        return "up East5", 5
    j3, e3 = find_extractable(C2)
    C3 = remove_at(C2, j3)
    sigma3 = C3 + (e1 - 1, e2 - 1, e3 - 1)
    W7 = sigma3[-7:]
    require(not is_far_apart_decomposable(W7), f"bad East7 window: {S}")
    require(j1 < n - 3 and j2 < len(C1) - 3 and j3 < len(C2) - 3,
            f"up East7 position bound: {S}")
    E7 = East7(W7)
    image = inject_right_to_left(sigma3[:-7] + E7[:-4],
                                 tuple(x + 1 for x in E7[-4:]))
    check_image(S, image, n, d, 1)
    return "up East7", 7

def checked_down(S, n, d, ell):
    S = tuple(S)
    if S == epsilon(n):
        image = omega(n)
        check_image(S, image, n, d, -1)
        return "down special", 3
    j1, f1 = find_extractable(S)
    D1 = remove_at(S, j1)
    candidate = D1 + (f1 - 1,)
    if find_extractable(candidate) is None:
        check_image(S, candidate, n, d, -1)
        return "down skeleton", 3
    j2, f2 = find_extractable(D1)
    D2 = remove_at(D1, j2)
    tau1 = D2 + (f1 - 1, f2 - 1)
    if West3(tau1[-3:]) is not None:
        require(j1 < n - 1 and j2 < len(D1) - 1,
                f"down West3 position bound: {S}")
        image = inject(tau1[:-1], tau1[-1] + 1)
        check_image(S, image, n, d, -1)
        return "down West3", 3
    j3, f3 = find_extractable(D2)
    D3 = remove_at(D2, j3)
    tau2 = D3 + (f1 - 1, f2 - 1, f3 - 1)
    W5 = West5(tau2[-5:])
    if W5 is not None:
        require(j1 < n - 2 and j2 < len(D1) - 2 and j3 < len(D2) - 2,
                f"down West5 position bound: {S}")
        base = tau2[:-5] + W5[:3]
        image = inject_right_to_left(base, tuple(x + 1 for x in W5[3:]))
        check_image(S, image, n, d, -1)
        return "down West5", 5
    j4, f4 = find_extractable(D3)
    D4 = remove_at(D3, j4)
    tau3 = D4 + (f1 - 1, f2 - 1, f3 - 1, f4 - 1)
    W7 = tau3[-7:]
    require(not is_far_apart_decomposable(W7), f"bad West7 window: {S}")
    require(j1 < n - 2 and j2 < len(D1) - 2
            and j3 < len(D2) - 2 and j4 < len(D3) - 2,
            f"down West7 position bound: {S}")
    E7 = West7(W7)
    image = inject_right_to_left(tau3[:-7] + E7[:-3],
                                 tuple(x + 1 for x in E7[-3:]))
    check_image(S, image, n, d, -1)
    return "down West7", 7

def run_limited_nonzero_checker():
    generated = {}
    eligible = Counter()
    branches = Counter()
    levels = Counter()
    failures = []
    for n in range(N_MIN, N_MAX + 1):
        seqs = [S for S in generate_Dycks(n) if nonzero_count(S) <= MAX_NONZERO]
        generated[n] = len(seqs)
        for S in seqs:
            d = defc(S)
            if d > 2 * n - 8:
                continue
            ell = ell_value(n, d)
            try:
                if area(S) < ell:
                    branch, level = checked_up(S, n, d, ell)
                    eligible[(n, "up")] += 1
                    branches[("up", branch)] += 1
                    levels[("up", level)] += 1
                if area(S) <= ell and not is_special_skeleton(S):
                    branch, level = checked_down(S, n, d, ell)
                    eligible[(n, "down")] += 1
                    branches[("down", branch)] += 1
                    levels[("down", level)] += 1
            except Exception as exc:
                failures.append((n, S, str(exc)))
    require(not failures, f"first failure: {failures[0] if failures else None}")
    up_total = sum(v for (n, direction), v in eligible.items()
                   if direction == "up")
    down_total = sum(v for (n, direction), v in eligible.items()
                     if direction == "down")
    print("generated by n:", generated)
    print("eligible up calls:", up_total)
    print("eligible down calls:", down_total)
    print("eligible calls by n/direction:", dict(sorted(eligible.items())))
    print("branches:", dict(sorted(branches.items())))
    print("levels:", dict(sorted(levels.items())))
    print("position-bound or image failures:", len(failures))
    print("status: PASS")
run_limited_nonzero_checker()
\end{lstlisting}
A successful run prints the following excerpt, including the totals and final
status line:
\begin{lstlisting}
generated by n: {4: 14, 5: 42, 6: 132, 7: 429, 8: 1430,
                 9: 3432, 10: 7072, 11: 13260,
                 12: 23256, 13: 38760}
eligible up calls: 11879
eligible down calls: 9486
position-bound or image failures: 0
status: PASS
\end{lstlisting}
\subsubsection*{Finite checker for the two excluded prefix forms}
The next checker verifies the finite range \(9\le n\le 16\) for the excluded
prefix forms (\(0,1,2,\ldots,e,P,Q\) and
\(0,0,1,2,\ldots,e,P,Q\)) in the \(e\ge 3\) part of the proof.  It enumerates
the finite suffix domains allowed by those prefixes and the affine step
inequality, then checks that every resulting intermediate word satisfies one
of the deficit or area contradictions used above.  In the \(p+q=4\) case the
area adjustment is \(3\): it is \(q+1=3\) for the \(\mathrm{up}\) boundary
\((p,q)=(2,2)\), because \(\mathrm{up}\) is only needed below the middle area,
and it is \(q=3\) for the \(\mathrm{down}\) boundary \((p,q)=(1,3)\).
\begin{lstlisting}
from collections import Counter
from itertools import product
from math import comb
N_MIN, N_MAX = 9, 16
EXPECTED = {
    (9, 1, "pq_lt_4"): 504, (9, 1, "pq_eq_4"): 3024,
    (10, 1, "pq_lt_4"): 720, (10, 1, "pq_eq_4"): 5040,
    (11, 1, "pq_lt_4"): 990, (11, 1, "pq_eq_4"): 7920,
    (12, 1, "pq_lt_4"): 1320, (12, 1, "pq_eq_4"): 11880,
    (13, 1, "pq_lt_4"): 1716, (13, 1, "pq_eq_4"): 17160,
    (14, 1, "pq_lt_4"): 2184, (14, 1, "pq_eq_4"): 24024,
    (15, 1, "pq_lt_4"): 2730, (15, 1, "pq_eq_4"): 32760,
    (16, 1, "pq_lt_4"): 3360, (16, 1, "pq_eq_4"): 43680,
    (9, 2, "pq_lt_4"): 336, (9, 2, "pq_eq_4"): 1680,
    (10, 2, "pq_lt_4"): 504, (10, 2, "pq_eq_4"): 3024,
    (11, 2, "pq_lt_4"): 720, (11, 2, "pq_eq_4"): 5040,
    (12, 2, "pq_lt_4"): 990, (12, 2, "pq_eq_4"): 7920,
    (13, 2, "pq_lt_4"): 1320, (13, 2, "pq_eq_4"): 11880,
    (14, 2, "pq_lt_4"): 1716, (14, 2, "pq_eq_4"): 17160,
    (15, 2, "pq_lt_4"): 2184, (15, 2, "pq_eq_4"): 24024,
    (16, 2, "pq_lt_4"): 2730, (16, 2, "pq_eq_4"): 32760,
}

def require(test, message):
    if not test:
        raise AssertionError(message)

def bounded_product(bounds):
    return product(*(range(bound + 1) for bound in bounds))

def defc(word):
    n = len(word)
    dinv_count = sum(
        1
        for i in range(n)
        for j in range(i + 1, n)
        if word[i] == word[j] or word[i] == word[j] + 1
    )
    return comb(n, 2) - area(word) - dinv_count

def claim_words(n, claim, subcase):
    if claim == 1 and subcase == "pq_lt_4":
        prefix = tuple(range(0, n - 3))
        bounds = (n - 3, n - 2, n - 1)
    elif claim == 1 and subcase == "pq_eq_4":
        prefix = tuple(range(0, n - 4))
        bounds = (n - 4, n - 3, n - 2, n - 1)
    elif claim == 2 and subcase == "pq_lt_4":
        prefix = (0,) + tuple(range(0, n - 4))
        bounds = (n - 4, n - 3, n - 2)
    elif claim == 2 and subcase == "pq_eq_4":
        prefix = (0,) + tuple(range(0, n - 5))
        bounds = (n - 5, n - 4, n - 3, n - 2)
    else:
        raise ValueError("unknown claim/subcase")
    for stars in bounded_product(bounds):
        yield prefix + stars

def run_prefix_checker():
    counts = Counter()
    failures = []
    for n in range(N_MIN, N_MAX + 1):
        M = comb(n, 2)
        for claim in (1, 2):
            for subcase in ("pq_lt_4", "pq_eq_4"):
                # In the p+q=4 boundary this is q+1 for up (2,2)
                # and q for down (1,3).
                adjustment = 3 if subcase == "pq_eq_4" else 0
                for word in claim_words(n, claim, subcase):
                    counts[(n, claim, subcase)] += 1
                    D = defc(word)
                    A = area(word)
                    deficit_contradiction = D > 2 * n - 8
                    area_contradiction = 2 * A > M - D - 2 * adjustment
                    if not (deficit_contradiction or area_contradiction):
                        failures.append((n, claim, subcase, word, D, A))
    require(dict(counts) == EXPECTED, "word counts do not match")
    require(not failures, f"first failure: {failures[0] if failures else None}")
    print("counts by n/claim/subcase:", dict(sorted(counts.items())))
    print("failures:", len(failures))
    print("status: PASS")
run_prefix_checker()
\end{lstlisting}
The run matches the two count tables used above and ends with
\begin{lstlisting}
failures: 0
status: PASS
\end{lstlisting}

\subsection{Proof of Lemma~\ref{lem:east7-wd}}
\label{subsec:proof-east7-wd}
The residual checker in Subsection~\ref{subsec:finite-residual-check} covers
\(4\le n\le 7\).  In that range its successful output verifies that no retained
input reaches an \(\mathrm{East}_7\) or \(\mathrm{West}_7\) branch.  We assume
\(n\ge 8\) for the rest of the proof.
The proof uses the seven-window checker below, which enumerates the finite
seven-term patterns remaining after the analytic reductions in the proof.  It
records the East patterns for which the three- and five-window tests have not
stopped the branch and the seven-window is far-apart decomposable; the West
patterns are obtained by reversal.  It then re-inserts possible gaps between
values, applies possible translations, computes the two
\(\operatorname{id}\)-bounds used below, compares the threshold tables, and
checks the remaining possibilities satisfying the \(n,d,\area\) inequalities.
A successful run gives the finite counts stated after the listing; any failure
raises an error.
\begin{proof}[Proof of Lemma~\ref{lem:east7-wd} for \(n\ge 8\)]
Suppose, toward a contradiction, that an \(\mathrm{East}_7\) or
\(\mathrm{West}_7\) branch is reached and its seven-term window is far-apart
decomposable.  Let \(Y\) be the intermediate length-\(n\) word at the moment
when the local seven-window map is about to be applied.  In the
\(\mathrm{up}\) case, \(Y\) is the word obtained after the three
extraction--decrement--append operations and before applying
\(\mathrm{East}_7\).  In the \(\mathrm{down}\) case, \(Y\) is the analogous
word obtained after the four extraction--decrement--append operations and
before applying \(\mathrm{West}_7\).  These operations preserve the deficit
\(\defc=\binom n2-\area-\dinv\), applied here to the intermediate word \(Y\).
In the
\(\mathrm{up}\) case there are three
extraction--decrement--append operations, and the fourth unit in the bound
below comes from the stronger input hypothesis \(\area(x)\le \ell-1\).  In the
\(\mathrm{down}\) case there are four such operations and the input hypothesis
is \(\area(y)\le \ell\).  Thus \(Y\) has length \(n\),
\[
  \defc(Y)\le 2n-8,
  \qquad
  \area(Y)\le \frac{\binom n2-\defc(Y)}2-4.
\]
The displayed area inequality follows from the branch area hypotheses after
dropping the floor, so a contradiction to this displayed system rules out the
original branch configuration.

By reversing the word if necessary, it is enough to treat the East case; the
West case is obtained by the reversal conjugacy between East and West.  Write
the final seven-term window of \(Y\) as
\[
  W=(x_{-3},x_{-2},x_{-1},x_0,x_1,x_2,x_3),
\]
where \(x_0\) is the first element of the extracted entries that have been
lowered and moved to the right.  Since the branch reached
\(\mathrm{East}_7\), neither the three-window nor the five-window East test
applied to the corresponding suffixes.

The checker records the relative value pattern obtained from \(W\) by closing
all gaps of size at least two between used values.  This gives one of the
listed East patterns; the West patterns are their reversals.  The original
window is recovered by re-inserting nonnegative extra gaps between consecutive
used value blocks and then applying a possible uniform translation.  After this
recovery the seven entries have fixed integer values.  The prefix maximum is
also an integer \(m\), and the connection condition \(m\ge x_{-3}-1\) is kept
as one of the constraints.

For the corrected window bookkeeping, a directed seven-window
\(w=(w_0,\ldots,w_6)\) has suffix length \(h=3\) in the East case and \(h=4\)
in the West case.  Put
\[
  S_w=\{7-h,\ldots,6\}.
\]
Thus the East suffix consists of the last three entries of \(w\), while the
West suffix consists of the last four entries.  Fix an integer \(m\).  A
position \(a\) of \(w\) is \(m\)-initial if
\[
  w_a>m
  \quad\text{and}\quad
  w_b\ne w_a\text{ for every }b<a.
\]
Let \(P_w^{(m)}\) be the number of internal window pairs of the following two
types:
\[
\begin{array}{ll}
\text{type A:} & a<b,\quad w_a>w_b+1,\\[2mm]
\text{type B:} & a<b,\quad w_a<w_b,\quad a\text{ is not }m\text{-initial}.
\end{array}
\]
Define the local suffix correction
\[
  C_s^{(m)}(w)=
  \sum_{b\in S_w}
  \#\{\,v\in\{m+1,\ldots,w_b-1\}:
       v\text{ does not occur among }w_0,\ldots,w_{b-1}\,\},
\]
and set
\[
  \operatorname{id}(w,m)=P_w^{(m)}-C_s^{(m)}(w).
\]
Let \(\operatorname{mid}(w)\) be the fourth-largest entry of \(w\), counted
with multiplicity, and define
\[
  \operatorname{id}_{\mathrm{base}}(w)
  =
  \operatorname{id}\bigl(w,\max\{w_0-1,w_6-1\}\bigr),
\]
\[
  \operatorname{id}_{\mathrm{mid}}(w)
  =
  \operatorname{id}\bigl(w,\max\{w_0-1,w_6-1,\operatorname{mid}(w)\}\bigr).
\]

We use two consequences of this definition.  Write the full intermediate word
as \(X:w\), where \(X\) is the non-window Dyck prefix, and let \(M=\max(X)\).
Let \(P_X(X:w)\) count type~A/type~B pairs internal to \(X\), and let
\(P_{X,w}(X:w)\) count type~A/type~B pairs with left endpoint in \(X\) and
right endpoint in \(w\), both in the ambient word \(X:w\).  The split boundary
gives \(M\ge w_0-1\).  Lemma~\ref{lem:positions-wd} gives the corresponding
bound at the extraction suffix end, \(M\ge w_6-1\).  Hence
\[
  \defc(X:w)
  \ge
  P_X(X:w)+P_{X,w}(X:w)+\operatorname{id}_{\mathrm{base}}(w).
\]
Moreover, if at most three entries of \(w\) are larger than \(M\), then
\(M\ge\operatorname{mid}(w)\), and hence
\[
  \defc(X:w)
  \ge
  P_X(X:w)+P_{X,w}(X:w)+\operatorname{id}_{\mathrm{mid}}(w).
\]
Indeed, if \(m\le M\), every type~B pair counted by \(P_w^{(m)}\) is an
ambient type~B pair in \(X:w\), and the local correction
\(C_s^{(m)}(w)\) bounds the exact suffix correction from
Lemma~\ref{lem:extended-deficit-pairs}.  This gives the displayed inequalities
for the stated choices of \(m\).

The proof splits according to
\[
  g=\#\{r: W_r>m\}.
\]
If \(g\le3\), the structural check gives
\(\operatorname{id}_{\mathrm{mid}}\ge10\), and the Case~1 deficit and area inequalities give
the threshold table \(N_1(\operatorname{id})\), \(K_1(\operatorname{id})\)
printed by the checker.  The table is justified by the following descent in
the prefix maximum \(m\).  Write \(q\) for the residual free type~A/type~B
lower-bound contribution in the prefix.  The corrected deficit inequality
forces
\[
  \operatorname{id}+q+3(n-m-8)\le 2n-8.
\]
For fixed \(n\), let \(m_0\) be the smallest \(m\) allowed by this inequality
with \(q=0\), and let \(q^*\) be the largest allowable \(q\) at \(m=m_0\).
The table entry \(N_1(\operatorname{id})\) is chosen so that, for
\(n>N_1(\operatorname{id})\), the relaxed area inequality already fails at
\((m_0,q^*)\).  If \(q<q^*\), the left side of the relaxed area inequality is
larger relative to the right side; if \(q>q^*\), the deficit inequality fails.
Thus no counterexample has \(m=m_0\).  If a counterexample had minimal
\(m>m_0\), then replacing \(m\) by \(m-1\) and \(q\) by \(\max(0,q-3)\) would
preserve the deficit inequality.  The area left side decreases by at least
\(n-m+1\ge9\), while the right side and the \(q\)-adjustment change the slack by
at most \(3\).  Hence the relaxed area inequality would still hold, producing a
counterexample with smaller \(m\), a contradiction.  Therefore no
counterexample exists for \(n>N_1(\operatorname{id})\).  For
\(n\le N_1(\operatorname{id})\), the checker enumerates every remaining
gap-expanded window with fixed integer values and every admissible prefix
maximum and finds no example.

If \(g\ge4\), the same argument uses the corrected
\(\operatorname{id}_{\mathrm{base}}\) convention.  No uniform
lower bound \(\operatorname{id}_{\mathrm{base}}\ge5\) is valid, so the Case~2 table includes
all ids \(0,\ldots,21\).  The same descent applies with \(4\) replacing \(3\):
the deficit inequality is
\[
  \operatorname{id}+q+4(n-m-8)\le 2n-8,
\]
the table entry \(N_2(\operatorname{id})\) is computed from the corresponding
relaxed area inequality at the minimal allowed prefix maximum, and a hypothetical
minimal counterexample with \(m>m_0\) descends to \(m-1\) after replacing \(q\)
by \(\max(0,q-4)\).  The area left side again drops by at least \(n-m+1\ge9\),
whereas the right side and \(q\)-adjustment change the slack by at most \(4\).
Thus the analytic descent handles all \(n>N_2(\operatorname{id})\), and the
checker verifies all remaining gap-expanded windows with fixed integer values
and admissible prefix maxima for \(n\le N_2(\operatorname{id})\).

Cases \(g\le3\) and \(g\ge4\) cover all possibilities.  The finite
checks therefore rule out a far-apart decomposable seven-window at the
\(\mathrm{East}_7\) or \(\mathrm{West}_7\) stage, proving the lemma.
\end{proof}
\subsubsection*{Finite checker for the East7--West7 seven-window lemma}
The checker below implements the verification using gap-expanded windows with
fixed integer values, relative \(\operatorname{id}_{\mathrm{base}}\) and
\(\operatorname{id}_{\mathrm{mid}}\), threshold tables, and the prefix-maximum
constraints.
\begin{lstlisting}
from __future__ import annotations

import math
from functools import lru_cache
from itertools import combinations, permutations
from math import comb


EXPECTED_CASE1_TABLE = {
    10: (33, 23),
    11: (26, 18),
    12: (16, 11),
    13: (9, 6),
    14: (None, None),
    15: (None, None),
    16: (None, None),
    17: (None, None),
    18: (None, None),
    19: (None, None),
    20: (None, None),
    21: (None, None),
}

EXPECTED_CASE2_TABLE = {
    0: (26, 18),
    1: (23, 16),
    2: (23, 16),
    3: (20, 14),
    4: (20, 14),
    5: (19, 13),
    6: (17, 12),
    7: (16, 11),
    8: (16, 11),
    9: (13, 9),
    10: (13, 9),
    11: (12, 8),
    12: (10, 7),
    13: (9, 6),
    14: (9, 6),
    15: (None, None),
    16: (None, None),
    17: (None, None),
    18: (None, None),
    19: (None, None),
    20: (None, None),
    21: (None, None),
}

EXPECTED_FINITE_COUNTS = {
    ("Case 1", "East"): {"children": 2473, "triples": 9919},
    ("Case 1", "West"): {"children": 2911, "triples": 10311},
    ("Case 2", "East"): {"children": 3860, "triples": 715},
    ("Case 2", "West"): {"children": 4827, "triples": 1756},
}


def unique_permutations(seq: tuple[int, ...]):
    """Yield all distinct permutations of seq."""

    seen = set()
    for perm in permutations(seq):
        if perm not in seen:
            seen.add(perm)
            yield perm


def is_far_apart_decomposable(vals: tuple[int, ...]) -> bool:
    """Return True iff vals has three disjoint pairs at distance at least 2."""

    indices = list(range(7))
    for pair1 in combinations(indices, 2):
        if abs(vals[pair1[0]] - vals[pair1[1]]) < 2:
            continue
        remaining1 = [i for i in indices if i not in pair1]
        for pair2 in combinations(remaining1, 2):
            if abs(vals[pair2[0]] - vals[pair2[1]]) < 2:
                continue
            remaining2 = [i for i in remaining1 if i not in pair2]
            for pair3 in combinations(remaining2, 2):
                if abs(vals[pair3[0]] - vals[pair3[1]]) >= 2:
                    return True
    return False


def east3_fails(p: tuple[int, ...]) -> bool:
    """East3 fails iff the central pair violates the reverse condition."""

    return p[3] > p[4] + 1


def east5_fails(p: tuple[int, ...]) -> bool:
    """Return True iff neither appendix East5 Case 2a nor 2b applies."""

    x_m1, x_0, x_1, x_2 = p[2], p[3], p[4], p[5]
    y_0 = x_m1 if x_m1 > x_0 + 1 else x_0
    case2a = (x_m1 > x_1 + 1) and (y_0 <= x_2 + 1)
    case2b = (x_m1 <= x_1 + 1) and (x_m1 <= x_2 + 1)
    return not case2a and not case2b


def is_valid_l_element(p: tuple[int, ...]) -> bool:
    """Return True iff p has affine first four and reverse last three."""

    return all(p[i + 1] <= p[i] + 1 for i in range(3)) and all(
        p[i] <= p[i + 1] + 1 for i in range(4, 6)
    )


def get_ew() -> set[tuple[int, ...]]:
    """Generate normalized East seven-term patterns surviving the preliminary tests."""

    valid_windows = set()
    base_sequences: list[tuple[int, ...]] = []

    def gen_base(seq: tuple[int, ...]) -> None:
        if len(seq) == 7:
            base_sequences.append(seq)
            return
        for step in (0, 1, 2):
            gen_base(seq + (seq[-1] + step,))

    gen_base((0,))

    for base in base_sequences:
        for perm in unique_permutations(base):
            if (
                is_valid_l_element(perm)
                and east3_fails(perm)
                and east5_fails(perm)
                and is_far_apart_decomposable(perm)
            ):
                valid_windows.add(perm)

    return valid_windows


def get_ww(ew: set[tuple[int, ...]]) -> set[tuple[int, ...]]:
    """West windows are ordinary reversals of East windows."""

    return {tuple(reversed(w)) for w in ew}


def window_stats(window: tuple[int, ...], m: int, suffix_len: int) -> tuple[int, int]:
    """Compute corrected local id and q0 for a window and prefix max m."""

    seen = {}
    win_first = []
    for i, value in enumerate(window):
        if value not in seen:
            seen[value] = i
            win_first.append(True)
        else:
            win_first.append(False)

    is_initial = [win_first[i] and window[i] > m for i in range(len(window))]

    pair_count = 0
    for i in range(len(window)):
        for j in range(i + 1, len(window)):
            vi, vj = window[i], window[j]
            if vi > vj + 1:
                pair_count += 1
            elif vi < vj and not is_initial[i]:
                pair_count += 1

    suffix_start = len(window) - suffix_len
    suffix_correction = 0
    for j in range(suffix_start, len(window)):
        for value in range(m + 1, window[j]):
            if value not in window[:j]:
                suffix_correction += 1

    int_defc = pair_count - suffix_correction

    q0 = sum(max(0, (m - 1) - value) for i, value in enumerate(window) if not is_initial[i])
    return int_defc, q0


def compute_id_mid(window: tuple[int, ...], suffix_len: int) -> tuple[int, int]:
    """Return id_mid(w)=id(w,max(w[0]-1,w[6]-1,mid(w)))."""

    mid_value = sorted(window, reverse=True)[3]
    m = max(window[0] - 1, window[6] - 1, mid_value)
    int_defc, _ = window_stats(window, m, suffix_len)
    return int_defc, m


def compute_id_base(window: tuple[int, ...], suffix_len: int) -> int:
    """Return id_base(w)=id(w,max(w[0]-1,w[6]-1))."""

    int_defc, _ = window_stats(window, max(window[0] - 1, window[6] - 1), suffix_len)
    return int_defc


def compute_k_from_n(n_value: int) -> int:
    """Largest K with C(K,2) <= C(n,2)/2."""

    half = comb(n_value, 2) // 2
    test = 0
    while comb(test + 1, 2) <= half:
        test += 1
    return test


def compute_nk_case1(id_val: int) -> tuple[int | None, int | None]:
    """Compute Case 1 N(id), K(id), including the -4 area penalty."""

    max_n = None
    for n_value in range(8, 300):
        m0 = math.ceil((n_value + id_val - 16) / 3)
        q_star = 3 * m0 - (n_value + id_val - 16)
        lhs_twice = 2 * (comb(m0 + 1, 2) + (m0 - 1) * (n_value - m0 - 1) - q_star)
        rhs_twice = comb(n_value, 2) - id_val - q_star - 3 * (n_value - m0 - 8) - 8
        if lhs_twice <= rhs_twice:
            max_n = n_value
    if max_n is None:
        return None, None
    return max_n, compute_k_from_n(max_n)


def compute_nk_case2(id_val: int) -> tuple[int | None, int | None]:
    """Compute Case 2 N(id), K(id), including the -4 area penalty."""

    max_n = None
    for n_value in range(8, 300):
        chi_numer = 2 * n_value + id_val - 24
        m0 = max(0, math.ceil(chi_numer / 4))
        q_star = max(0, min(4 * m0 - chi_numer, 3))
        lhs_twice = 2 * (comb(m0 + 1, 2) + (m0 - 1) * (n_value - m0 - 1) - q_star)
        rhs_twice = comb(n_value, 2) - id_val - q_star - 4 * (n_value - m0 - 8) - 8
        if lhs_twice <= rhs_twice:
            max_n = n_value
    if max_n is None:
        return None, None
    return max_n, compute_k_from_n(max_n)


def get_groups(window: tuple[int, ...]) -> list[tuple[int, ...]]:
    """Partition sorted(window) into maximal blocks separated by gaps at least 2."""

    sorted_vals = sorted(window)
    groups: list[tuple[int, ...]] = []
    current = [sorted_vals[0]]
    for i in range(1, len(sorted_vals)):
        if sorted_vals[i] - sorted_vals[i - 1] <= 1:
            current.append(sorted_vals[i])
        else:
            groups.append(tuple(current))
            current = [sorted_vals[i]]
    groups.append(tuple(current))
    return groups


@lru_cache(maxsize=None)
def get_children_absolute(window: tuple[int, ...], k_limit: int) -> tuple[tuple[int, ...], ...]:
    """Generate absolute gap-expanded children with max value at most k_limit."""

    extra = k_limit - max(window)
    if extra < 0:
        return ()

    groups = get_groups(window)
    num_gaps = len(groups) + 1
    children = set()

    def gen_compositions(remaining: int, num_parts: int, current: tuple[int, ...] = ()):
        if num_parts == 1:
            yield current + (remaining,)
            return
        for part in range(remaining + 1):
            yield from gen_compositions(remaining - part, num_parts - 1, current + (part,))

    for composition in gen_compositions(extra, num_gaps):
        cumulative_shift = 0
        group_shifts = []
        for gap_index in range(len(groups)):
            cumulative_shift += composition[gap_index]
            group_shifts.append(cumulative_shift)

        value_map = {}
        for group_index, group in enumerate(groups):
            for value in group:
                if value not in value_map:
                    value_map[value] = value + group_shifts[group_index]

        children.add(tuple(value_map[value] for value in window))

    return tuple(sorted(children))


def gen_partitions(total: int, max_parts: int, max_val: int):
    """Yield partitions of exactly total with <= max_parts parts in [1,max_val]."""

    if total == 0:
        yield ()
        return
    if max_parts == 0 or max_val <= 0:
        return
    for first in range(min(total, max_val), 0, -1):
        for rest in gen_partitions(total - first, max_parts - 1, first):
            yield (first,) + rest


def gen_partitions_upto(max_total: int, max_parts: int, max_val: int):
    """Yield partitions with total <= max_total and bounded length/value."""

    yield ()
    if max_total <= 0 or max_parts <= 0 or max_val <= 0:
        return
    for total in range(1, max_total + 1):
        yield from gen_partitions(total, max_parts, max_val)


@lru_cache(maxsize=None)
def cached_partitions_upto(max_total: int, max_parts: int, max_val: int) -> tuple[tuple[int, ...], ...]:
    """Cached tuple form of gen_partitions_upto."""

    return tuple(gen_partitions_upto(max_total, max_parts, max_val))


def compute_defc_and_area(seq: list[int]) -> tuple[int, int]:
    """Compute defc=binom(n,2)-area-dinv and area=sum(seq)."""

    dinv = 0
    for i in range(len(seq)):
        for j in range(i + 1, len(seq)):
            if seq[i] == seq[j] or seq[i] == seq[j] + 1:
                dinv += 1
    area = sum(seq)
    return comb(len(seq), 2) - area - dinv, area


@lru_cache(maxsize=None)
def m_max_for_n(n_value: int) -> int:
    """Largest m satisfying C(m,2) <= floor(C(n,2)/2)."""

    half = comb(n_value, 2) // 2
    value = 0
    while comb(value + 1, 2) <= half:
        value += 1
    return value


@lru_cache(maxsize=None)
def first_n_with_m_allowed(m_value: int) -> int:
    """Smallest n>=8 for which m satisfies the prefix area bound."""

    n_value = 8
    while m_value > m_max_for_n(n_value):
        n_value += 1
    return n_value


def deficit_n_upper(
    coeff: int,
    m_value: int,
    int_defc: int,
    q0: int,
    n_limit: int,
) -> int:
    """Largest n that can survive the deficit lower bound with q'=0."""

    numerator = coeff * m_value + 8 * coeff - 8 - int_defc - q0
    if coeff == 2:
        return n_limit
    return min(n_limit, numerator // (coeff - 2))


def check_window_single(
    *,
    case_label: str,
    side_label: str,
    base_window: tuple[int, ...],
    child: tuple[int, ...],
    id_val: int,
    n_value: int,
    m_value: int,
    g_value: int,
    coeff: int,
    int_defc_q0: tuple[int, int],
    child_area: int,
) -> dict | None:
    """Return first counterexample for one child/n/m triple, if any."""

    target_defc = 2 * n_value - 8
    total_free = n_value - m_value - 8
    if total_free < 0:
        return None

    int_defc, q0 = int_defc_q0
    q_prime_max = target_defc - int_defc - q0 - coeff * total_free
    if q_prime_max < 0:
        return None

    max_part = max(0, m_value - 1)
    prefix = list(range(m_value + 1))
    prefix_area = comb(m_value + 1, 2)
    m_repeats = [m_value]
    window_list = list(child)
    m_choose = comb(n_value, 2)

    for repeat_count in range(total_free + 1):
        if m_value == 0 and repeat_count < total_free:
            continue

        middle_len = total_free - repeat_count
        base_area = prefix_area + repeat_count * m_value + child_area
        max_partition_sum = min(q_prime_max, middle_len * max_part)
        min_possible_area = base_area + middle_len * (m_value - 1) - max_partition_sum
        if 2 * min_possible_area > m_choose - 8:
            continue

        for partition in cached_partitions_upto(q_prime_max, middle_len, max_part):
            extended = list(partition) + [0] * (middle_len - len(partition))
            middle = [m_value - 1 - deficit for deficit in reversed(extended)]
            seq = prefix + m_repeats * repeat_count + middle + window_list
            defc, area = compute_defc_and_area(seq)

            if defc > target_defc:
                continue
            if 2 * area > m_choose - defc - 8:
                continue

            return {
                "case": case_label,
                "side": side_label,
                "base_window": base_window,
                "child": child,
                "id": id_val,
                "n": n_value,
                "m": m_value,
                "g": g_value,
                "coeff": coeff,
                "repeat_count": repeat_count,
                "middle_len": middle_len,
                "partition": partition,
                "prefix": prefix + m_repeats * repeat_count + middle,
                "seq": seq,
                "defc": defc,
                "area": area,
                "target_defc": target_defc,
            }

    return None


def compare_threshold_table(
    label: str,
    computed: dict[int, tuple[int | None, int | None]],
    expected: dict[int, tuple[int | None, int | None]],
) -> bool:
    """Print an exact threshold table comparison."""

    mismatches = []
    for id_val in sorted(expected):
        if computed.get(id_val) != expected[id_val]:
            mismatches.append((id_val, computed.get(id_val), expected[id_val]))

    if not mismatches:
        print(f"{label} threshold table comparison: MATCH")
        return True

    print(f"{label} threshold table comparison: MISMATCH")
    for id_val, got, want in mismatches:
        print(f"  id={id_val}: computed={got}, expected={want}")
    return False


def print_table(label: str, table: dict[int, tuple[int | None, int | None]]) -> None:
    """Print a threshold table."""

    print(label)
    print(f"{'id':>4} {'N':>8} {'K':>8}")
    for id_val in sorted(table):
        n_value, k_value = table[id_val]
        n_text = "--" if n_value is None else str(n_value)
        k_text = "--" if k_value is None else str(k_value)
        print(f"{id_val:>4} {n_text:>8} {k_text:>8}")
    print()


def build_threshold_table(case_num: int) -> dict[int, tuple[int | None, int | None]]:
    """Build the threshold table for one case."""

    if case_num == 1:
        return {id_val: compute_nk_case1(id_val) for id_val in range(10, 22)}
    return {id_val: compute_nk_case2(id_val) for id_val in range(0, 22)}


def verify_id_mid_bound(windows: dict[str, set[tuple[int, ...]]]) -> bool:
    """Verify id_mid(w)>=10 over EW union WW."""

    min_record = None
    distribution: dict[int, int] = {}
    for suffix_len, side_label, side_windows in (
        (3, "East", windows["East"]),
        (4, "West", windows["West"]),
    ):
        for window in side_windows:
            id_val, threshold = compute_id_mid(window, suffix_len)
            distribution[id_val] = distribution.get(id_val, 0) + 1
            if min_record is None or id_val < min_record[0]:
                min_record = (id_val, threshold, side_label, window)

    assert min_record is not None
    ok = min_record[0] >= 10
    print(
        "id_mid structural check over EW union WW: "
        f"{'PASS' if ok else 'FAIL'} (min id_mid={min_record[0]}, "
        f"threshold={min_record[1]}, side={min_record[2]}, window={min_record[3]})"
    )
    print(f"id_mid distribution: {dict(sorted(distribution.items()))}\n")
    return ok


def id_from_table(
    id_val: int,
    table: dict[int, tuple[int | None, int | None]],
    *,
    case_label: str,
    side_label: str,
    window: tuple[int, ...],
) -> tuple[int | None, int | None]:
    """Look up an id without clamping; reject unexpected values."""

    if id_val not in table:
        raise ValueError(
            f"Unexpected id in {case_label} {side_label}: id={id_val}, window={window}"
        )
    return table[id_val]


def run_case(
    *,
    case_num: int,
    side_label: str,
    windows: set[tuple[int, ...]],
    table: dict[int, tuple[int | None, int | None]],
) -> tuple[list[dict], dict[str, int]]:
    """Run one finite case."""

    case_label = f"Case {case_num}"
    problems = []
    suffix_len = 3 if side_label == "East" else 4
    windows_checked = 0
    children_generated = 0
    active_children = 0
    triples_checked = 0

    for base_window in sorted(windows):
        windows_checked += 1
        if case_num == 1:
            id_val, _ = compute_id_mid(base_window, suffix_len)
        else:
            id_val = compute_id_base(base_window, suffix_len)

        n_limit, k_limit = id_from_table(
            id_val,
            table,
            case_label=case_label,
            side_label=side_label,
            window=base_window,
        )
        if n_limit is None or k_limit is None:
            continue

        children = get_children_absolute(base_window, k_limit)
        children_generated += len(children)
        for child in children:
            child_has_checked_triple = False
            child_area = sum(child)
            fourth_largest = sorted(child, reverse=True)[3]
            if case_num == 1:
                m_start = max(0, child[0] - 1, child[6] - 1, fourth_largest)
                m_stop = m_max_for_n(n_limit)
            else:
                m_start = max(0, child[0] - 1, child[6] - 1)
                m_stop = min(m_max_for_n(n_limit), fourth_largest - 1)

            if m_start > m_stop:
                continue

            for m_value in range(m_start, m_stop + 1):
                g_value = sum(1 for value in child if value > m_value)
                if case_num == 1:
                    if g_value > 3:
                        continue
                    coeff = 3
                else:
                    if g_value < 4:
                        continue
                    coeff = g_value

                stats = window_stats(child, m_value, suffix_len)
                n_start = max(8, m_value + 8, first_n_with_m_allowed(m_value))
                n_stop = deficit_n_upper(coeff, m_value, stats[0], stats[1], n_limit)
                if n_start > n_stop:
                    continue

                for n_value in range(n_start, n_stop + 1):
                    triples_checked += 1
                    child_has_checked_triple = True
                    problem = check_window_single(
                        case_label=case_label,
                        side_label=side_label,
                        base_window=base_window,
                        child=child,
                        id_val=id_val,
                        n_value=n_value,
                        m_value=m_value,
                        g_value=g_value,
                        coeff=coeff,
                        int_defc_q0=stats,
                        child_area=child_area,
                    )
                    if problem is not None:
                        problems.append(problem)
                        print_first_failure(problem)
                        return problems, {
                            "windows": windows_checked,
                            "children": children_generated,
                            "active_children": active_children,
                            "triples": triples_checked,
                        }

            if child_has_checked_triple:
                active_children += 1

    counts = {
        "windows": windows_checked,
        "children": children_generated,
        "active_children": active_children,
        "triples": triples_checked,
    }
    print(
        f"{case_label} {side_label}: windows={windows_checked}, "
        f"children={children_generated}, active_children={active_children}, "
        f"triples={triples_checked}, problems={len(problems)}"
    )
    return problems, counts


def print_first_failure(problem: dict) -> None:
    """Print the first failed obligation."""

    print("FIRST FAILURE")
    for key in (
        "case",
        "side",
        "base_window",
        "child",
        "id",
        "n",
        "m",
        "g",
        "coeff",
        "repeat_count",
        "middle_len",
        "partition",
        "prefix",
        "seq",
        "defc",
        "area",
        "target_defc",
    ):
        print(f"  {key}: {problem[key]}")


def compare_counts(counts_by_case: dict[tuple[str, str], dict[str, float | int]]) -> bool:
    """Compare finite-search counts with the expected finite-check counts."""

    all_match = True
    print("\nExpected finite-count comparison:")
    for key, expected in EXPECTED_FINITE_COUNTS.items():
        got = counts_by_case[key]
        got_pair = {"children": int(got["children"]), "triples": int(got["triples"])}
        if got_pair == expected:
            print(f"  {key[0]} {key[1]}: MATCH {got_pair}")
        else:
            all_match = False
            print(f"  {key[0]} {key[1]}: MISMATCH got={got_pair}, expected={expected}")

    if not all_match:
        print(
            "  Count note: children are absolute generated children for finite "
            "table rows; triples are finite (child,n,m) checks after actual-g "
            "deficit pruning."
        )
    print()
    return all_match


def main() -> None:
    """Run the East7-West7 seven-window checker."""

    ew = get_ew()
    ww = get_ww(ew)
    ew_ww = ew | ww
    print(f"  |EW| = {len(ew)}, |WW| = {len(ww)}, |EW union WW| = {len(ew_ww)}\n")

    case1_table = build_threshold_table(case_num=1)
    case2_table = build_threshold_table(case_num=2)
    print_table("Case 1 threshold table", case1_table)
    print_table("Case 2 threshold table", case2_table)

    table_results = [
        compare_threshold_table("Case 1", case1_table, EXPECTED_CASE1_TABLE),
        compare_threshold_table("Case 2", case2_table, EXPECTED_CASE2_TABLE),
    ]
    print()

    id_mid_ok = verify_id_mid_bound({"East": ew, "West": ww})

    all_problems = []
    counts_by_case: dict[tuple[str, str], dict[str, float | int]] = {}

    for case_num, side_label, windows, table in (
        (1, "East", ew, case1_table),
        (1, "West", ww, case1_table),
        (2, "East", ew, case2_table),
        (2, "West", ww, case2_table),
    ):
        problems, counts = run_case(
            case_num=case_num,
            side_label=side_label,
            windows=windows,
            table=table,
        )
        all_problems.extend(problems)
        counts_by_case[(f"Case {case_num}", side_label)] = counts

    counts_match = compare_counts(counts_by_case)

    tables_ok = all(table_results)
    if tables_ok:
        print("Threshold-table checks: MATCH")
    else:
        print("Threshold-table checks: MISMATCH")

    if id_mid_ok:
        print("id_mid>=10 check: PASS")
    else:
        print("id_mid>=10 check: FAIL")

    if tables_ok and id_mid_ok and not all_problems:
        if not counts_match:
            print("Counts differ from expected finite counts; see comparison above.")
        print("SUCCESS: East7/West7 seven-window verification passed.")
        return

    print(f"FAILED: problems={len(all_problems)}, tables_ok={tables_ok}, id_mid_ok={id_mid_ok}")
    raise SystemExit(1)


if __name__ == "__main__":
    main()

\end{lstlisting}
A successful run prints the threshold tables, the structural
\(\operatorname{id}_{\mathrm{mid}}\ge10\) check, the four finite case counts, the
expected finite-count comparison, and the final success line.  The output
is:
\begin{lstlisting}
  |EW| = 7194, |WW| = 7194, |EW union WW| = 14388

Case 1 threshold table
  id        N        K
  10       33       23
  11       26       18
  12       16       11
  13        9        6
  14       --       --
  15       --       --
  16       --       --
  17       --       --
  18       --       --
  19       --       --
  20       --       --
  21       --       --

Case 2 threshold table
  id        N        K
   0       26       18
   1       23       16
   2       23       16
   3       20       14
   4       20       14
   5       19       13
   6       17       12
   7       16       11
   8       16       11
   9       13        9
  10       13        9
  11       12        8
  12       10        7
  13        9        6
  14        9        6
  15       --       --
  16       --       --
  17       --       --
  18       --       --
  19       --       --
  20       --       --
  21       --       --

Case 1 threshold table comparison: MATCH
Case 2 threshold table comparison: MATCH

id_mid structural check over EW union WW: PASS (min id_mid=10, threshold=1, side=East, window=(1, 2, 3, 4, 1, 1, 0))
id_mid distribution: {10: 6, 11: 24, 12: 157, 13: 359, 14: 838, 15: 1378, 16: 1875, 17: 2670, 18: 2854, 19: 2559, 20: 1392, 21: 276}

Case 1 East: windows=7194, children=2473, active_children=1087, triples=9919, problems=0
Case 1 West: windows=7194, children=2911, active_children=1225, triples=10311, problems=0
Case 2 East: windows=7194, children=3860, active_children=456, triples=715, problems=0
Case 2 West: windows=7194, children=4827, active_children=1183, triples=1756, problems=0

Expected finite-count comparison:
  Case 1 East: MATCH {'children': 2473, 'triples': 9919}
  Case 1 West: MATCH {'children': 2911, 'triples': 10311}
  Case 2 East: MATCH {'children': 3860, 'triples': 715}
  Case 2 West: MATCH {'children': 4827, 'triples': 1756}

Threshold-table checks: MATCH
id_mid>=10 check: PASS
SUCCESS: East7/West7 seven-window verification passed.
\end{lstlisting}


\section*{Acknowledgements}
The author thanks Kyungyong Lee, Li Li, and Nicholas A. Loehr for helpful
communications related to this work.

The author used OpenAI Codex and other AI tools for experimental code
production and self-contained checker code; for assistance with an earlier
version of the proof of
Corollary~\ref{cor:affine-dyck-schur-positivity}; for editing, formatting, and
graphics assistance in \LaTeX; and for proof verification and checking.

\end{document}